\documentclass[11pt, leqno]{amsart}

\usepackage{etex}
\usepackage{float}
\usepackage{graphics}
\usepackage{amssymb,amsmath,amsthm,amscd}
\usepackage{mathrsfs}
\usepackage{enumerate}
\usepackage{color}
\usepackage[all]{xy}
\usepackage{pictexwd}

\newtheorem{theorem}{Theorem}[section]
\newtheorem{prop}[theorem]{Proposition}
\newtheorem{corollary}[theorem]{Corollary}
\newtheorem{lemma}[theorem]{Lemma}

\theoremstyle{definition}
\newtheorem{definition}[theorem]{Definition}
\newtheorem{example}[theorem]{Example}
\newtheorem{remark}[theorem]{Remark}

\numberwithin{equation}{section}

\numberwithin{equation}{section}

\newcommand{\C}{{\mathbb C}}

\newcommand{\Z}{{\mathbb Z}}
\newcommand{\B}{{\mathbf{B}}}

\newcommand{\V}{{\mathbf V}}
\newcommand{\W}{{\mathbf W}}

\newcommand{\Prop}{\begin{prop}}
\newcommand{\enprop}{\end{prop}}
\newcommand{\Lemma}{\begin{lemma}}
\newcommand{\enlemma}{\end{lemma}}
\newcommand{\Th}{\begin{theorem}}
\newcommand{\enth}{\end{theorem}}

\newcommand{\Cor}{\begin{corollary}}
\newcommand{\encor}{\end{corollary}}

\newcommand{\Def}{\begin{definition}}
\newcommand{\edf}{\end{definition}}

\newcommand{\g}{{\mathfrak{g}}}

\newcommand{\qn}{{\mathfrak{q}(n)}}

\newcommand{\Hom}{\operatorname{Hom}}
\newcommand{\End}{\operatorname{End}}

\newcommand{\tensor}{\otimes}

\newcommand{\eq}{\begin{eqnarray}}
\newcommand{\eneq}{\end{eqnarray}}

\newcommand{\eqn}{\begin{eqnarray*}}
\newcommand{\eneqn}{\end{eqnarray*}}
\newcommand{\op}{{\operatorname{op}}}
\newcommand{\on}{\operatorname}

\newcommand{\QED}{\end{proof}}
\newcommand{\Proof}{\begin{proof}}

\newcommand{\ba}{\begin{array}}
\newcommand{\ea}{\end{array}}

\newcommand{\bi}{\begin{enumerate}[{\rm(i)}]}

\newcommand{\eqsub}{\begin{subequations}\begin{eqnarray}}
\newcommand{\eneqsub}{\end{eqnarray}\end{subequations}}

\newcommand{\ol}{\overline}

\newcommand{\nc}{\newcommand}
\nc{\la}{\lambda}
\nc{\lam}{\lambda}
\nc{\U}[1][\g]{U_q(#1)}
\nc{\te}{\tilde{e}}
\nc{\tei}{\tilde{e}_i}
\nc{\tf}{\tilde{f}}
\nc{\tfi}{\tilde{f}_i}
\nc{\tU}{\widetilde U_q(\g)}
\nc{\tE}{\tilde{E}}
\nc{\tF}{\tilde{F}}

\nc{\tk}{\tilde{k}}
\nc{\tkone}{\tk_{\ol{1}}}
\nc{\teone}{\tilde{e}_{\ol{1}}}
\nc{\tfone}{\tilde{f}_{\ol{1}}}

\nc{\teibar}{\tilde{e}_{\ol{i}}} \nc{\tfibar}{\tilde{f}_{\ol{i}}}
\nc{\tki}{{\tk}_{\ol {i}}}
\newcommand{\td}{\widetilde{d}}

\nc{\BZ}{{\mathbb{Z}}}
\nc{\al}{\alpha}
\nc{\qs}{{q}}
\nc{\lan}{\langle}
\nc{\ran}{\rangle}
\nc{\re}{{\mathrm{re}}}
\nc{\wt}{\operatorname{wt}}
\nc{\ch}{\operatorname{ch}}
\nc{\Uf}[1][\g]{U^-_q(#1)}
\nc{\Ue}{U^+_q(\g)}
\nc{\eps}{\varepsilon}
\nc{\vphi}{\varphi}
\nc{\sphi}{\varphi^*}
\nc{\seps}{\varepsilon^*}

\nc{\nn}{\nonumber}

\nc{\vp}{\varpi}
\nc{\cls}{{\operatorname{cl}}}
\nc{\Wt}{{\operatorname{Wt}}}
\nc{\Us}{U'_q(\g)}
\nc{\La}{\Lambda}
\nc{\ro}{{\rm(}}
\nc{\rf}{{\rm)}}
\nc{\norm}{{\mathrm{norm}}}
\nc{\qbox}{\quad\mbox}
\nc{\braid}{{\mathfrak{B}}}
\nc{\Ad}{\operatorname{Ad}}
\nc{\Aut}{\operatorname{Aut}}
\nc{\dt}[1]{\tilde{\tilde #1}}
\nc{\Sn}{S^{{\mathrm{norm}}}}
\nc{\aff}{{\mathrm{aff}}}
\nc{\rk}{{\mathrm{rk}}}
\nc{\tQ}{\widetilde{Q}}
\nc{\tP}{\widetilde{P}}
\nc{\tW}{\widetilde{W}}
\nc{\Dyn}{\mathrm{Dyn}}
\nc{\tD}{\widetilde{\Delta}}
\nc{\height}{{\operatorname{ht}}}
\nc{\bl}{\bigl}
\nc{\br}{\bigr}
\nc{\Hecke}{\mathrm{H}}
\nc{\HA}{\Hecke^{\mathrm{A}}}
\nc{\HB}{\Hecke^{\mathrm{B}}}
\nc{\K}{\mathrm{K}}
\newcommand{\scbul}{{\,\raise1pt\hbox{$\scriptscriptstyle\bullet$}\,}}
\nc{\vac}{{\phi}}
\nc{\Bt}{\B_\theta(\g)}
\nc{\be}{\begin{enumerate}}
\nc{\ee}{\end{enumerate}}
\nc{\low}{{\mathrm{low}}}
\nc{\upper}{{\mathrm{up}}}
\nc{\Zodd}{\Z_{\mathrm{odd}}}
\nc{\Ft}[1][n]{\mathbb{P}\mathrm{ol}_{#1}}
\nc{\Ftf}[1][n]{\widetilde{\mathbb{P}\mathrm{ol}}_{#1}}
\nc{\KA}{\on{K}^{\mathrm{A}}}
\nc{\KB}{\on{K}^{\mathrm{B}}}
\nc{\Res}{\on{Res}}
\nc{\Fc}[1][{n,m}]{\mathbf{F}_{#1}}
\nc{\tphi}{\tilde{\varphi}}
\nc{\CO}{\mathscr{O}}
\nc{\inte}{\mathrm{int}}
\nc{\Oint}{\mathcal{O}^{\ge0}_{\inte}}
\nc{\vs}{\vspace}
\nc{\tL}{\widetilde{L}}
\nc{\noi}{\noindent}

\nc{\DD}{\overrightarrow{D}}
\nc{\BB}{\overrightarrow{B}_{r,s}}
\nc{\flip}{\operatorname{flip}}
\nc{\mix}{\V^{\tensor r} \tensor \W^{\tensor s}}
\nc{\ii}{{\mathbf i}}
\nc{\jj}{{\mathbf j}}
\nc{\iiL}{{\mathbf i}^L}
\nc{\iiR}{{\mathbf i}^R}
\nc{\jjL}{{\mathbf j}^L}
\nc{\jjR}{{\mathbf j}^R}
\nc{\idj}{_{\ii} d _{\jj}}
\nc{\wtd}{\wt(_{\ii} d _{\jj})}
\nc{\Vr}{\V^{\tensor r}}
\nc{\Vrs}{\V^{\tensor (r+s)}}
\nc{\Ws}{\W^{\tensor s}}
\nc{\Flip}{\operatorname{Flip}}
\nc{\Endc}{\End_{\C}}
\nc{\GL}{{\rm GL}_n(\C)}
\nc{\Vk}{\V^{\tensor k}}
\nc{\kk}{\mathbf{k}}
\nc\gln{\mathfrak{gl}(n|n)}

\newenvironment{rouge}
{\color{red}}
{}

\nc{\bred}{\begin{rouge}}
\nc{\ered}{\end{rouge}}

\newlength{\mylength}
\title[Mixed Schur-Weyl-Sergeev duality]
{Mixed Schur-Weyl-Sergeev duality \\ for queer Lie superalgebras}

\author[J. H. Jung, S.-J. Kang]{Ji Hye Jung$^{1}$, Seok-Jin Kang$^{2}$}

\address{Department of Mathematical Sciences
         and
         Research Institute of Mathematics \\
         Seoul National University \\ Seoul 151-747, Korea}

         \email{jhjung@math.snu.ac.kr}

\address{Department of Mathematical Sciences
         and
         Research Institute of Mathematics \\
         Seoul National University \\ Seoul 151-747, Korea}

         \email{sjkang@snu.ac.kr}

\thanks{$^{1}$This work was supported by NRF Grant \# 2011-0027952
and by the second stage of the BK 21 Project,
The Development Project of Human Resources in Mathematics, KAIST in 2011.}
\thanks{$^{2}$This work was supported by NRF Grant \#
2012-005700 and NRF Grant \# 2011-0027952.}


\begin{document}

\begin{abstract}
We introduce a new family of superalgebras $\BB$ for $r, s \ge 0$
such that $r+s>0$,
which we call the walled Brauer superalgebras, and prove the mixed
Scur-Weyl-Sergeev duality for queer Lie superalgebras. More
precisely, let $\mathfrak {q}(n)$ be the queer Lie superalgebra, $\V
=\mathbb{C}^{n|n}$ the natural representation of $\mathfrak{q}(n)$
and $\W$ the dual of $\V$. We prove that, if $n \ge r+s$, the
superalgebra $\BB$ is isomorphic to the supercentralizer algebra
$\End_{\qn}(\mix)^{\op}$ of the $\mathfrak{q}(n)$-action on the
mixed tensor space $\V^{\otimes r} \otimes \W^{\otimes s}$. As an
ingredient for the proof of our main result, we construct a new
diagrammatic realization $\DD_{k}$ of the Sergeev superalgebra
$Ser_{k}$. Finally, we give a presentation of $\BB$ in terms of
generators and relations.
\end{abstract}

\maketitle

\section*{Introduction}
The general linear group ${\rm GL}_n(\C)$ of $n \times n$ invertible
complex matrices acts on $V=\C^n$ by matrix multiplication. It is
easy to see that the diagonal action of ${\rm GL}_n(\C)$ is
well-defined on the $k$-fold tensor space $V^{\tensor k}$. On the
other hand, the symmetric group $\Sigma_k$ acts on $V^{\otimes k}$
by place permutation. Clearly, these actions commute with each
other. Moreover, Schur proved that, if $n \ge k$, then these two
group actions generate the full centralizer of each other \cite{S1,
S2}. This celebrated result, often referred to as the {\it
Schur-Weyl duality}, connects the representation theories of ${\rm
GL}_n(\C)$ and $\Sigma_k$ in a fundamental way.

There are several generalizations of the Schur-Weyl duality. For
example, consider the \emph{mixed tensor space} $V^{\tensor r}
\tensor W^{\tensor s}$, where $V=\C^n$ is the natural representation
of $\GL$ and $W=V^*$ is its dual. The vector spaces $W^{\tensor s}$
and  $V^{\tensor r} \tensor W^{\tensor s}$ inherit $\GL$-module
structures in a natural way. To describe the centralizer algebra
$\End_{\GL}(V^{\tensor r} \tensor W^{\tensor s})$, Koike and Turaev
independently introduced the notion of \emph{walled Brauer algebra}
$B_{r,s}(n)$ \cite{Ko, T}. (See also \cite{BCHLLS}.) As is the case
with the pair ${\rm GL}_n(\C)$ and $\Sigma_k$, the actions of $\GL$
and $B_{r,s}(n)$ on the mixed tensor space $V^{\otimes r} \otimes
W^{s}$ generate the full centralizers of each other whenever $n \ge
r+s$. It is called the \emph{mixed Schur-Weyl duality}.

On the other hand, in \cite{Ser}, Sergeev considered the \emph{queer
Lie superalgebra} $\qn$ and its natural representation
$\V=\C^{n|n}$. To describe the supercentralizer algebra
$\End_{\qn}(\V^{\tensor k})$, Sergeev introduced a super-extension
of the symmetric group $\Sigma_k$, denoted by $Ser_{k}$, and showed
that $Ser_k$ is isomorphic to $\End_{\qn}(\V^{\tensor k})$ whenever
$n \ge k$. Moreover, he obtained a decomposition of $V^{\tensor k}$
into a direct sum of irreducible $(\qn, Ser_k)$-bimodules. The
superalgebra $Ser_k$ is called the {\it Sergeev superalgebra} and
the duality thus obtained is called the {\it Schur-Weyl-Sergeev
duality}.  (See \cite[Section 13]{Kl} for more details.)

In this paper, we prove the mixed Schur-Weyl-Sergeev duality for the
action of $\qn$ on the mixed tensor space $\V^{\tensor r} \tensor
\W^{\tensor s}$, where $\V=\C^{n|n}$ is the natural representation
of $\qn$ and $\W=\V^*$ is the dual of $\V$. To prove our main
result, we first construct a new diagrammatic realization of
$Ser_{k}$, denoted by $\DD_k$, which has a basis consisting of
$k$-superdiagrams. A {\it $k$-superdiagram} is a diagram with
$k$-vertices on its top and bottom row and the $k$ edges connecting
vertices such that each vertex on the top row is connected to
exactly one vertex in the bottom row and each edge may (or may not)
be marked. The multiplication on $\DD_{k}$ is defined by {\it marked
concatenation} of diagrams with sign.   We show that the
superalgebra $\DD_{k}$ is isomorphic to $Ser_{k}$, the
supercentralizer algebra $\End_{\qn}(\V^{\tensor k})$ (Theorem
\ref{th:DD_{k}}).

Next, we move on to define the {\it walled Brauer superalgebra}
$\BB$ with a basis consisting of $(r,s)$-superdiagrams. An {\it
$(r,s)$-superdiagram} is a diagram with $(r+s)$ vertices on the top
and bottom row, the edges connecting vertices, and a vertical wall
separating the $r$-th and $(r+1)$-th vertices in each row. Each
vertex must be connected to exactly one other vertex, and each edge
may (or may not) be marked. In addition, each vertical edge cannot
cross the wall and each horizontal edge should cross the wall. There
is a natural action of each $(r,s)$-superdiagram on the mixed tensor
space $\mix$.

As a superspace, we show that $\BB$ is isomorphic to the
supercentralizer algebra $\End_{\qn}(\mix)$ whenever $n \ge r+s$. We
give an explicit description of this linear isomorphism in Section
3.  We define a multiplication on $\BB$ by marked concatenation of
diagrams  with sign such that the product is zero whenever we get a
loop in the middle row. Thus $\BB$ becomes a superalgebra. It is
easy to see that the even part of $\BB$ contains the walled Brauer
algebra $B_{r,s}(0)$ as a subalgebra. Our main result shows that the
walled Brauer superalgebra $\BB$ is isomorphic to the
supercentralizer algebra $\End_{\qn}(\mix)^{\op}$ whenever $n \ge
r+s$ (Theorem \ref{th:mixed Schur-Weyl-Sergeev duality}). We also
give a presentation of $\BB$ in terms of generators and relations
(Theorem \ref{th:presentation of B_r,s}).

This paper is organized as follows. In Section 1, we briefly recall
the basic properties of Sergeev superalgebras. In Section 2, we
construct a diagrammatic realization of Sergeev superalgebras. In
Section 3, we define the superspace $\BB$ with a basis consisting of
$(r,s)$-superdiagrams and describe the natural action of $\BB$ on
$\mix$. Moreover, we show that there is a linear isomorphism between
$\BB$ and $\End_{\qn}(\mix)$ whenever $n \ge r+s$. In Section 4, we
define a multiplication on $\BB$ and prove our main result: the
superalgebra $\BB$ is isomorphic to the supercentralizer algebra
$\End_{\qn}(\mix)^{\op}$ whenever $n \ge r+s$. Finally, in Section
5, we give a set of generators and defining relations for the walled
Brauer superalgebra $\BB$.

\vskip3mm
\medskip \noindent
{\it Acknowledgements.} The authors would like to thank Professor
Alberto Elduque at University of Zaragoza for valuable discussions
on walled Brauer algebras.

\vskip 1cm

\section{The Sergeev superalgebras}

In this paper, we will follow the notations in \cite[Section 12]{Kl}
for the super-objects: superspaces, tensor product of superspaces,
superalgebras, supermodules, superalgebra homomorphisms, supermodule
homomorphisms, etc. We will also consider non-associative
superalgebras.
%add one thing: we will distinguish associative
%superalgebras and non-associative superalgebras.

The ground field  in this paper will be $\C$, the field of complex
numbers. We denote by $\Z_{\geq 0}$ the set of nonnegative integers
and set $\Z_2 = \Z / 2\Z$. For a homogeneous element $v$ in a
superspace $V=V_{\overline 0} \oplus V_{\overline 1}$, we write $|v|
\in \Z_{2}$ for its degree.

Let $\Sigma_k$ be the symmetric group of $k$ letters which is
generated by the transpositions $s_1, \ldots, s_{k-1}$.

\begin{definition}
  The \emph{Sergeev superalgebra} $Ser_k$ is the associative superalgebra
  generated by $s_1, \ldots, s_{k-1}$ and $c_1, \ldots, c_k$ with the following defining relations (for admissible $i,j$):

\begin{align}
    &s_i^2=1, \ s_is_{i+1}s_i=s_{i+1}s_is_{i+1},\ s_is_j=s_js_i \    &&\text{($|i-j|>1$)}, \label{eq:symmetric group}\\
  &c_i^2=-1, \ c_ic_j=-c_jc_i   \  &&\text{($i \neq j$)}, \label{eq:Clifford}\\
 &s_ic_is_i=c_{i+1}, \ s_ic_j=c_js_i   \  &&\text{($j \neq i, i+1$)}.
  \end{align}
  \end{definition}

The generators $s_1, \ldots, s_{k-1}$ are regarded as {\em even} and
$c_1, \ldots, c_k$ are {\em odd}. The subalgebra generated by $s_1,
\ldots, s_{k-1}$
%with the defining relations \eqref{eq:symmetric
%group}
is isomorphic to the group algebra $\C \Sigma_k$ of $\Sigma_{k}$ and
the subalgebra generated by  $c_1, \ldots, c_k$ is isomorphic to the
{\it Clifford superalgebra} $Cl_{k}$. Note that $Ser_{k}$ is
isomorphic to the superalegbra $\C \Sigma_k \ltimes Cl_k$, which is
$\C \Sigma_k \tensor Cl_k$ as a superspace with the multiplication
given by
$$(\sigma \tensor c_{i_1} \cdots c_{i_{\ell}})(\tau \tensor c_{j_1}\cdots c_{j_m})
=\sigma\tau \tensor c_{\tau^{-1}(i_1)}\cdots c_{\tau^{-1}(i_{\ell})}
c_{j_1}\cdots c_{j_m}$$ for $1 \le i_{s}, j_{t} \le k$.

Let $\V= \C^{n|n}$ be the superspace with $\V_{\overline 0}=\C^n$
and $\V_{\overline 1}=\C^n$. Choose a basis $\B_{\overline 0} =\{
v_{1}, \ldots, v_{n}\}$ (resp. $\B_{\overline 1} = \{v_{\overline
1}, \ldots, v_{\overline n} \}$) of $\V_{\overline 0}$ (resp.
$\V_{\overline 1}$). We write $|i|:= |v_{i}| \in \Z_{2}$ for $i =1,
\ldots, n, {\overline 1}, \ldots, {\overline n}$. Define an odd
operator $P: \V \rightarrow \V$ by
\begin{align} \label{map:odd involution}
v_i \longmapsto -v_{\ol i}, \ \ v_{\ol i} \longmapsto v_{i} \ \
(i=1, \ldots, n).
\end{align}
If we abuse the notation $\overline{\overline i} =i$ for $i=1,
\ldots, n$, \eqref{map:odd involution} can be written as
$$P(v_i)=(-1)^{|i|+1}v_{\ol i} \quad (i=1, \ldots, n, \ol 1,
\ldots, \ol n).$$

Recall that the superbracket on $\End_{\C}(\V)$ is given by
$$[f, g] = fg - (-1)^{|f||g|} gf$$
for homogeneous elements $f, g \in \End_{\C}(\V) $.
Define
$$\qn:= \{ f \in \End_{\C} (\V) \ | \ [f,P]=0 \}. $$
Note that it is closed under the superbracket.
We call $\qn$ the {\it queer Lie superalgebra}.
 With respect to the basis
$\B = \B_{\overline 0} \sqcup \B_{\overline 1}$, we have $$P= \left(
\begin{matrix} 0 & I \\ -I & 0\end{matrix} \right)$$ and the queer
Lie superalgebra $\qn$ can be expressed in the matrix form
\begin{align*}
\qn=\left\{ \left(
                          \begin{array}{cc}
                            A & B \\
                            B & A \\
                          \end{array}
                        \right)  \Big{|} \ A,B \text{ are arbitrary $n \times n$ complex matrices} \right\}.
\end{align*}

There is a natural action of $\qn$ on $\V$ by matrix multiplication,
which extends to an action on the $k$-fold tensor product
$\V^{\otimes k}$; i.e.,
\begin{align*}
  & g\cdot(w_{1} \tensor \cdots \tensor w_{k})\\
  &=\sum_{j=1}^{k} (-1)^{(|w_1|+\cdots+|w_{j-1}|)|g|}
  w_{1}\tensor \cdots \tensor w_{j-1}  \tensor gw_{j}  \tensor w_{j+1} \tensor \cdots\tensor w_{k},
\end{align*}
where the elements $g \in \qn$ and $w_{j} \in \V$ $(j=1, \cdots, k)$ are all homogeneous.

There is also a natural (right) action of the Sergeev superalgebra
$Ser_k$ on $\V^{\tensor k}$ \cite{Ser}. Let $w_{j}$ be the
homogeneous element in $\V$. The  $s_j$'s act on $\V^{\tensor k}$ by
graded place permutation:
\begin{align*}
  (w_{1}\tensor \cdots \tensor w_{k})\cdot s_j
  = (-1)^{|w_j||w_{j+1}|} w_{1}\tensor \cdots \tensor w_{j-1}\tensor w_{j+1}\tensor w_j
  \tensor w_{j+2} \tensor \cdots  \tensor w_{k}.
\end{align*}
The action of $c_j$'s on $\V^{\tensor k}$ is defined as follows:
\begin{align*}
  (w_{1}\tensor \cdots \tensor w_{k})\cdot c_j
  &= (-1)^{|w_1|+ \cdots +|w_{j-1}|}w_{1}\tensor \cdots \tensor w_{j-1}\tensor P(w_{j}) \tensor w_{j+1} \tensor \cdots \tensor w_{k},
\end{align*}
where $P$ is defined in \eqref{map:odd involution}. It is easy to
check that the actions of $s_i$ and $c_j$ $(i=1, \ldots, k-1, j=1,
\cdots, k)$ give rise to an action of the Sergeev superalgebra
$Ser_k$ on $\V^{\tensor k}$. Thus we obtain an algebra homomorphism
\begin{equation} \label{def:action of Ser_k}
\Phi_k: Ser_k \longrightarrow \End_{\C}(\V^{\tensor k})^{\op}.
\end{equation}
Since $Ser_{k}$ acts on $\V^{\otimes k}$ from the right, we consider
the opposite algebra $\End_{\C}(\V^{\tensor k})^{\op}$.

Let $\End_{\qn}(\V^{\tensor k})$ be the supercentralizer algebra of
the $\qn$-action on $\V^{\otimes k}$.  Then we have
$$\End_{\qn}(\V^{\tensor k})=\End_{\qn}(\V^{\tensor k})_{\ol 0}
\oplus \End_{\qn}(\V^{\tensor k})_{\ol 1}, $$ where
\begin{align*}
\End_{\qn}(\V^{\tensor k})_{j}:& = \{f \in  \End_{\C}(\V^{\tensor k})_{j} \ | \ \  f(g \cdot w)=(-1)^{|g||f|} \, g \cdot f(w)  \\
&\text{  for all homogeneous elements $g \in \qn$ and $w\in
\V^{\tensor k}$} \}.
\end{align*}
In \cite{Ser}, Sergeev proved the following fundamental theorem, the
{\it Schur-Weyl-Sergeev duality}.

\begin{theorem} \cite[Theorem 3,4]{Ser} \label{th:Sergeev duality}
\hfill

{\rm
(a)  The actions of $\qn$ and $Ser_k$ on $\V^{\tensor k}$
supercommute with each other; i.e., the image of $\Phi_{k}$ is in
$\End_{\qn}(\V^{\tensor k})$.

(b) The superalgebra homomorphism
   $\Phi_k: Ser_k \longrightarrow \End_{\qn}(\V^{\tensor k})^{{\rm op}}$ is
   surjective.

(c) If $n \geq k$, $\Phi_k$ is an isomorphism. }

\end{theorem}

\vskip 5mm

\section{Diagrammatic realization of Sergeev superalgebras}

In this section, we construct a new diagrammatic realization of
Sergeev superalgebras. A \emph{k-superdiagram} is defined to be a
diagram consisting of $k$ vertices on its top and bottom row
together with $k$ edges such that each vertex on the top row must be
connected to exactly one vertex on the bottom row and each edge may
(or may not) be \emph{marked}. We denote the normal edge by $
{\beginpicture \setcoordinatesystem units <0.6cm,0.3cm> \setplotarea
x from 0.5 to 1.5, y from -1.5 to 2.5 \put{$\bullet$} at  1 2
\put{$\bullet$} at 1 -1 \plot 1 2  1 -1 /
\endpicture}
$
and
the marked edge by $
{\beginpicture
\setcoordinatesystem units <0.6cm,0.3cm>
\setplotarea x from 0.5 to 1.5, y from -1.5 to 2.5
\put{$\bullet$} at  1 2
\put{$\bullet$} at  1 -1
\plot 1 2  1 -1 /
\arrow <3 pt> [1,2] from 1 -1 to 1 0.75
\endpicture}
$.
We number the vertices in each row of a $k$-superdiagram
from left to right with $1,2, \ldots, k$.
In Figure 1, $d$ is an example of a $5$-superdiagram.
\begin{figure}[!h]
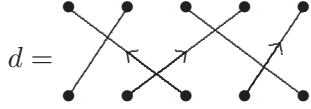

$d={\beginpicture
\setcoordinatesystem units <0.78cm,0.39cm>
%\setplotarea x from 0 to 6, y from -2 to 3
\put{$\bullet$} at  1 -1  \put{$\bullet$} at  1 2
\put{$\bullet$} at  2 -1  \put{$\bullet$} at  2 2
\put{$\bullet$} at  3 -1  \put{$\bullet$} at  3 2
\put{$\bullet$} at  4 -1  \put{$\bullet$} at  4 2
\put{$\bullet$} at  5 -1  \put{$\bullet$} at  5 2
\plot 1 2  3 -1 /
\plot 2 2  1 -1 /
\plot 3 2  5 -1 /
\plot 4 2  2 -1 /
\plot 5 2  4 -1 /
\arrow <3 pt> [1,2] from 3 -1 to  2 0.5
\arrow <3 pt> [1,2] from 2 -1 to  3 0.5
\arrow <3 pt> [1,2] from 4 -1 to  4.6 0.8
\endpicture}$
\caption{$5$-superdiagram $d$} \label{Fig:example1}
\end{figure}

As usual, we can identify $\tau \in \Sigma_k$ with \emph{its
permutation diagram}, the diagram with $k$ vertices on the top and
bottom row and $k$ normal vertical edges connecting the $i$-th
vertex on the top row to the $\tau(i)$-th vertex (from left) on the
bottom row for each $1 \leq i \leq k$.

The $k$-superdiagram which has even (respectively, odd) number of
marked edges is regarded as \emph{even} (respectively, \emph{odd}).
Let $\overrightarrow{D}_k$ be the superspace with a basis consisting
of $k$-superdiagrams. The $k$-superdiagrams without marked edges
form a basis of $\C \Sigma_k$.

Now we define a multiplication on $\DD_k$ in two steps.

%It is obtained by the two steps.

{\bf Step 1:} Marked concatenation

For $d_1, d_2 \in \DD_k$, we define the {\it marked concatenation}
$d_1 * d_2$ as follows. We first put $d_1$ under $d_2$ and identify
the vertices on the bottom row of $d_2$ with the vertices on the top
row of $d_1$. Next, we declare an edge in this diagram is marked if
and only if the number of marked edges from $d_1$ and $d_2$ to form
this edge is odd. The diagram thus obtained is the marked
concatenation $d_1 * d_2$.

For example, if
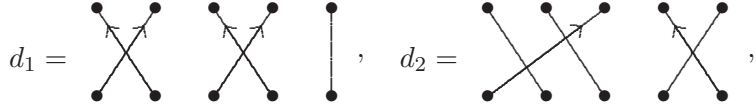
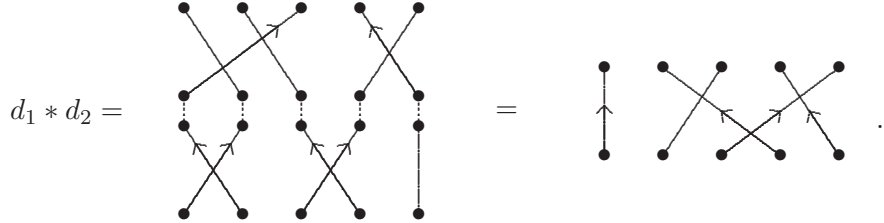
\begin{center}
$\ {\beginpicture
\setcoordinatesystem units <0.78cm,0.39cm>
\setplotarea x from 0 to 6, y from -2 to 3
\put{$d_1 = $} at 0 0.3
\put{$\bullet$} at  1 -1  \put{$\bullet$} at  1 2
\put{$\bullet$} at  2 -1  \put{$\bullet$} at  2 2
\put{$\bullet$} at  3 -1  \put{$\bullet$} at  3 2
\put{$\bullet$} at  4 -1  \put{$\bullet$} at  4 2
\put{$\bullet$} at  5 -1  \put{$\bullet$} at  5 2
\plot 1 2  2 -1 /
\plot 2 2  1 -1 /
\plot 3 2  4 -1 /
\plot 4 2  3 -1 /
\plot 5 2  5 -1 /
\arrow <3 pt> [1,2] from 2 -1 to  1.2 1.4
\arrow <3 pt> [1,2] from 4 -1 to  3.2 1.4
\arrow <3 pt> [1,2] from 1 -1 to  1.8 1.4
\arrow <3 pt> [1,2] from 3 -1 to  3.8 1.4
\put{,} at 5.5 0.3
\endpicture}$
${\beginpicture
\setcoordinatesystem units <0.78cm,0.39cm>
\setplotarea x from 0 to 6, y from -2 to 3
\put{$d_2 = $} at 0 0.3
\put{$\bullet$} at  1 -1  \put{$\bullet$} at  1 2
\put{$\bullet$} at  2 -1  \put{$\bullet$} at  2 2
\put{$\bullet$} at  3 -1  \put{$\bullet$} at  3 2
\put{$\bullet$} at  4 -1  \put{$\bullet$} at  4 2
\put{$\bullet$} at  5 -1  \put{$\bullet$} at  5 2
\plot 1 2  2 -1 /
\plot 2 2  3 -1 /
\plot 3 2  1 -1 /
\plot 4 2  5 -1 /
\plot 5 2  4 -1 /
\arrow <3 pt> [1,2] from 1 -1 to  2.6 1.4
\arrow <3 pt> [1,2] from 5 -1 to  4.2 1.4
%\arrow <3 pt> [1,2] from 4 -1 to  4.6 0.8
\put{,} at 5.5 0.3
\endpicture}$
\end{center}
then
$${\beginpicture
\setcoordinatesystem units <0.78cm,0.39cm>
\setplotarea x from 0 to 6, y from -2 to 3
\put{$d_1 *d_2 = $} at -1 -1.5
\put{$\bullet$} at  1 -1  \put{$\bullet$} at  1 2
\put{$\bullet$} at  2 -1  \put{$\bullet$} at  2 2
\put{$\bullet$} at  3 -1  \put{$\bullet$} at  3 2
\put{$\bullet$} at  4 -1  \put{$\bullet$} at  4 2
\put{$\bullet$} at  5 -1  \put{$\bullet$} at  5 2
\plot 1 2  2 -1 /
\plot 2 2  3 -1 /
\plot 3 2  1 -1 /
\plot 4 2  5 -1 /
\plot 5 2  4 -1 /
\arrow <3 pt> [1,2] from 1 -1 to  2.6 1.4
\arrow <3 pt> [1,2] from 5 -1 to  4.2 1.4
\put{$\bullet$} at  1 -5  \put{$\bullet$} at  1 -2
\put{$\bullet$} at  2 -5  \put{$\bullet$} at  2 -2
\put{$\bullet$} at  3 -5  \put{$\bullet$} at  3 -2
\put{$\bullet$} at  4 -5  \put{$\bullet$} at  4 -2
\put{$\bullet$} at  5 -5  \put{$\bullet$} at  5 -2
\plot 1 -2  2 -5 /
\plot 2 -2  1 -5 /
\plot 3 -2  4 -5 /
\plot 4 -2  3 -5 /
\plot 5 -2  5 -5 /
\arrow <3 pt> [1,2] from 2 -5 to  1.2 -2.6
\arrow <3 pt> [1,2] from 4 -5 to  3.2 -2.6
\arrow <3 pt> [1,2] from 1 -5 to  1.8 -2.6
\arrow <3 pt> [1,2] from 3 -5 to  3.8 -2.6
\setdashes  <.4mm,.3mm>
\plot 1 -2   1 -1 /
\plot 2 -2   2 -1 /
\plot 3 -2   3 -1 /
\plot 4 -2   4 -1 /
\plot 5 -2   5 -1 /
\setsolid
\put{=} at 6.5 -1.5 \
\put{$\bullet$} at  8 -3  \put{$\bullet$} at  8 0
\put{$\bullet$} at  9 -3  \put{$\bullet$} at  9 0
\put{$\bullet$} at  10 -3  \put{$\bullet$} at  10 0
\put{$\bullet$} at  11 -3  \put{$\bullet$} at  11 0
\put{$\bullet$} at  12 -3  \put{$\bullet$} at  12 0
\plot 8 -3  8 0 /
\plot 9 -3  10 0 /
\plot 10 -3  12 0 /
\plot 11 -3  9 0 /
\plot 12 -3  11 0 /
\arrow <3 pt> [1,2] from 8 -3 to  8 -1.3
\arrow <3 pt> [1,2] from 11 -3 to  10 -1.5
\arrow <3 pt> [1,2] from 10 -3 to 11 -1.5
\arrow <3 pt> [1,2] from 12 -3 to 11.5 -1.5
\put{.} at 12.7 -2
\endpicture}$$

{\bf Step 2:} The multiplication on $\DD_{k}$

% We define
%\begin{align} \label{eq:multiplication in D_k}
%d_1d_2=(-1)^{\rho(d_1,d_2)+\ell(d_1,d_2)}d_1 *d_2,
%\end{align}
%where $\rho(d_1,d_2), \ell(d_1,d_2)  \in \Z_{\geq 0}$ are defined below.

We first define the numbers  $\rho(d_1,d_2), \ell(d_1,d_2)  \in
\Z_{\geq 0}$ as follows.

\begin{enumerate}
\item Let $\rho(d_1,d_2)$ be the number of edges in $d_1 *d_2$ obtained by
connecting a marked edge in $d_1$ and a marked edge in $d_2$.

\item Consider the $i$-th marked edge $e_i$ in $d_1$ reading the top vertices from left to
right. Let $a_i$ be the top vertex of the edge in $d_2$ that is
connected to  $e_i$. Then we obtain a sequence $a_1 \cdots a_p$. Let
$b_1 \cdots b_q$ be the sequence obtained by reading the top
vertices of the marked edges in $d_2$.

Consider the sequence $a_1 \cdots a_p b_1 \cdots b_q$. Suppose that
the smallest entry in this sequence appears $i$ times, and the next
smallest one appears $j$ times, etc. We replace the $i$ occurrences
of the smallest entry by $1, 2, \cdots, i$ {\it from left to right},
the $j$ occurrences of the next smallest entry by  $i+1, \ldots,
i+j$ {\it from left to right}, etc. We denote by $a'_1 \cdots a'_p
b'_1 \cdots b'_q$ the new sequence obtained in this manner. Then we
obtain a permutation $\left(
                                         \begin{array}{ccccccc}
                                           1 & 2  &         \cdots &p& p+1 &\cdots &p+q\\
                                           a'_1 & a'_2 &\cdots & a'_p& b'_1 &\cdots&b'_{q}   \\
                                         \end{array}
                                       \right)$
in $\Sigma_{p+q}$. Let $\ell(d_1,d_2)$ be the length of the above
permutation.
\end{enumerate}

We now define the multiplication on $\DD_{k}$ by
\begin{equation} \label{eq:multiplication in D_k}
d_{1} d_{2} = (-1)^{\rho(d_1, d_2) + \ell(d_{1}, d_{2})} d_1 * d_2.
\end{equation}
The number $\ell(d_1, d_2)$ is called the
\emph{arranging number} for a sequence $a_1 \cdots a_p b_1 \cdots
b_q$.

\begin{example}
In the above example, we have $\rho(d_1,d_2)=1$. In (2), we obtain
$a_1a_2a_3a_4=3125$, $b_1b_2=34$ and
$a'_1a'_2a'_3a'_4b'_1b'_2=312645$. Hence
$\ell(d_1,d_2)=\ell(s_2s_1s_5s_4)=4$. It follows that
$$d_1d_2=(-1)^{1+4} d_1*d_2. $$ 
\end{example}

If $d_1$ has $i$ marked edges and $d_2$ has $j$ marked edges, then
the number of the marked edges in $d_1d_2$ is $i+j-2 \rho(d_1,d_2)$.
%So the multiplication on $\DD_k$ is $\Z_2$-graded.
It follows that
$\DD_k$ is a superalgebra (which may not be associative at this
point). The identity element in $\DD_k$ is the diagram such that
each vertex on the top row is connected with the corresponding
vertex on the bottom row by the normal edge.

Let $d$ be a $k$-superdiagram. If we forget the marks on $d$, we get
a permutation, say, $\sigma$. By reading the vertices at the top of
marked edges  from left to right, we obtain a sequence $i_1 \cdots
i_m$ such that $1 \leq i_1 < \cdots < i_m \leq k$. For example, in
Figure \ref{Fig:example1}, we obtain
$$\sigma=s_2s_4s_1s_3=\left(
  \begin{array}{ccccc}
   1 & 2 & 3 & 4 & 5 \\
   3 & 1 & 5 & 2 & 4
  \end{array}
\right) \in \Sigma_5, \qquad i_1i_2i_3=145.$$

Recall that the Sergeev superalgebra $Ser_k$ has a basis
\begin{align} \label{def:basis element form of Ser_k}
\{ \sigma c_{i_1} \cdots c_{i_m} \ | \ \sigma \in \Sigma_k, 1 \leq i_1 < \cdots <i_m \leq k \}.
\end{align}
Therefore we get a linear map $\phi_{k}: \DD_{k} \rightarrow
Ser_{k}$ given by
$$ d \longmapsto \sigma c_{i_{1}} \cdots c_{i_m},$$
where $\sigma$ and $i_1, \ldots, i_m$ are defined above. Since $\dim
\DD_{k} = 2^k k! = \dim Ser_{k}$, $\phi$ is a linear isomorphism.

%We map a $k$-superdiagram $d$ into an element $\sigma c_{i_1} \cdots c_{i_m} \in Ser_k$.
%Denote this linear map by
%\begin{align} \label{def:map between D_k and Ser_k}
%\phi_k \ : \  \overrightarrow{D}_k  \longrightarrow \ Ser_k.
%\end{align}
%Since the linear map $\phi_k$ carries the basis of $\DD_k$ onto the basis of $Ser_k$
%and $\dim_{\C} \DD_k =\dim_{\C}Ser_k=2^k k!$,
%$\DD_k$ is isomorphic to $Ser_k$ as a superspace via the map $\phi_k$.

We would like to stress that, for instance, the $5$-superdiagram
${\beginpicture \setcoordinatesystem units <0.6cm,0.3cm>
\setplotarea x from 0.5 to 5.5, y from -2 to 3
%\put{$d_1 = $} at 0 .5
\put{$\bullet$} at  1 -1  \put{$\bullet$} at  1 2
\put{$\bullet$} at  2 -1  \put{$\bullet$} at  2 2
\put{$\bullet$} at  3 -1  \put{$\bullet$} at  3 2
\put{$\bullet$} at  4 -1  \put{$\bullet$} at  4 2
\put{$\bullet$} at  5 -1  \put{$\bullet$} at  5 2
\plot 1 2  1 -1 /
\plot 2 2  2 -1 /
\plot 3 2  3 -1 /
\plot 4 2  4 -1 /
\plot 5 2  5 -1 /
\arrow <3 pt> [1,2] from 1 -1 to  1 0.7
\arrow <3 pt> [1,2] from 2 -1 to  2 0.7
\endpicture}$
corresponds to an element $c_1c_2 \in Ser_5$, not $c_2 c_1$.

We now prove the main result of this section which gives a
diagrammatic realization of the Sergeev superalgebra $Ser_{k}$.

\begin{theorem}\label{th:DD_{k}}
{\rm
The linear map $\phi_k$ preserves the multiplication:
 $$\phi_k(d_1d_2)=\phi_k(d_1)\phi_k(d_2) \quad \text{for all} \ d_1,d_2 \in
 \DD_k.$$
 Hence $\DD_{k}$ is isomorphic to $Ser_{k}$ as an associative
 superalgebra.
 }
\end{theorem}

\begin{proof}
 Since $\phi_k((\DD_k)_{i}) \subset (Ser_k)_{i}$ for all $i \in \Z_2$, $\phi_k$ is an even linear map.
 For $d_1,d_2 \in \DD_k$, let $\phi_k(d_1)=\sigma c_{i_1} \cdots c_{i_{p}}, \phi_k(d_2)=\tau c_{j_1} \cdots c_{j_{q}}$,
  ($\sigma,\tau \in \Sigma_k, 1 \leq i_1 < \cdots <i_p \leq k, 1 \leq j_1 < \cdots <j_q \leq k$).
  Then
 \begin{align*}
   \phi_k(d_1)\phi_k(d_2)=\sigma c_{i_1} \cdots c_{i_{p}} \tau c_{j_1} \cdots c_{j_{q}}=
   \sigma\tau c_{\tau^{-1}(i_1)} \cdots c_{\tau^{-1}(i_p)}c_{j_1} \cdots c_{j_q}.
    \end{align*}
%To make $c_{\tau^{-1}(i_1)} \cdots c_{\tau^{-1}(i_p)}c_{j_1} \cdots c_{j_q}$ into the form of
%\eqref{def:basis element form of Ser_k},
%we rearrange it.
We rearrange $c_{\tau^{-1}(i_1)} \cdots c_{\tau^{-1}(i_p)}c_{j_1}
\cdots c_{j_q}$ to obtain an element of the form \eqref{def:basis
element form of Ser_k}.
%Let $c_{h_1} \cdots c_{h_t}$ be its basis element form.
Since $${c_i}^2=-1, c_ic_j=-c_jc_i \quad \text{for} \ i \neq j,$$ we
have  $$c_{\tau^{-1}(i_1)} \cdots c_{\tau^{-1}(i_p)}c_{j_1} \cdots
c_{j_q}= (-1)^{x+y}c_{h_1} \cdots c_{h_t},$$ where $x$ is the number
of pairs with the same entries in $\tau^{-1}(i_1) \cdots
\tau^{-1}(i_p)j_1 \cdots j_q$ and $y$ is the smallest number of
transpositions that needed to rearrange $\tau^{-1}(i_1) \cdots
\tau^{-1}(i_p)j_1 \cdots j_q$ in order from the smallest one to the
biggest one (from left to right). We observe
$$\tau^{-1}(i_1) \cdots \tau^{-1}(i_p)j_1 \cdots j_q=a_1 \cdots a_p b_1 \cdots b_q$$
and $x=\rho(d_1,d_2)$, $y=\ell(d_1,d_2)$. Also,
$\phi_k(d_1 *d_2)=\sigma\tau c_{h_1} \cdots c_{h_t}$. Therefore,
\begin{align*}
  \phi_k(d_1)\phi_k(d_2)&=(-1)^{\rho(d_1,d_2)+\ell(d_1,d_2)}\sigma\tau  c_{h_1} \cdots c_{h_t}\\
  &=\phi_k(d_1d_2).
\end{align*}
\end{proof}

%Since $\phi_k$ is an isomorphism and $Ser_k$ is associative, we obtain:
%\begin{theorem}\label{th:DD_{k}}
%  A superspace $\DD_k$ is an associative superalgebra. Moreover, $\DD_k$ is isomorphic to $Ser_k$ as an associative superalgebra.
%\end{theorem}

\begin{remark}
There have been a lot of works done on diagram algebras  with marked
edges or marked vertices
 (see, for example, \cite{B,E, PKem, PKen}).
 This work is different from those in that we deal with
 superalgebras.
\end{remark}

Since $\DD_{k}$ is isomorphic to $Ser_{k}$, it is isomorphic to
$\End_{\qn}(\V^{\tensor k})^{\rm op}$. Still, we will give a direct
proof of this fact  because it will serve as a guideline for the
proof of our main theorem in Section 4.

We define a (right) action of each $k$-superdiagram on $\V^{\tensor
k}$ as follows. Let $I:=\{ 1, \ldots, n, \ol 1, \ldots, \ol n \}$
and ${\bf i}:=(i_1, \ldots, i_k) \in I^k$. Define $|i|:=0$ if $i \in
\{1, \ldots, n \}$ and $|i|:=1$ if $i \in \{\ol 1, \ldots, \ol n
\}$. For $d \in \DD_{k}$ and $\ii, \jj \in I^{k}$, we label the
vertices at the bottom row of $d$ with
 $i_1, \ldots , i_{k}$ (from left to right) and the vertices at the top row of $d$ by $j_1, \ldots, j_{k}$
 (from left to right).
  We denote this labeled diagram by $_{\ii} d_{\jj}$ and
define the \emph{weight} of the labeled diagram $_{\ii} d_{\jj}$ to
be
 \begin{align*}
   \wtd := \prod_{e} \delta_e \prod_c(-1)^{|c|}\prod_{\widetilde{e}_i}(-1)^{m(\widetilde{e}_i)},
 \end{align*}
 where the notations are explained below:
 \begin{enumerate}
   \item  The first product is taken over all edges $e$ in $\idj$.
For each normal edge $e$ with labels  $i$ and $j$ at the end points,
let $\delta_e=\delta_{i,j}$.
   We define $\delta_e=\delta_{{\ol i},j}$ for each marked edge $e$ with label $i$ at the bottom vertex
   and $j$ at the top vertex.

   \item The second product is taken over all crossings $c$ in $\idj$.
    We define $|c|=|e_1||e_2|$ for each crossing of two different edges $e_1$ and $e_2$, where
   we define $|e_i|=|i_m|$ if the edge $e_i$ is labeled $i_m$
   at the \emph{bottom} vertex.

   \item The last product is taken over all marked edges in $\idj$.
 Read the top vertices of the marked edges from left to right to
 obtain a sequence $b_1 \cdots b_q$ for $1 \leq b_1< \cdots <b_q \leq k$
   and let each marked edge be $\widetilde{e}_i$.
   We define ${m(\widetilde{e}_i)}=|j_1| + \cdots + |j_{b_i}|$.
 \end{enumerate}

If $\wt(\idj)\neq 0$, we say that $\idj$ is \emph{consistently
labeled}. Since there are only vertical edges in a $k$-superdiagram
$d$, the labeling $\ii \in I^k$ determines a unique labeling $\jj
\in I^k$ such that $\wt(\idj) \neq 0$.

\begin{example}
 For $\ii=(1,2, \ol 1, \ol 1, \ol 2)$ and $\jj=(1,1,\ol 2, \ol 2, 1)$,
the diagram $\idj$ of $d$ in Figure 1 is consistently labeled and
has $\wt(\idj)=-1$ for $n=2$:
$$\idj=
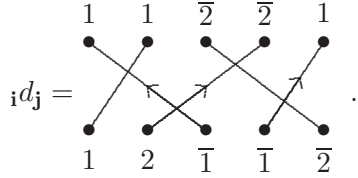
{\beginpicture
\setcoordinatesystem units <0.78cm,0.39cm>
%\setplotarea x from 0 to 6, y from -2 to 3
\put{$\bullet$} at  1 -1  \put{$\bullet$} at  1 2
\put{$\bullet$} at  2 -1  \put{$\bullet$} at  2 2
\put{$\bullet$} at  3 -1  \put{$\bullet$} at  3 2
\put{$\bullet$} at  4 -1  \put{$\bullet$} at  4 2
\put{$\bullet$} at  5 -1  \put{$\bullet$} at  5 2
\put{$1$} at  1 -2
\put{$2$} at  2 -2
\put{$\ol 1$} at  3 -2
\put{$\ol 1$} at  4 -2
\put{$\ol 2$} at  5 -2

\put{$1$} at  1 3
\put{$1$} at  2 3
\put{$\ol 2$} at  3 3
\put{$\ol 2$} at  4 3
\put{$1$} at  5 3

\plot 1 2  3 -1 /
\plot 2 2  1 -1 /
\plot 3 2  5 -1 /
\plot 4 2  2 -1 /
\plot 5 2  4 -1 /
\arrow <3 pt> [1,2] from 3 -1 to  2 0.5
\arrow <3 pt> [1,2] from 2 -1 to  3 0.5
\arrow <3 pt> [1,2] from 4 -1 to  4.6 0.8
\endpicture} \ \ .$$
To calculate the weight, we use
\begin{align*}
\prod_{c}(-1)^{|c|}&=(-1)^{|1||\ol 1|+|2||\ol 1|+|2||\ol 2|+|\ol 1||\ol 2|}=-1,\\
(-1)^{m(\te_1)+m(\te_2)+m(\te_3)}&=(-1)^{|1| + (|1|+|1|+|\ol 2|+|\ol 2|)+(|1|+|1|+|\ol 2|+|\ol 2|+|1|)}=1.
\end{align*}
\end{example}

Let $v_{\ii}=v_{i_1} \tensor \cdots \tensor v_{i_k}$ for $\ii=(i_1,
\ldots, i_k) \in I^k$. Obviously, $\{v_{\ii} \ | \ \ii \in I^k \}$
forms a basis of $\V^{\tensor k}$. Define a linear map
$\widetilde{\Phi}_{k}$ from $\DD_k$ into $\End_{\C}(\Vk)$ by
 \begin{eqnarray} \label{def:action of D_k}
  \widetilde{\Phi}_{k} \ : \ \DD_k &\longrightarrow &  \End_{\C}(\Vk)\\
  \nonumber       d &\longmapsto &  \left( v_{\ii}  \mapsto
        \sum_{\jj \in I^{k}} \wt(\idj) \ v_{\jj}  \right).
\end{eqnarray}

\nc{\tPhi}{\widetilde{\Phi}}
\vskip 3mm

\begin{lemma} \label{lem:duality for D_k and qn}
 {\rm $\widetilde{\Phi}_k(d)=\Phi_k(\phi_k(d))$ for all $d \in
 \DD_k$.}
\end{lemma}
\begin{proof}
We observe that the action $\tPhi_k$ of the $k$-superdiagram without
marked edges is the same as the action $\Phi_k$ of the corresponding
permutation on $\Vk$.
    Assume that $d \in \DD_k$ has some marked edges.
     If we forget the marks on $d$, we obtain a permutation $\sigma$.
    Read the top vertices of the marked edges from left to right to obtain a
    sequence $b_1 \cdots b_q$ for $1 \leq b_1< \cdots <b_q \leq k$.
   For $\ii \in I^k$, $\idj$ is consistently labeled when
    $\jj=(i_{\sigma(1)}, \cdots, i_{\sigma(b_1 -1)}, \ol{i_{\sigma(b_1)}}, \cdots, \ol{i_{\sigma(b_q)}},
    i_{\sigma(b_q + 1)}, \cdots, i_{\sigma(k)})$.
    Therefore,
   $\tPhi_k(d)(v_{\ii})=(-1)^{\prod |c|}(-1)^{m(\te_1) +\cdots + m(\te_q)} v_{\jj}$,
   where $(-1)^{\prod |c|}$ comes from the crossing part of $\idj$.

   By the definition, $\phi_k(d)=\sigma c_{b_1} \cdots c_{b_q}$.
   Therefore we have
\begin{align*}
\Phi_k(\phi_k(d))& \big(v_{\ii} \big) = \Phi_k(\sigma c_{b_1} \cdots c_{b_q})(v_{\ii})=\Phi_k(c_{b_q})\cdots \Phi_{k}(c_{b_1}) \Phi_k(\sigma) \big(v_{\ii} \big) \\
   &=(-1)^{\prod |c|}\Phi_k(c_{b_q})\cdots \Phi_{k}(c_{b_1}) \big( v_{i_{\sigma(1)}} \tensor \cdots \tensor v_{i_{\sigma(k)}}) \\
   &=(-1)^{\prod |c|+m(\te_1)} \Phi_k(c_{b_q})\cdots \Phi_{k}(c_{b_2} \big)
   \big(v_{i_{\sigma(1)}} \tensor \cdots \tensor v_{\ol{i_{\sigma(b_1)}}} \tensor \cdots \tensor v_{i_{\sigma(k)}}\big) \\
   &=(-1)^{\prod |c|+m(\te_1) +\cdots +m(\te_q)}
   \ v_{i_{\sigma(1)}} \tensor \cdots \tensor v_{\ol{i_{\sigma(b_1)}}} \tensor \cdots
   \tensor v_{\ol{i_{\sigma(b_q)}}} \tensor \cdots \tensor
   v_{i_{\sigma(k)}},
   \end{align*}
which yields the desired result.
\end{proof}

Therefore, the map $\widetilde{\Phi}_k$ provides a well-defined
(right) action of $\DD_k$ on $\V^{\tensor k}$. Note that the action
of the $k$-superdiagram without marked edges (i.e., permutation
diagram) on $\V^{\otimes k}$ is given by graded place permutation.
Moreover, we have:

\begin{prop} \label{prop:duality between D_k and qn} \hfill

  {\rm (a) The actions of $\DD_k$ and $\qn$ on $\V^{\otimes k}$ supercommute with each other.

  (b) The homomorphism
   $\widetilde{\Phi}_k: \DD_k \longrightarrow \End_{\qn}(\V^{\tensor k})^{{\rm op}}$ is
   surjective.

  (c) If $n \geq k$, $\tPhi_{k}$ is an isomorphism.}

\end{prop}

\begin{proof}
Because $\phi_k$ is an even homomorphism, we deduce (a) from Theorem
\ref{th:Sergeev duality} and Lemma \ref{lem:duality for D_k and qn}.
The assertions (b) and (c) follow from  Theorem \ref{th:Sergeev
duality} and Lemma \ref{lem:duality for D_k and qn} immediately.
\end{proof}

\vskip 5mm

\section{The walled Brauer superalgebras}

In this section, we will introduce a new family of superspaces $\BB$
with a basis consisting of $(r,s)$-superdiagrams. We will also
describe the natural (right) action of $\BB$ on $\mix$ and show that
it defines a linear isomorphism between $\BB$ and $\End_{\qn}(\mix)$
whenever $n \ge r+s$. Here, $\V=\C^{n|n}$ is the natural
representation of $\qn$ and  $\W=\V^{*}$ is its dual.

To begin with, we explain the $\qn$-supermodule structure on the
mixed tensor space $\mix$. Fix
$r,s \in \Z_{\geq 0}$ such that $r+s > 0$.
Let $I =\{1, \ldots, n, \ol 1, \ldots,
\ol n \}$. Recall that we fix a basis $\B = \{v_1, \ldots, v_n,
v_{\ol 1}, \ldots, v_{\ol n}\}$ of $\V$. Let us denote its dual
basis by $\B^{*} =\{w_1, \ldots, w_n, w_{\ol 1}, \ldots, w_{\ol
n}\}$ of $\W$ such that $w_i(v_j)=\delta_{i,j}$ for $i,j \in I$. The
universal enveloping algebra $U(\qn)$ is a Hopf superalgebra with
the comultiplication $\Delta$ and the antipode $S$ given by
$$\Delta(g)=g \tensor 1 +1 \tensor g, \quad S(g) =-g \ \ (g \in
\qn).$$ The dual space $\W$ becomes a $\qn$-supermodule via
$$(g w)(v):=(-1)^{|g||w|}w(S(g) v)$$ for homogeneous elements
$g \in \qn, w \in \W$ and $v\in \V$.

Let $E_{i,j}$ be the $n \times n$ matrix having $1$ at the
$(i,j)$-entry and $0$ elsewhere.
We define
$$e_{i,j}:=\left(
            \begin{array}{cc}
              E_{i,j} & 0 \\
              0 & E_{i,j} \\
            \end{array}
          \right), \ \ e_{i,\ol j}:=\left(
            \begin{array}{cc}
              0 & E_{i,j}  \\
              E_{i,j} &0\\
            \end{array}
          \right)$$
for $1 \leq i,j \leq n$. We obtain the formulae
\begin{equation} \label{eq:action of q(n)-elements on V,W}
\begin{aligned}
  e_{i,j} \ v_k &=\delta_{j,k}v_i+\delta_{\ol j,k}v_{\ol i},\\
    e_{i,j} \ w_k&=-(-1)^{|j||k|}(\delta_{i,k}w_j+\delta_{\ol i,k}w_{\ol j})
\end{aligned}
\end{equation}
for all $i \in \{ 1, \ldots, n\}$ and $j, k \in I$. Using the
comultiplication $\Delta$, we give a $\qn$-supermodule structure to
the mixed tensor space $\mix$. Let us denote this action by
$\rho:U(\qn) \rightarrow \End_{\C}(\mix)$. Then the {\it
supercentralizer algebra} $\End_{\qn}(\mix)$ of the $\qn$-action on
$\mix$ has a decomposition
$$\End_{\qn}(\mix):=\End_{\qn}(\mix)_{\ol 0} \oplus \End_{\qn}(\mix)_{\ol 1} ,$$ where
$$
\begin{aligned}
\End_{\qn} & (\mix)_{j}:  = \{f \in  \End_{\C}(\mix)_{j} \ |  \\
& \rho(g) f  =(-1)^{|g||f|} f \rho(g)  \ \  \text{for all
homogeneous} \ g \in \qn \}.
\end{aligned}
$$

% We will often consider the $\mathfrak{gl}(n|n)$-supermodule structure on $\V$.
%If we define the action of $\mathfrak{gl}(n|n)$ on $\V$
%by the matrix multiplication, we can consider $\V$ as a $\mathfrak{gl}(n|n)$-supermodule.
%Since the antipode and comultiplication of $\mathfrak{gl}(n|n)$ and $\qn$ are the same,
%we can consider $\W$, $\V^{\tensor r}$, $\W^{\tensor s}$ and $\mix$ as $\mathfrak{gl}(n|n)$-supermodules.

\vskip3mm

Now we construct a combinatorial model for $\End_{\qn}(\mix)$. An
\emph{$(r,s)$-superdiagram} is a diagram with $(r+s)$ vertices on
the top and bottom rows, and a vertical wall separating the $r$-th
and $(r+1)$-th vertices (from left) in each row. Each vertex must be
connected to exactly one other vertex and each edge may (or may not)
be \emph{marked}. In addition, we require that each vertical edge
cannot cross the wall and each horizontal edge should cross the
wall. We number the vertices in each row of an $(r,s)$-superdiagram
from left to right with $1,2, \ldots, r+s$. The examples of
$(3,2)$-superdiagrams are given in Figure 2:
\begin{center}
\begin{figure}[!h]
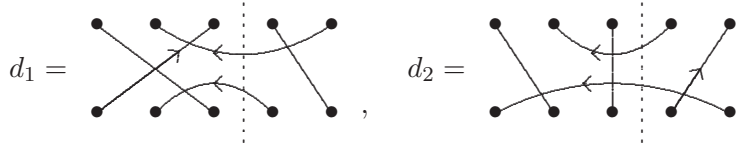

${\beginpicture
\setcoordinatesystem units <0.78cm,0.39cm>
\setplotarea x from 0 to 5.5, y from -1.5 to 4
\put{$d_1 = $} at 0 1.5
\put{$\bullet$} at  1 3  \put{$\bullet$} at  1 0
\put{$\bullet$} at  2 3  \put{$\bullet$} at  2 0
\put{$\bullet$} at  3 3  \put{$\bullet$} at  3 0
\put{$\bullet$} at  4 3  \put{$\bullet$} at  4 0
\put{$\bullet$} at  5 3  \put{$\bullet$} at  5 0
\plot 1 3  3 0 /
\plot 3 3  1 0 /
\plot 4 3  5 0 /
\arrow <3 pt> [1,2] from 1 0 to  2.4 2.1
\arrow <3 pt> [1,2] from 3.05 2.05 to 3 2.05
\arrow <3 pt> [1,2] from 3.05 1 to 3 1
\setdashes  <.4mm,1mm>
\plot 3.5 -1   3.5 4 /
\setsolid
\setquadratic
\plot 2 3  3.5 2 5 3 /
\plot 2 0 3 1  4 0 /
\endpicture}$,  \hskip1em  ${\beginpicture
\setcoordinatesystem units <0.78cm,0.39cm>
\setplotarea x from 0 to 5.5, y from -1.5 to 4
\put{$d_2 = $} at 0 1.5
\put{$\bullet$} at  1 3  \put{$\bullet$} at  1 0
\put{$\bullet$} at  2 3  \put{$\bullet$} at  2 0
\put{$\bullet$} at  3 3  \put{$\bullet$} at  3 0
\put{$\bullet$} at  4 3  \put{$\bullet$} at  4 0
\put{$\bullet$} at  5 3  \put{$\bullet$} at  5 0
\plot 1 3  2 0 /
\plot 3 3  3 0 /
\plot 5 3  4 0 /
\arrow <3 pt> [1,2] from 4 0 to  4.5 1.5
\arrow <3 pt> [1,2] from 2.7 2.05 to 2.65 2.05
\arrow <3 pt> [1,2] from 2.55 1 to 2.5 1
\setdashes  <.4mm,1mm>
\plot 3.5 -1   3.5 4 /
\setsolid
\setquadratic
\plot 2 3  3 2 4 3 /
\plot 1 0 3 1  5 0 /
\endpicture}$
\caption{$(3,2)$-superdiagrams} \label{Fig:example2}
\end{figure}
\end{center}
The $(r,s)$-superdiagram which has even (respectively, odd)
number of marked edges is regarded as \emph{even} (respectively,
\emph{odd}).

Let $\BB$ be the superspace with a basis consisting of
$(r,s)$-superdiagrams. The $(r,s)$-superdiagrams without marked
edges form a basis of the walled Brauer algebra $B_{r,s}(\delta)$
$(\delta \in \C)$ (See, for example, \cite{BCHLLS, BS, CW, N, SM}).

From an $(r+s)$-superdiagram in $\DD_{r+s}$, we obtain an $(r,s)$-superdiagram
by adding the wall between the $r$th and $(r+1)$th vertices and
interchanging the vertices on the top row
with the vertices on the bottom row on the right side of the wall
without disconnecting any of the edges and changing the mark of the edges.
We denote this map by
\begin{align} \label{def:flip between Ser_r+s and B_r,s}
{\rm flip}: \DD_{r+s} \longrightarrow \BB.
\end{align}
For example, the $(3,2)$-superdiagrams $d_1,d_2$ in Figure \ref{Fig:example2} can be obtained
by flipping the following $5$-superdiagrams in $\DD_{5}$, respectively:

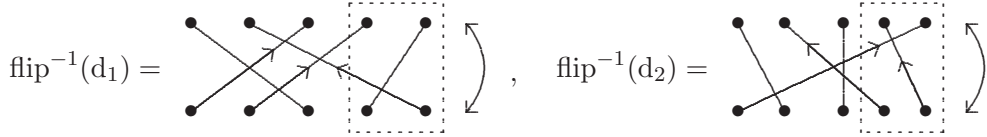
\begin{center}
${\beginpicture
\setcoordinatesystem units <0.78cm,0.39cm>
\setplotarea x from -1.5 to 6, y from -1 to 4
\put{$\rm{flip}^{-1}(d_1) = $} at -0.8 1.5
\put{$\bullet$} at  1 3  \put{$\bullet$} at  1 0
\put{$\bullet$} at  2 3  \put{$\bullet$} at  2 0
\put{$\bullet$} at  3 3  \put{$\bullet$} at  3 0
\put{$\bullet$} at  4 3  \put{$\bullet$} at  4 0
\put{$\bullet$} at  5 3  \put{$\bullet$} at  5 0
\plot 1 3  3 0 /
\plot 2 3  5 0 /
\plot 3 3  1 0 /
\plot 4 3 2 0 /
\plot 5 3 4 0  /
\arrow <3 pt> [1,2] from 1 0 to  2.4 2.1
\arrow <3 pt> [1,2] from 2 0 to 3 1.5
\arrow <3 pt> [1,2] from 5 0 to 3.5 1.5
\put{,} at 6.5 1
\arrow <3 pt> [1,2] from 5.75 2.9 to 5.7 3.
\arrow <3 pt> [1,2] from 5.75 0.1 to 5.7 0
\setdashes  <.4mm,1mm>
\plot 3.7 -0.7   3.7 3.7 /
\plot 3.7 3.7  5.3 3.7 /
\plot 5.3 -0.7 5.3 3.7  /
\plot 5.3 -0.7 3.7 -0.7 /
\setsolid
\setquadratic
\plot 5.7 3 6 1.5 5.7 0 /
\endpicture}$
 \ \ \
 {\beginpicture
\setcoordinatesystem units <0.78cm,0.39cm>
%\setplotarea x from -1.5 to 5.5, y from -0.5 to 3.5
\put{$\rm{flip}^{-1}(d_2) = $} at -0.8 1.5
\put{$\bullet$} at  1 3  \put{$\bullet$} at  1 0
\put{$\bullet$} at  1.8 3  \put{$\bullet$} at  1.8 0
\put{$\bullet$} at  2.8 3  \put{$\bullet$} at  2.8 0
\put{$\bullet$} at  3.5 3  \put{$\bullet$} at  3.5 0
\put{$\bullet$} at  4.2 3 \put{$\bullet$} at  4.2 0
\plot 1 3  1.8 0 /
\plot 1.8 3  3.5 0 /
\plot 2.8 3  2.8 0 /
\plot 3.5 3 4.2 0 /
\plot 4.2 3 1 0 /
\arrow <3 pt> [1,2] from 3.5 0 to  2.2 2.29
\arrow <3 pt> [1,2] from 4.2 0 to 3.8 1.71
\arrow <3 pt> [1,2] from 1 0 to 3.4 2.25

\arrow <3 pt> [1,2] from 4.95 2.9 to 4.9 3.
\arrow <3 pt> [1,2] from 4.95 0.1 to 4.9 0
\setdashes  <.4mm,1mm>
\plot 3.1 -0.7   3.1 3.7 /
\plot 3.1 3.7    4.5 3.7 /
\plot 4.5 -0.7   4.5 3.7  /
\plot 4.5 -0.7   3.1 -0.7 /
\setsolid
\setquadratic
\plot 4.9 3 5.2 1.5 4.9 0 /
\endpicture}
\end{center}
It is clear that the map $\flip$ is a bijection, which implies
$\dim_{\C} \BB=\dim_{\C} \DD_{r+s}=2^{r+s}(r+s)!$.

As we did for $\DD_{k}$ on $\V^{\otimes k}$ in Section 2, we will
define a (right) action of each $(r,s)$-superdiagram in $\BB$ on the
mixed tensor space $\mix$. In this section, we will define a linear
map $\Psi_{r,s}$ from $\BB$ to $\End_{\qn}(\mix)$ and will complete
the job in Section 4 by showing that $\Psi_{r,s}$ is a superalgebra
homomorphism.

Given ${\bf i}=(i_1, \ldots, i_{r}, i_{r+1}, \ldots, i_{r+s}) \in
I^{r+s}$, let ${\bf i}^L:=(i_1, \ldots, i_{r}), \ {\bf i}^R
:=(i_{r+1}, \ldots, i_{r+s})$ and ${\bf i}={\bf i}^L{\bf i}^R$. We
write $v_{\bf{i}}=v_{i_1} \tensor \cdots \tensor v_{i_r}$ and
$w_{\bf{j}}=w_{j_1} \tensor \cdots \tensor w_{j_s}$ for $\ii \in
I^r, \jj \in I^s$. Clearly,
% $\{w_{\bf{j}} \ | \ {\bf j} \in I^s \}$
%is the basis of $\W^{\tensor s}$, and
$\{v_{{\bf i}^L} \tensor w_{{\bf i}^R} \ | \ {\bf i}^L\in I^r, {\bf
i}^R \in I^s  \}$ forms a basis of $\mix$.
 For $\ii, \jj \in I^{r+s}$ and $d \in \BB$, we label
 the vertices on the bottom row of $d$ with
 $i_1 \ldots i_{r+s}$ (from left to right) and the vertices on the top row with $d$ by $j_1 \ldots, j_{r+s}$
 (from left to right). We denote this labeled diagram by $_{\ii} d_{\jj}$.
 We define the \emph{weight} of the labeled diagram $\idj$ to be
 \begin{align*}
   \wtd := \prod_{e} \delta_e \prod_{h}(-1)^{|h|}\prod_c(-1)^{|c|}\prod_{e_i, \widetilde{e}_i}(-1)^{m(e_i)+m(\widetilde{e}_i)},
 \end{align*}
 where the notations are explained below:
 \begin{enumerate}
   \item  The first product is taken over all edges $e$ in $\idj$.
For each normal (vertical or horizontal) edge $e$ with label $i$ and
$j$ at the end points, let $\delta_e=\delta_{i,j}$.
   We define $\delta_e=\delta_{{\ol i},j}$ for each marked vertical edge $e$ with label $i$ at the bottom vertex
   and $j$ at the top vertex,
   or for each marked horizontal edge $e$ with label $i$ at the left vertex and $j$
   at the right vertex.

   \item The second product is taken over all (normal or marked) horizontal edges $h$ on the bottom row in $\idj$.
   For each (normal or marked) horizontal edge $h$ on the bottom row
   with label $j$ at the \emph{right} vertex of the edge,
   we define $|h|=|j|$.

   \item The third product is taken over all the crossings $c$ in $\idj$.
    We define $|c|=|e_1||e_2|$ for each crossing of two different edges $e_1$ and $e_2$, where
   we define $|e_i|=|i_m|$ if the edge $e_i$ is a (normal or marked) vertical edge with label $i_m$
   at the \emph{bottom} vertex, or a (normal or marked) horizontal edge with label $i_m$ at the \emph{right} vertex.

   \item The last product is taken all marked (vertical or horizontal) edges in $\idj$.

   (i) Read the \emph{left} vertices of the marked horizontal edges on the bottom row in order from left to
   right. Then we obtain a sequence $a_1 \cdots a_p$ for $1 \leq a_1 < \cdots <a_p \leq r$. Let each marked
   edge be $e_i$. We define
   $$
   \begin{aligned}
   m(e_i)&=m(e_{i-1})+|i_{a_{i-1}+1}| + \cdots + |i_{a_{i}}|+1,\\
   m(e_1)&=|i_1|+\cdots+|i_{a_1}|+1.
   \end{aligned}
   $$

   (ii) Read the \emph{top} vertices of the marked vertical edges and the \emph{left} vertices of the marked horizontal edges
   on the top row in order from left to right.
   Then we obtain a sequence $b_1 \cdots b_q$ for $1 \leq b_1< \cdots <b_q \leq
   r+s$. Let each marked edge $\widetilde{e}_i$.
   We define ${m(\widetilde{e}_i)}$ as follows: (let $|j_0|=0$)
   \begin{align*}
   &{m(\widetilde{e}_i)}=|j_1| + \cdots + |j_{b_i}| &&\text{ if   } \; b_i \leq r,\\
      &{m(\widetilde{e}_i)}=|j_1| +\cdots +|j_{(b_i)-1}| &&\text{ if    } \; b_i > r.
\end{align*}
 \end{enumerate}
If $\wt(\idj)\neq 0$, we say that $\idj$ is \emph{consistently
labeled}.

\begin{example}
The labeled diagram $\idj$ of $d$ in Figure 3 has
$$\wt(\idj)=\delta_{\ol i_1,i_3} \delta_{\ol i_2,j_2} \delta_{\ol i_4,j_3} \delta_{j_1, j_4} (-1)^{|i_3|}
(-1)^{|i_2||i_3|+|i_2||j_4|+|i_4||j_4|} (-1)^{(|i_1|+1) +(|j_1|+|j_2|)+(|j_1|+|j_2|)}.$$

\begin{figure}[!h]
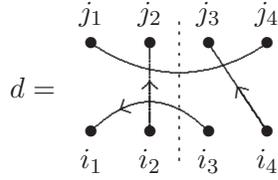

$${\beginpicture
\setcoordinatesystem units <0.78cm,0.39cm>
\setplotarea x from 0 to 4.5, y from -1.5 to 5
\put{$d= $} at 0 1.5
\put{$\bullet$} at  1 0  \put{$\bullet$} at  1 3
\put{$\bullet$} at  2 0  \put{$\bullet$} at  2 3
\put{$\bullet$} at  3 0  \put{$\bullet$} at  3 3
\put{$\bullet$} at  4 0  \put{$\bullet$} at  4 3

\put{$i_1$} at 1 -1
\put{$i_2$} at 2 -1
\put{$i_3$} at 3 -1
\put{$i_4$} at 4 -1
\put{$j_1$} at 1 4
\put{$j_2$} at 2 4
\put{$j_3$} at 3 4
\put{$j_4$} at 4 4

\plot 2 3 2 0 /
\plot 3 3 4 0 /
\arrow <3 pt> [1,2] from 2 0 to  2 1.7
\arrow <3 pt> [1,2] from 1.55 0.8 to  1.5 0.75
\arrow <3 pt> [1,2] from 4 0 to 3.5 1.5
\setdashes  <.4mm,1mm>
\plot 2.5 -1   2.5 4 /
\setsolid
\setquadratic
\plot 1 3 2.5 2 4 3 /
\plot 1 0 2 1 3 0  /
\endpicture}$$
\caption{$(2,2)$-superdiagram $d$} \label{ex:wt of d}
\end{figure}
\end{example}

The above procedure provides a linear map $\Psi_{r,s}$ from $\BB$
into $\End_{\C}(\mix)$:
\begin{eqnarray} \label{def:action of B_r,s}
  \Psi_{r,s}:   \BB & \longrightarrow &  \End_{\C}(\mix)\\
 \nonumber              d &\longmapsto         &  \left( v_{\iiL} \tensor w_{\iiR} \mapsto
        \sum_{\jj \in I^{r+s}} \wt(\idj) \ v_{\jjL} \tensor w_{\jjR} \right).
\end{eqnarray}
Note that the linear map $\Psi_{r,s}$ corresponding to the
$(r,s)$-superdiagram without marked edges is the same as the action
of the walled Brauer diagram given in \cite[Section 7]{BS}.

Define $c_i$ to be the following $(r,s)$-superdiagram:

\begin{center}
    \begin{equation} \label{def:a superdiagram c_i-1}
   %\begin{aligned}
    %&
    {\beginpicture
\setcoordinatesystem units <0.78cm,0.39cm>
\setplotarea x from -0.5 to 12, y from -1.5 to 4.5
\put{$c_i := $} at 0 1.5
\put{$\bullet$} at  1 0  \put{$\bullet$} at  1 3
%\put{$\bullet$} at  2 0  \put{$\bullet$} at  2 3
\put{$\bullet$} at  3 0  \put{$\bullet$} at  3 3
\put{$\bullet$} at  4 0  \put{$\bullet$} at  4 3
\put{$\bullet$} at  5 0  \put{$\bullet$} at  5 3
\put{$\bullet$} at  7 0  \put{$\bullet$} at  7 3
\put{$\bullet$} at  8 0  \put{$\bullet$} at  8 3
\put{$\bullet$} at  10 0  \put{$\bullet$} at  10 3
%\put{$\bullet$} at  11 0  \put{$\bullet$} at  11 3
\put{$\cdots$} at 2 1.5
\put{$\cdots$} at 6 1.5 \put{$\cdots$} at 9 1.5
\put{\small{ $1$}} at 1 4
\put{\small{$i-1$}} at 3 4
\put{\small{$i$}} at 4 4
\put{\small{$i+1$}} at 5 4
\put{\small{$r+s$}} at 10 4
\put{\small{$r$}} at 7 4
%\put{$r+1$} at 8 4
\put{if $ \ i \leq r$,} at 12.5 1.5
\plot 1 3 1 0 /
\plot 3 3 3 0  /
\plot 4 3 4 0 /
\plot 5 3 5 0 /
\plot 7 3 7 0 /
\plot 8 3 8 0 /
\plot 10 3 10 0 /
\arrow <3 pt> [1,2] from 4 0 to  4 1.7
\setdashes  <.4mm,1mm>
\plot 7.5 -1   7.5 4 /
\endpicture}
\end{equation}
%\\ \ \ \   %&
\begin{equation} \label{def:a superdiagram c_i-2}
{\beginpicture \setcoordinatesystem units <0.78cm,0.39cm>
\setplotarea x from -0.5 to 13, y from -1.5 to 4.5 \put{$c_i := $}
at 0 1.5 \put{$\bullet$} at  1 0  \put{$\bullet$} at  1 3
%\put{$\bullet$} at  2 0  \put{$\bullet$} at  2 3
\put{$\bullet$} at  3 0  \put{$\bullet$} at  3 3
\put{$\bullet$} at  4 0  \put{$\bullet$} at  4 3
\put{$\bullet$} at  6 0  \put{$\bullet$} at  6 3
\put{$\bullet$} at  7 0  \put{$\bullet$} at  7 3
\put{$\bullet$} at  8 0  \put{$\bullet$} at  8 3
\put{$\bullet$} at  10 0  \put{$\bullet$} at  10 3
%\put{$\bullet$} at  11 0  \put{$\bullet$} at  11 3
\put{$\cdots$} at 2 1.5
\put{$\cdots$} at 5 1.5 \put{$\cdots$} at 9 1.5
\put{\small{$1$}} at 1 4
\put{\small{$i-1$}} at 6 4
\put{\small{$i$}} at 7 4
\put{\small{$i+1$}} at 8 4
\put{\small{$r+s$}} at 10 4
\put{\small{$r$}} at 3 4
%\put{$r+1$} at 4 4
\put{if $ \ i \geq r+1$.} at 13 1.5
\plot 1 3 1 0 /
\plot 3 3 3 0  /
\plot 4 3 4 0 /
\plot 6 3 6 0 /
\plot 7 3 7 0 /
\plot 8 3 8 0 /
\plot 10 3 10 0 /
\arrow <3 pt> [1,2] from 7 0 to  7 1.7
\setdashes  <.4mm,1mm>
\plot 3.5 -1   3.5 4 /
\endpicture} %\end{aligned}
\end{equation}
\end{center}

We decompose the $(2,2)$-superdiagram $d$ in Figure \ref{ex:wt of d}
into the three parts as follows:

$$
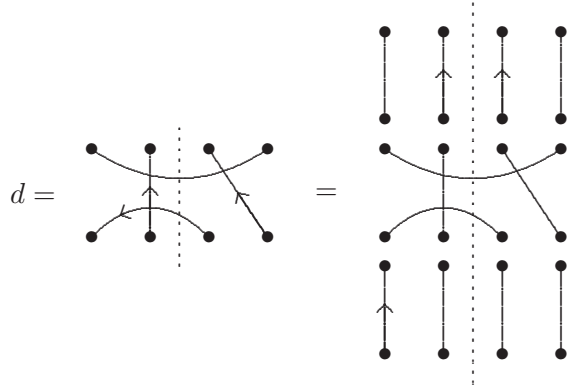
{\beginpicture
\setcoordinatesystem units <0.78cm,0.39cm>
%\setplotarea x from 0 to 6, y from -2 to 3
\put{$d = $} at 0 1.5
\put{$\bullet$} at  1 0  \put{$\bullet$} at  1 3
\put{$\bullet$} at  2 0  \put{$\bullet$} at  2 3
\put{$\bullet$} at  3 0  \put{$\bullet$} at  3 3
\put{$\bullet$} at  4 0  \put{$\bullet$} at  4 3
\plot 2 3 2 0 /
\plot 3 3 4 0 /
\arrow <3 pt> [1,2] from 2 0 to  2 1.7
\arrow <3 pt> [1,2] from 1.55 0.8 to  1.5 0.75
\arrow <3 pt> [1,2] from 4 0 to 3.5 1.5
\setdashes  <.4mm,1mm>
\plot 2.5 -1   2.5 4 /
\setsolid

\put{$=$} at 5 1.5

\put{$\bullet$} at  6 0  \put{$\bullet$} at  6 3
\put{$\bullet$} at  7 0  \put{$\bullet$} at  7 3
\put{$\bullet$} at  8 0  \put{$\bullet$} at  8 3
\put{$\bullet$} at  9 0  \put{$\bullet$} at  9 3
\plot 7 3 7 0 /
\plot 8 3 9 0 /
%\arrow <3 pt> [1,2] from 2 0 to  2 1.7
%\arrow <3 pt> [1,2] from 1.45 0.75 to  1.4 0.75
%\arrow <3 pt> [1,2] from 4 0 to 3.5 1.5

\put{$\bullet$} at  6 4  \put{$\bullet$} at  6 7
\put{$\bullet$} at  7 4  \put{$\bullet$} at  7 7
\put{$\bullet$} at  8 4  \put{$\bullet$} at  8 7
\put{$\bullet$} at  9 4  \put{$\bullet$} at  9 7
\plot 6 4 6 7 /
\plot 9 4 9 7 /
\plot 7 4 7 7 /
\plot 8 4 8 7 /
\arrow <3 pt> [1,2] from 8 4 to 8 5.7
\arrow <3 pt> [1,2] from 7 4 to 7 5.7

\put{$\bullet$} at  6 -1  \put{$\bullet$} at  6 -4
\put{$\bullet$} at  7 -1  \put{$\bullet$} at  7 -4
\put{$\bullet$} at  8 -1  \put{$\bullet$} at  8 -4
\put{$\bullet$} at  9 -1  \put{$\bullet$} at  9 -4
\plot 6 -1 6 -4 /
\plot 7 -1 7 -4 /
\plot 8 -1 8 -4 /
\plot 9 -1 9 -4 /
\arrow <3 pt> [1,2] from 6 -4 to 6 -2.3

\setdashes  <.4mm,1mm>
\plot 7.5 -5   7.5 8 /
\setsolid

\setquadratic
\plot 6 3 7.5 2 9 3 /
\plot 6 0 7 1 8 0  /
\plot 1 3 2.5 2 4 3 /
\plot 1 0 2 1 3 0  /
\endpicture}$$

\vskip1em Let $d'$ be the $(2,2)$-superdiagram obtained from $d$ by
forgetting the marks on $d$. By a direct calculation, we observe
$$\Psi_{2,2}(d)=\Psi_{2,2}(c_1)\Psi_{2,2}(d')\Psi_{2,2}(c_2)
\Psi_{2,2}(c_3).$$ The general case is stated in the following
lemma, which can be proved by a similar calculation.

\begin{lemma} \label{lem:decomposition of a superdiagram}
{\rm  Let $d$ be an $(r,s)$-superdiagram with the sequences $a_1
\cdots a_p$ and $b_1 \cdots b_q$ which are given in (4) (i), (ii).
% Let $d
%\in \BB$ have the same indices as (4)(i),(ii) when we calculate
%\$wt(\idj)$.
If $d'$ is the $(r,s)$-superdiagram obtained from $d$ by forgetting
the marks on $d$,
   then we have
  $$\Psi_{r,s}(d)= \Psi_{r,s}(c_{a_1}) \cdots \Psi_{r,s}(c_{a_p}) \Psi_{r,s} (d') \Psi_{r,s}(c_{b_1}) \cdots
  \Psi_{r,s}(c_{b_q}) \in \End_{\C}(\mix)^{\op}. $$}
\end{lemma}

Since $\Psi_{r,s}(d')$ is even and $\Psi_{r,s}(c_j)$ is odd, we
conclude that the map $\Psi_{r,s}$ is an even linear map. Moreover,
we obtain

\begin{prop} \label{lem:supercommute between B_r,s and q(n)}
{\rm The linear maps $\Psi_{r,s}(d)$ and $\rho(g)$ on $\mix$
supercommute with each other for all $d \in \BB$ and $g \in \qn$.
That is, the image of $\Psi_{r,s}$ is contained in
$\End_{\qn}(\mix)$. }
%The action $\Psi_{r,s}$ of $d \in \BB$ and the action $\rho$ of $g
%\in \qn$ on $\mix$
% supercommute with each other: i.e., the image of the map $\Psi_{r,s}$ is in $\End_{\qn}(\mix)$.
\end{prop}
\begin{proof}
  By \cite[Lemma 7.4]{BS}, we know that the action of the $(r,s)$-superdiagram
  without the marked edges commutes with the action of $\qn$.
  To prove the general case, by Lemma \ref{lem:decomposition of a superdiagram}, it is enough to show that
   \begin{equation} \label{eq:supercommute}
   \rho(g) \Psi_{r,s}(c_j)=(-1)^{|g|}\Psi_{r,s}(c_j) \rho(g)
   \end{equation}
      for all homogeneous elements $g \in \qn$.

%Define $\overline{v_i}:=v_{\overline i}$ and $\overline{w_i}:=w_{\overline i}$ for
%all basis elements $\V$ and $\W$ and expand it linearly.

Define the bar involution on $\V$ (resp. on $\W$) by $\overline{v_i}
= v_{\overline{i}}$ (resp. $\overline{w_i} = w_{\overline{i}}$) for
all $i\in I$. Using \eqref{eq:action of q(n)-elements on V,W}, it is
easy to check that
  \begin{align} \label{eq:equality for supercommute}
    \overline{e_{i,j}v_k}=e_{i,j}v_{\ol k},  \ \ \overline{e_{i,j}w_k}=(-1)^{|j|}e_{i,j}w_{\overline{k}}
  \end{align}
for all $i \in \{ 1, \ldots, n\}$ and $j, k \in I$. Here, we use the
notation $\overline{\overline i} =i$ for $i=1, \ldots, n$ and
$\delta_{i,j}=\delta_{\ol i,\ol j}$, $\delta_{\ol i, j}=\delta_{i,
\ol j}$ for $i, j \in I$.

   Suppose $r+1 \leq j \leq r+s$ and  we will compare the both sides of \eqref{eq:supercommute}.
   Set  $$v_{\iiL}=\iiL \ \  \text{and} \ \ w_{\iiR} =\iiR \quad
   \text{for} \ \
   \ii \in I^{r+s}.$$ Then the left-hand side of
   \eqref{eq:supercommute} is equal to
   \begin{align*}
    & \rho(g) \left(\Psi_{r,s}(c_j) (\iiL \tensor \iiR) \right) \\
    & = \rho(g) \big( (-1)^{|i_1| + \cdots +|i_{j-1}|}
     \ i_1 \tensor \cdots \tensor i_{j-1} \tensor \ol{i_j} \tensor i_{j+1} \tensor  \cdots  \tensor i_{r+s} \big)\\
   &=\sum_{k < j} (-1)^{|i_1| + \cdots +|i_{j-1}| +|g|(|i_1|+\cdots+|i_{k-1}|)}
   \ i_1 \tensor \cdots \tensor g \cdot i_k \tensor \cdots \tensor \ol{i_j} \tensor \cdots \tensor i_{r+s}\\
   &  \ \ + (-1)^{|i_1| + \cdots +|i_{j-1}| +|g|(|i_1|+\cdots+|i_{j-1}|)}
   \  i_1 \tensor \cdots \tensor g \cdot {\overline{i_j}} \tensor \cdots \tensor i_{r+s}\\
   & \ \ +\sum_{k>j} (-1)^{|i_1| + \cdots +|i_{j-1}| +|g|(|i_1|+\cdots+ (|i_j|+1)+ \cdots +|i_{k-1}|) }
    \ i_1 \tensor \cdots \tensor \ol{i_j} \tensor  \cdots \tensor g \cdot i_k \tensor \cdots \tensor i_{r+s}.
       \end{align*}
   On the other hand, the right-hand side of \eqref{eq:supercommute}
   is the same as
  \begin{align*}
     &\Psi_{r,s}(c_j) \left(\rho(g) (\iiL \tensor \iiR) \right)\\
     &=\Psi_{r,s}(c_j) \big(\sum_{k=1}^{r+s} (-1)^{|g|(|i_1|+\cdots+|i_{k-1}|)}
       \ i_1 \tensor \cdots \tensor i_{k-1} \tensor g \cdot i_k \tensor i_{k+1} \tensor \cdots \tensor i_{r+s} \big)\\
   &=\sum_{k < j} (-1)^{|g|(|i_1|+\cdots+|i_{k-1}|) +|i_1| + \cdots+(|g|+|i_k|)+ \cdots +|i_{j-1}|}
    \  i_1 \tensor \cdots \tensor g \cdot i_k \tensor \cdots \tensor \ol{i_j} \tensor \cdots \tensor i_{r+s}\\
   & \ \ +(-1)^{|g|(|i_1|+\cdots+|i_{j-1}|) +|i_1| + \cdots +|i_{j-1}|}
     \  i_1 \tensor \cdots \tensor \overline{g \cdot i_j} \tensor \cdots \tensor i_{r+s} \\
   & \ \ +\sum_{k>j} (-1)^{|g|(|i_1|+\cdots+|i_{k-1}|) +|i_1| + \cdots +|i_{j-1}|}
      \ i_1 \tensor \cdots \tensor \ol{i_j} \tensor \cdots \tensor g \cdot i_k \tensor \cdots \tensor i_{r+s} .
      \end{align*}
   Now our assertion follows from \eqref{eq:equality for supercommute}.

   The other case $1 \leq j \leq r$ can be verified in a similar manner.
\end{proof}

We will show that the map $\Psi_{r,s}: \BB \rightarrow
\End_{\qn}(\mix)$ is a linear isomorphism whenever $n \ge r+s$. Our
argument will follow the outline given in \cite[Theorem 3.4]{N} and
\cite[Theorem 7.8]{BS}, in which the case of $\mathfrak{gl}(m)$ and
$\mathfrak{gl}(m|n)$ were treated, respectively.
%To show that the
%linear map $\Psi_{r,s}:\BB \longrightarrow \End_{\qn}(\mix)$ is a
%bijection whenever $n \geq {r+s}$, we use the similar manner as the
%proof of \cite[Theorem 3.4]{N},which treats $\mathfrak{gl}(m)$, and
%\cite[Theorem 7.8]{BS}, which treats $\mathfrak{gl}(m|n)$.

Let $M,N,K$ and $L$ be finite dimensional $\qn$-supermodules. We
define the action of $\qn$ on $\Hom_{\C}(M,N)$ by
$$(gf)(m):=g(f(m))-(-1)^{|g||f|}f(gm)$$
for homogeneous elements $f \in \Hom_{\C}(M,N)$, $g \in \qn$ and $m
\in M$. Note that a linear map $f:M \rightarrow N$ is a
$\qn$-supermodule homomorphism if and only if  $f$ is annihilated by
all $g \in \qn$. We define $f^* \in \Hom_{\C}(N^*, M^*)$ by
$$(f^*(\lambda))(m):=(-1)^{|f||\lambda|} \lambda(f(m))$$
for homogeneous element $ \lambda \in N^*$ and all $m \in M$.
We identify
$$\Hom_{\C}(M,N) \tensor \Hom_{\C}(K,L)=\Hom_{\C}(M \tensor K, N \tensor L) $$
by
 \begin{align} \label{eq:identification}
   (f \tensor g)(m \tensor k)=(-1)^{|g||m|} f(m) \tensor g(k)
 \end{align}
for $f \in \Hom_{\C}(M,N), k \in K$ and homogeneous elements $g \in
\Hom_{\C}(K,L), m \in M$.

By \cite[Lemma 7.4]{BS}, one can deduce that the map ${\it flip}$ is a
$\qn$-supermodule isomorphism:
\begin{equation}
\begin{aligned}
   {\it flip}: & \ \End_{\C}(\Vrs)& &\longrightarrow &&\End_{\C}(\mix)\\
           &  \ \ \  \ f\tensor g&    &\longmapsto &&  \ \  \ \ f \tensor g^*,
\end{aligned}
\end{equation}
where $f \in \End_{\C}(\Vr)$ and $g \in \End_{\C}(\V^{\tensor s})$.

Recall the map ${\rm flip}: \DD_{r+s} \longrightarrow \BB$ given in
\eqref{def:flip between Ser_r+s and B_r,s}.  For a given
$(r+s)$-superdiagram $d \in \DD_{r+s}$, let $\ell$ (respectively,
$m$) be the number of the marked vertical edges on the left-hand
side (respectively, right-hand side)  of the wall in $\flip(d)$. Let
$x$ (respectively, $y$) be the number of the marked horizontal edges
on the top row (respectively, the bottom row) in $\flip(d)$.

We define
\begin{equation*}
\begin{aligned}
{\rm Flip} \ :  \ &  \DD_{r+s} &\longrightarrow & \ \ \ \ \ \ \ \ \ \ \BB \\
                  & \ \ d  & \longmapsto & \ (-1)^{u(d)+ (\ell+x)y}{\rm flip}(d),
\end{aligned}
\end{equation*}
where $u(d)$ is the number of crossings that are involved in  the
marked edges
 in $d$ such that the top vertices are between the $(r+1)$th vertex and
 $(r+s)$th vertex.

 \nc{\irh}{\ii^{RH}}
 \nc{\jrh}{\jj^{RH}}
 \nc{\irv}{\ii^{RV}}
 \nc{\jrv}{\jj^{RV}}
 \nc{\ilh}{\ii^{LH}}
 \nc{\jlh}{\jj^{LH}}
 \nc{\ilv}{\ii^{LV}}
 \nc{\jlv}{\jj^{LV}}

\begin{lemma} \label{lem:Lemma}
 {\rm The following diagram is commutative:
  \begin{equation} \label{diagram}
  \xymatrix
  {\ar @{} [dr] |{\circlearrowleft} {\DD_{r+s}} \ar[r]^{\Flip} \ar[d]^{\tPhi_{r+s} } & {\BB} \ar[d]^{\Psi_{r,s}}\\
{\End_{\C}(\Vrs)} \ar[r]^(0.45){\it flip} &{\End_{\C}(\mix)} }
\end{equation}
}
\end{lemma}

\begin{proof}
  For a superdiagram without the marked edges, it was proved in
  \cite[Lemma 7.7]{BS}.

  For $\ii,\jj, \kk \in I^r$, we define $e_{\ii,\jj} \in \End_{\C}(\Vr)$
  by $e_{\ii,\jj}(v_{{\mathbf k}})=\delta_{\jj,\kk }v_{\ii}$ and for
  $\ii,\jj,\kk \in I^s$, we define $f_{\ii,\jj}\in \Endc(\Ws)$
  by $f_{\ii,\jj}(w_{\kk})=\delta_{\jj,\kk}w_{\ii}$.
  For $\ii \in I^t$, let
  $$|\ii|=|i_1| +\cdots + |i_t|, \ \ p(\ii)=\sum_{1 \leq a< b \leq t} |i_a||i_b| \ \ (\text{if   }t \geq 2),
  \ \ p(\ii)=0 \ \ (\text{if   }  t=1).$$
  As in the proof of \cite[Lemma 7.7]{BS}, we can check that
  \begin{align} \label{eq:dual element}
w_{\ii}(v_{\jj})=(-1)^{p(\ii)} \delta_{\ii,\jj}, \ \
  (e_{\ii, \jj})^*=(-1)^{(|\ii|+|\jj|)|\ii|+p(\ii)+p(\jj)} f_{\jj,\ii}.
  \end{align}
  From the identification \eqref{eq:identification}, we have
  \begin{align*}
  \tPhi_{r+s}(d)=\sum_{\ii,\jj \in I^{r+s}} (-1)^{(|\iiR|+|\jjR|)|\iiL|}\wt(\idj) e_{\jjL,\iiL} \tensor e_{\jjR,\iiR}
  \end{align*}
  and
  $$\Psi_{r,s}(\Flip(d))=(-1)^{u(d)+(\ell+x)y}
  \sum_{\ii,\jj \in I^{r+s}} (-1)^{(|\iiR|+|\jjR|)|\iiL|}
  \wt(_{\ii} {\widetilde d}_{\jj}) e_{\jjL,\iiL} \tensor f_{\jjR,\iiR} $$
   for $ d \in \DD_{r+s}$ and $\widetilde{d}:=\flip(d)$.
If we interchange the variables $\iiR$ and $\jjR$ and apply the
${\it flip}$ map, we have
 \begin{align*}
  {\it flip}(\tPhi_{r+s}(d))=\sum_{\ii,\jj \in I^{r+s}} (-1)^{(|\iiR|+|\jjR|)|\ii| +p(\iiR) +p(\jjR) }
  \wt(_{\iiL \jjR} d_{\jjL \iiR}) e_{\jjL,\iiL} \tensor f_{\jjR,\iiR}.
  \end{align*}
Therefore, it is enough to show that
\begin{align} \label{eq:flip and weight}
  (-1)^{(|\iiR|+|\jjR|)|\iiR| +p(\iiR) +p(\jjR) } \wt(_{\iiL \jjR} d_{\jjL \iiR})
  = (-1)^{u(d)+(\ell+x)y}  \wt(_{\ii} {\widetilde d}_{\jj}).
\end{align}
  Clearly, the labeled diagram $_{\iiL \jjR} d_{\jjL \iiR}$ is consistently labeled
  if and only if $_{\ii} {\widetilde d}_{\jj}$ is consistently labeled.
  So we may assume $\wt(_{\iiL \jjR} d_{\jjL \iiR})$ is not zero.

\vskip 3mm

Let
$${\beginpicture \setcoordinatesystem units <0.78cm,0.39cm>
%\setplotarea x from -1 to 11, y from -1.7 to 3.7
\put{$c_i := $} at 0 0.4
\put{$\bullet$} at  1 -1  \put{$\bullet$} at  1 2
%\put{$\bullet$} at  2 0  \put{$\bullet$} at  2 3
\put{$\bullet$} at  3 -1  \put{$\bullet$} at  3 2
\put{$\bullet$} at  4 -1  \put{$\bullet$} at  4 2
\put{$\bullet$} at  5 -1  \put{$\bullet$} at  5 2
\put{$\bullet$} at  7 -1  \put{$\bullet$} at  7 2
\put{$\cdots$} at 2 0.5
\put{$\cdots$} at 6 0.5
\put{\small{$1$}} at 1 3
\put{\small{$i-1$}} at 3 3
\put{\small{$i$}} at 4 3
\put{\small{$i+1$}} at 5 3
\put{\small{$r+s$}} at 7 3

\put {$\in \DD_{r+s}$ for $1 \leq i \leq r+s$.} at 10.5 0.8

\plot 1 2 1 -1 /
\plot 3 2 3 -1  /
\plot 4 2 4 -1 /
\plot 5 2 5 -1 /
\plot 7 2 7 -1 /
\arrow <3 pt> [1,2] from 4 -1 to  4 0.7
\endpicture}$$

\vskip 3mm

\noindent If we forget the marks on $d$, we get a permutation
$\sigma$. Note that
$$d=\sigma c_{b_1} \cdots c_{b_{\ell+x}} c_{a_1}
\cdots c_{a_{y+m}},$$ where $b_i$, $a_i$ are the numbers obtained by
reading the top vertices of marked edges in $d$ in order from left
to right. The condition on the marked edges of $\widetilde{d}$
implies
$$ 1 \leq b_1< \cdots <b_{\ell+x} \leq r, \ r+1 \leq a_1 <
\cdots < a_{y+m} \leq r+s.$$
 By the relations in $\DD_{r+s}$, we obtain
  \begin{equation} \begin{aligned} \label{eq:deformation of d}
 d &= (-1)^{(\ell+x)(y+m)}\sigma c_{a_1} \cdots c_{a_{y+m}} c_{b_1} \cdots c_{b_{\ell+x}} \\
  &=(-1)^{(\ell+x)(y+m)} c_{\sigma{(a_1)}} \cdots c_{\sigma{(a_{y+m})}} \sigma c_{b_1} \cdots c_{b_{\ell+x}}\\
  &= (-1)^{(\ell+x)(y+m)+u(d)} c_{a'_1} \cdots c_{a'_{y}} c_{a''_1} \cdots c_{a''_m}
   \sigma c_{b_1} \cdots c_{b_{\ell+x}},
   \end{aligned}\end{equation}
   where $1 \leq a'_1 < \cdots < a'_y \leq r < a''_1 < \cdots <a''_m \leq r+s$.
   We rearrange the indices of $c_{\sigma{(a_1)}} \cdots c_{\sigma{(a_{y+m})}}$
 in order from the smallest one to the biggest one (from left to right).

From Proposition \ref{prop:duality between D_k and qn} (1), we know
$\tPhi$ preserves the multiplication on $\DD_{r+s}$. Since all
elements $\sigma, c_j$ have only vertical edges and $_{\iiL \jjR}
d_{\jjL \iiR}$ is consistently labeled, we obtain
\begin{equation}  \label{eq:modified weight}
\begin{aligned}
 \wt(_{\iiL \jjR} d_{\jjL \iiR})  =  &(-1)^{(\ell+x)(y+m)+u(d)}
    \wt(_{\iiL \jjR} c_{a'_1}  {_{\kk^1}}) \cdots  \wt(_{\kk^{y+m-1}}c_{a''_{m}}  {_{\kk^{y+m}}}) \\
     & \times \wt(_{\kk^{y+m}}\sigma_{{\widetilde{\kk}}^1})
     \wt(  {_{{\widetilde{\kk}}^1}} c_{b_1} {_{{\widetilde{\kk}}^2}})
     \cdots\wt({_{{\widetilde{\kk}}^{\ell+x}}} c_{b_{\ell+x}}  {_{\jjL \iiR}} ),
 \end{aligned}
 \end{equation}
 where $\kk^j$ (respectively, $\widetilde{\kk}^j$) $\in I^{r+s}$ is determined by $\kk^{j-1}$
 (respectively, $\widetilde{\kk}^{j-1}$), $\kk^0=_{\iiL \jjR}$, and
$\widetilde{\kk}^{0}=\kk^{y+m}$. Therefore, to prove \eqref{eq:flip
and weight}, we need to calculate the right-hand side of
\eqref{eq:modified weight}.

\vskip1em
(I) First, we calculate $(|\iiR|+|\jjR|)|\iiR| +p(\iiR) +p(\jjR) $.

Consider the partition of the set $\{ r+1, \ldots, r+s \} = \{ p_1<
\cdots < p_{k_1} \} \sqcup \{q_1 < \cdots <q_{k_2} \}$ such that the
$p_i$-th vertices are the right vertices of horizontal edges on the
top row in $\widetilde{d}$, and $q_i$-th vertices are the top
vertices of vertical edges on the right-hand side of the wall in
$\widetilde{d}$. Set $\jrh:=(j_{p_1}, \cdots, j_{p_{k_1}}),
\jrv:=(j_{q_1}, \cdots, j_{q_{k_2}})$ for
$_{\ii}\widetilde{d}_{\jj}$. Similarly, we define $\irh$ and $\irv$
such that $\irh$ list the labels of the right vertices of all
horizontal edges on the bottom row and $\irv$ list the labels of the
bottom vertices of all vertical edges on the right side of the wall
for $_{\ii}\widetilde{d}_{\jj}$. Then we obtain
$$|\jjR|=|\jrh|+|\jrv|, \ \  p(\jjR)=p(\jrh)+|\jrh||\jrv|+p(\jrv), $$
$$|\iiR|=|\irh|+|\irv|, \ \ p(\iiR)=p(\irh)+|\irh||\irv|+p(\irv). $$

Since the $a''_k$-th vertex is the bottom vertex of the marked
vertical edge in $d$ for all $1 \leq k \leq m$, we note $\{ a''_1,
\ldots, a''_m \} \subseteq \{q_1 , \cdots , q_{k_2}\} $ and
$j_{a''_k}=\ol{i_{b''_k}}$, where $b''_k$-th vertex is the top
vertex of the marked edge connected with $a''_k$-th vertex on the
bottom in $d$. Thus  we obtain the following equalities in $\Z_2$:
\begin{align*} \label{eq:parity and weight}
|\irv|=|\jrv|+m, \ p(\irv)=p(\jrv)+m|\jrv|-(|j_{a''_1}|+ \cdots +|j_{a''_m}|) +  \binom{m}{2},
\end{align*}
where $\dbinom{a}{b}=\dfrac{a!}{(a-b)!b!}$ and let $\dbinom{0}{2}=\dbinom{1}{2}=0$.
From the above formulae, we obtain the following equality in $\Z_2$:
\begin{equation} \label{eq:x-1}
\begin{aligned}
&(|\iiR|+|\jjR|)|\iiR|  +p(\iiR) +p(\jjR)\\
&=|\irh|+p(\irh)+p(\jrh) +|\irh||\jrh| \\
& \ \ + m|\jrh|+m|\irh| -(|j_{a''_1}|+ \cdots +|j_{a''_m}| ) +\binom{m}{2}+m.
\end{aligned}
 \end{equation}

\vskip1em
(II) Secondly, we calculate the crossing part.

 An edge which does not cross the wall in $\td$ has the same crossings in $d$ and $\td$.
 Observe that there are only crossing parts in
$\wt(_{\kk^{y+m}}\sigma_{{\widetilde{\kk}}^1})$ and
$$\kk^{y+m}=(i_1, \ldots,\ol{i_{a'_1}}, \ldots , \ol{i_{a'_2}},
\ldots, \ol{j_{a''_m}}, \ldots, j_{r+s} )$$ in \eqref{eq:modified
weight}. For example, consider the three crossings of the edge with
labels $i_2$ and $j_3$ in the left diagram given below. We observe
$\prod_{c}(-1)^{|c|}$ associated with the edge with labels $i_2$
and $j_3$ is $(-1)^{|i_2||i_3| + |i_2|(|i_1|+1)+|i_2||j_4|}$. In
$\td$, we obtain $\prod_{c}(-1)^{|c|}$ associated with the edge
with labels $i_2$ and $j_3$ is $(-1)^{|i_2||i_3| +
|i_2||i_4|+|i_2||j_4|}$. Since $_{\iiL\jjR}d_{\jjL\iiR}$ is
consistently labeled, they are the same.

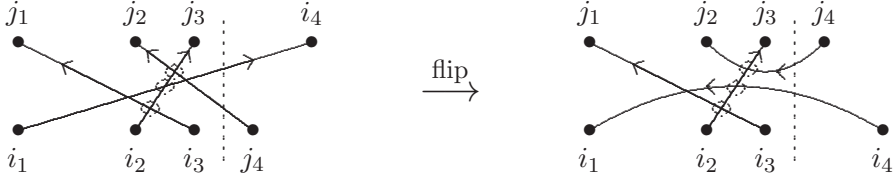
\begin{center}
${\beginpicture
\setcoordinatesystem units <0.78cm,0.39cm>
%\setplotarea x from 0 to 6, y from -2 to 3
%\put{$d_1 = $} at 0 1.5
\put{$\bullet$} at  0 -1  \put{$\bullet$} at  0 2
\put{$\bullet$} at  2 -1  \put{$\bullet$} at  2 2
\put{$\bullet$} at  3 -1  \put{$\bullet$} at  3 2
\put{$\bullet$} at  4 -1  \put{$\bullet$} at  5 2
\plot 5 2 0 -1 /
\plot 3 2 2 -1 /
\plot 0 2 3 -1 /
\plot 2 2 4 -1 /
\arrow <3 pt> [1,2] from 3 -1 to  0.75 1.25
\arrow <3 pt> [1,2] from 0 -1 to  4 1.4
\arrow <3 pt> [1,2] from 4 -1 to  2.2 1.7
\arrow <3 pt> [1,2] from 2 -1 to  2.9 1.7

\setdashes  <.4mm,1mm>
\plot 3.5 -2   3.5 3 /
\setsolid
\setquadratic
\put{$i_1$} at 0 -2
\put{$i_2$} at 2 -2
\put{$i_3$} at 3 -2
\put{$j_4$} at 4 -2
\put{$j_1$} at 0 3
\put{$j_2$} at 2 3
\put{$j_3$} at 3 3
\put{$i_4$} at 5 3

\setdashes  <.4mm,.4mm>
\setquadratic
\plot 2.66  1.25  2.81  1  2.66 0.75 /
\plot 2.66  1.25  2.51  1  2.66 0.75 /

\plot 2.5  .75  2.65  0.5  2.5 0.25 /
\plot 2.5  .75  2.35  0.5  2.5 0.25 /

\plot 2.25  0  2.4  -0.25  2.25 -0.5 /
\plot 2.25  0  2.1  -0.25  2.25 -0.5 /
\setsolid
\endpicture}$
\hskip3em {\Large $\stackrel{{\rm flip}}{\longrightarrow}$} \hskip3em
${\beginpicture
\setcoordinatesystem units <0.78cm,0.39cm>
%\setplotarea x from 0 to 6, y from -2 to 3
%\put{$d_1 = $} at 0 1.5
\put{$\bullet$} at  0 -1  \put{$\bullet$} at  0 2
\put{$\bullet$} at  2 -1  \put{$\bullet$} at  2 2
\put{$\bullet$} at  3 -1  \put{$\bullet$} at  3 2
\put{$\bullet$} at  5 -1  \put{$\bullet$} at  4 2
\plot 3 2 2 -1 /
\plot 0 2 3 -1 /

\arrow <3 pt> [1,2] from 3 -1 to  0.75 1.25
\arrow <3 pt> [1,2] from 3.25 1 to  3.2 1
%\arrow <3 pt> [1,2] from 0.75 -0.1 to  .7 -0.15
\arrow <3 pt> [1,2] from 2 -1 to  2.9 1.7
\arrow <3 pt> [1,2] from 2.05 0.5 to  2 0.5

\setdashes  <.4mm,1mm>
\plot 3.5 -2   3.5 3 /
\setsolid
\setquadratic
\put{$i_1$} at 0 -2
\put{$i_2$} at 2 -2
\put{$i_3$} at 3 -2
\put{$i_4$} at 5 -2
\put{$j_1$} at 0 3
\put{$j_2$} at 2 3
\put{$j_3$} at 3 3
\put{$j_4$} at 4 3

\setdashes  <.4mm,.4mm>
\setquadratic
\plot 2.25  0  2.4  -0.25  2.25 -0.5 /
\plot 2.25   0  2.1  -0.25  2.25 -0.5 /

\plot 2.5  .7  2.65  0.45  2.5 0.2 /
\plot 2.5  .7  2.35  0.45  2.5 0.2 /

\plot 2.71  1.35  2.86  1.1  2.71 0.85 /
\plot 2.71  1.35  2.56  1.1  2.71 0.85 /

\setsolid
\setquadratic
\plot 0 -1  2.5 0.5 5 -1  /
\plot 2 2  3 1 4 2 /
\endpicture}$
\end{center}
%Note that we use $c_{1}$ instead of $c_4$ when we calculate
%\eqref{eq:modified weight} for the left diagram above.
Since we calculate crossing parts in
$\wt(_{\kk^{y+m}}\sigma_{{\widetilde{\kk}}^1})$ for the left diagram above, we must
multiply by $|i_1|+1$. When the edge which does not cross the wall
in $\td$ is in the right-hand side of the wall, the situation is the
same. Therefore there is no change in this part before and after
flipping.

Consider two edges in $d$ such that both edges cross the wall in
$\td$. They cross before flipping if and only if they do not cross
after flipping. For example, in the following case, the crossing of
the left diagram is $(-1)^{|j_4||j_6|+|j_5||j_6|}$, and the one of
the right diagram is $(-1)^{|j_4||j_5|}$, which differs by
$(-1)^{p(\jj^{RH})}$.
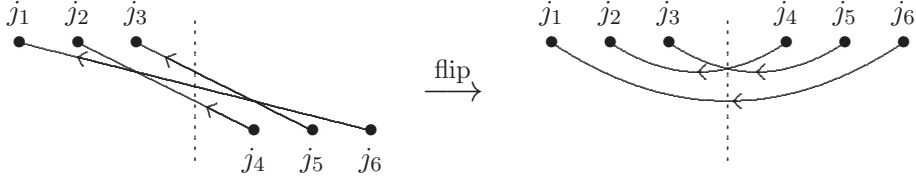
\begin{center}
${ \ \ \ \beginpicture
\setcoordinatesystem units <0.78cm,0.39cm>
%\setplotarea x from 0 to 6, y from -2 to 3
%\put{$d_1 = $} at 0 1.5
\put{$\bullet$} at  5 -1
\put{$\bullet$} at  6 -1
\put{$\bullet$} at  7 -1
\put{$\bullet$} at  1 2
\put{$\bullet$} at  2 2
\put{$\bullet$} at  3 2

\plot 2 2 5 -1 /
\plot 3 2 6 -1 /
\plot 1 2 7 -1 /

\arrow <3 pt> [1,2] from 7 -1 to  2 1.5
\arrow <3 pt> [1,2] from 5 -1 to  4.2 -0.2
\arrow <3 pt> [1,2] from 6 -1 to  3.5 1.5

\setdashes  <.4mm,1mm>
\plot 4 -2   4 3 /
\setsolid
\setquadratic
\put{$j_1$} at 1 3
\put{$j_2$} at 2 3
\put{$j_3$} at 3 3
\put{$j_4$} at 5 -2
\put{$j_5$} at 6 -2
\put{$j_6$} at 7 -2

\setdashes  <.4mm,.4mm>
%\setquadratic
%\plot 3  1.5  3.3  1  3 0.5 /
%\plot 3  1.5  2.7  1  3 0.5 /
%\setsolid
\endpicture}$
\hskip1em {\Large $\stackrel{{\rm flip}}{\longrightarrow}$}
{\beginpicture
\setcoordinatesystem units <0.78cm,0.39cm>
\setplotarea x from 0 to 8, y from -2 to 3
%\put{$d_1 = $} at 0 1.5
\put{$\bullet$} at  1 2
\put{$\bullet$} at  2 2
\put{$\bullet$} at  3 2
\put{$\bullet$} at  5 2
\put{$\bullet$} at  6 2
\put{$\bullet$} at  7 2

\setdashes  <.4mm,1mm>
\plot 4 -2   4 3 /
\setsolid
\setquadratic
\plot 1 2 4 0 7 2 /
\plot 2 2 3.5 1 5 2 /
\plot 3 2 4.5 1 6 2 /

\put{$j_1$} at 1 3
\put{$j_2$} at 2 3
\put{$j_3$} at 3 3
\put{$j_4$} at 5 3
\put{$j_5$} at 6 3
\put{$j_6$} at 7  3

\arrow <3 pt> [1,2] from 3.55 1 to  3.5 1
\arrow <3 pt> [1,2] from 4.55 1 to  4.5 1
\arrow <3 pt> [1,2] from 4.15 0 to  4.1 0
\endpicture} \end{center}

For the crossing of edges which become horizontal edges on the top
and the bottom row in $\td$,
 respectively, it differs by $(-1)^{|\ii^{RH}||\jj^{RH}|}$.\begin{center}
  $
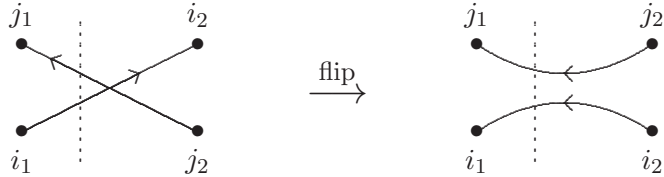
{\beginpicture
\setcoordinatesystem units <0.78cm,0.39cm>
%\setplotarea x from 0 to 6, y from -2 to 3
%\put{$d_1 = $} at 0 1.5
\put{$\bullet$} at  1 -1  \put{$\bullet$} at  1 2
\put{$\bullet$} at  4 -1  \put{$\bullet$} at  4 2
\plot 4 2 1 -1 /
\plot 1 2 4 -1 /
\arrow <3 pt> [1,2] from 1 -1 to  3 1
\arrow <3 pt> [1,2] from 4 -1 to  1.5 1.5
\setdashes  <.4mm,1mm>
\plot 2 -2   2 3 /
\setsolid
\setquadratic
\put{$i_1$} at 1 -2
\put{$j_2$} at 4 -2
\put{$j_1$} at 1 3
\put{$i_2$} at 4 3
\endpicture}$
\hskip3em {\Large $\stackrel{{\rm flip}}{\longrightarrow}$} \hskip3em
${\beginpicture
\setcoordinatesystem units <0.78cm,0.39cm>
%\setplotarea x from 0 to 6, y from -2 to 3
%\put{$d_1 = $} at 0 1.5
\put{$\bullet$} at  1 -1  \put{$\bullet$} at  1 2
\put{$\bullet$} at  4 -1  \put{$\bullet$} at  4 2

\setdashes  <.4mm,1mm>
\plot 2 -2   2 3 /
\setsolid
\setquadratic
\plot 1 -1 2.5 0 4 -1 /
\plot 1 2 2.5 1 4 2 /
\put{$i_1$} at 1 -2
\put{$i_2$} at 4 -2
\put{$j_1$} at 1 3
\put{$j_2$} at 4 3
\arrow <3 pt> [1,2] from 2.55 0 to  2.5 0
\arrow <3 pt> [1,2] from 2.55 1 to  2.5 1
\endpicture}$
\end{center}
The crossing part of the left diagram is
$(-1)^{(|i_1|+1)|j_2|}=(-1)^{|i_2||j_2|}$ because
$_{\iiL\jjR}d_{\jjL\iiR}$ is consistently labeled. So it differs by
$(-1)^{|\ii^{RH}||\jj^{RH}|}$.

Similarly, we can check that for the crossing of edges which become
horizontal edges on the bottom row in $\td$, it differs by
$(-1)^{p(\ii^{RH})}$. Therefore comparing the crossing part, we
obtain the equality
\begin{align}
\prod_{c}(-1)^{|c|}=\prod_{c'}(-1)^{|c'| + |\ii^{RH}||\jj^{RH}|+p(\ii^{RH})+p(\jj^{RH})}
\end{align}
for all crossings $c$  and $c'$ in the expression of
$\wt(_{\kk^{y+m}}\sigma_{{\widetilde{\kk}}^1})$ and
$\wt(_{\ii}\td_{\jj})$, respectively.

\vskip1em
(III) We calculate the part concerning with the marked edges.

In \eqref{eq:deformation of d}, the $b_i$-th vertex on the top row
in $d$
 is the top vertex of a vertical marked edge $\te$ on the left-hand side of the wall in $\td$ or
the left vertex of a horizontal marked edge $\te$ on the top row in
$\td$. We observe $ \wt(  {_{{\widetilde{\kk}}^1}} c_{b_1}
{_{{\widetilde{\kk}}^2}}) \cdots \wt({_{{\widetilde{\kk}}^{\ell+x}}}
c_{b_{\ell+x}}  {_{{\widetilde{\kk}}^{\ell+x+1}}} )$ is the same as
$\prod_{\te}(-1)^{m(\te)}$ in $\wt(_{\ii} \td_{\jj})$.
We  also observe that the $a'_i$-th
vertex on the bottom row is the left vertex of the horizontal marked
edge $e$ on the bottom row in $\td$ for all $1 \leq i \leq y$. Then
we have
 $\wt(_{\iiL \jjR} c_{a'_1}  {_{\kk^1}}) \cdots \wt(_{\kk^{y-1}}c_{a'_{y}}  {_{\kk^{y}}})$
is the same as $\prod_{e}(-1)^{m(e)}$ in $\wt(_{\ii} \td _{\jj})$.

Now we consider the marked edges which are the vertical edges on the
right-hand side of the wall in $\td$. Recall that their bottom
vertices in $d$ are $r+1 \leq a''_1 < \cdots < a''_m \leq r+s$.
Consider the partition of the set $\{1, \cdots, r \} = \{ s_{1}<
\cdots < s_{k_3}\} \sqcup \{t_1 < \cdots < t_{k_4} \}$ such that the
$s_i$-th vertices are the left vertices of the horizontal edges on
the bottom in $\td$, and the $t_i$-th vertices are the bottom
vertices of the vertical edges on the left-hand side of the wall in
$\td$. Let $\ilh:=(i_{s_1}, \cdots, i_{s_{k_3}}),\ilv:=(i_{t_1},
\cdots, i_{t_{k_4}})$. Similarly, we can define $\jlh,\jlv$ such
that $\jlh$ list the labels of the left vertices of all horizontal
edges on the top row and $\jlv$ list the labels of the top vertices
of all vertical edges on the left-hand side of the wall for
$_{\ii}\td_{\jj}$. Then we obtain the following equalities in
$\Z_2$:
\begin{align} \label{eq:parity}
|\ilv|=|\jlv|+\ell,  \ |\ilh|=|\irh|+y, \ |\irv|=|\jrv|+m, \ |\jrh|=|\jlh|+x .
\end{align}
By a direct calculation we can check that
\begin{align*}
\wt(_{\kk^{y}} c_{a''_1}  {_{\kk^{y+1}}})&= (-1)^{|\ii^{L}|+y+|j_{r+1}|+ \cdots + |j_{a''_1}|+1},\\
\wt(_{\kk^{y+1}} c_{a''_2}  {_{\kk^{y+2}}})
&= (-1)^{|\ii^{L}|+y+|j_{r+1}|+ \cdots + |j_{a''_1}|+1 +|j_{a''_1+1}| + \cdots + |j_{a''_2}|+1}.
\end{align*}
Consequently, we obtain
$$\wt(_{\kk^{y}} c_{a''_1}  {_{\kk^{y+1}}})
\cdots \wt(_{\kk^{y+m-1}}c_{a''_{m}}  {_{\kk^{y+m}}})
=(-1)^{(y+|\iiL|)m +  \sum_{k=1}^m \sum_{p=r+1}^{a''_k} |j_{p}| +
\binom{m+1}{2} }.$$ From the identity \eqref{eq:parity}, we see that
the following equality holds in $\Z_2$:
\begin{align} \label{eq:x-2}
  (y+|\iiL|)m =(y + |\ilv|+|\ilh|)m=\ell m + m|\jlv|+ m|\irh|.
\end{align}
Note that $\prod_{\te}(-1)^{m(\te)}=(-1)^{m |\jjL|+ \sum_{k=1}^m
\sum_{p=r+1}^{a''_k-1} |j_p|}$ for all marked vertical edges $\te$
on the right-hand side of the wall in $\td$.

Combining the equalities \eqref{eq:modified weight}--\eqref{eq:x-2}, we obtain the desired result
\eqref{eq:flip and weight}.
\end{proof}

We now prove the main result of this section.

\begin{theorem} \label{th:Psi is a bijection}
{\rm  The linear map $\Psi_{r,s}:\BB \longrightarrow
\End_{\qn}(\mix)$ is surjective. Moreover, when $n \geq r+s$,
  $\Psi_{r,s}$ is a bijection.}
\end{theorem}
\begin{proof}
  Since the map ${\it flip}: \Endc(\Vrs) \longrightarrow \Endc(\mix)$ is a $\qn$-supermodule isomorphism,
  our assertion follows from Proposition \ref{prop:duality between D_k and qn},
  Proposition \ref{lem:supercommute between B_r,s and q(n)},
   and Lemma \ref{lem:Lemma}.
\end{proof}

\vskip 5mm

\section{The mixed Schur-Weyl-Sergeev duality}

This section is devoted to the proof of our main result: the {\it
mixed Schur-Weyl-Sergeev duality}. We will define a multiplication
on $\BB$ and show that $\BB$ is isomorphic to the supercentralizer
algebra $\End_{\qn}(\mix)^{\op}$ as an associative superalgebra whenever
$n \ge r+s$.

As we did for $\DD_{k}$, we define a multiplication on $\BB$ in two
steps.

\vskip 3mm

{\bf Step 1:} Marked concatenation

\vskip 2mm

For $d_1, d_2 \in \BB$, we define the {\it marked concatenation}
$d_1* d_2$ as follows. We first put $d_1$ under $d_2$ and identify
the vertices on the bottom row of $d_2$ with the vertices on the top
row of $d_1$. If there is a loop in the middle row, we define $d_1
* d_2 =0$. If there is no loop in the middle row, we declare that
an edge in this diagram is marked if and
only if the number of marked edges from $d_1$ and $d_2$ to form this
edge is odd. The diagram thus obtained is the marked concatenation
$d_1 * d_2$.

\vskip 3mm

{\bf Step 2:} Multiplication on $\BB$.

\vskip 2mm

To define a multiplication, we first define the numbers $c(d_1,d_2),
\ell(d_1,d_2), \rho(d_1,d_2), p(d_1,d_2) \in \Z_{\geq 0}$ as
follows.

\begin{enumerate}
\item We say that a vertex is {\it good} if it is the left vertex of a
horizontal edge or the top vertex of a vertical edge. For a good
vertex of a marked edge in $d_1$ that becomes a horizontal edge in
the bottom row of $d_1 * d_2$, we color it with $\boxed{i}$ . On the
other hand, for a good vertex of a marked edge in $d_1$ that becomes
a horizontal edge in the top row of $d_1 * d_2$, we color it with
{\large \textcircled{\small{$i$}}}. A good vertex of a marked edge
in $d_1$ that becomes a vertical edge in $d_1 * d_2$ is colored with
{\large \textcircled{\small{$i$}}}. The numbering $i$ is determined
by reading the good vertices from bottom to top and left to right.
%for squares and circles, respectively.
For the diagram $d_2$, we
carry out the same procedure succeeding the numbering obtained from
$d_1$. For each {\large \textcircled{\small{$i$}}}, we count the
number of $\boxed{j}$'s such that $j >i$. Let $c(d_1,d_2)$ be the
sum of these numbers for all {\large \textcircled{\small{$i$}}}.

%For the good vertex of each marked edge of $d_1$ \bred which becomes
%a horizontal edge on the bottom row  \ered (respectively, a
%horizontal edge on the top row or the vertical edge) in $d_1*d_2$,
%we colour $\boxed{i}$ (respectively, {\large
%\textcircled{\small{$i$}}}) in order from the bottom to the top and
%left to right. For an $(r,s)$-superdiagram $d_2$, we do the same
%procedure, succeeding the numbers obtained from $d_1$.

\item Let $a_k$ be the good vertex in a new edge in $d_1 * d_2$
connected to a marked edge with color {\large
\textcircled{\small{$i_{k}$}}} such that $i_1 < \cdots < i_t$. Then
we obtain a sequence $a_1 \cdots a_t$. Let $\ell_1(d_1,d_2)$ be the
arranging number for $a_1 \cdots a_t$ defined in Section 2.

\item Let $\rho_1(d_1,d_2)$ be the number of
pairs with the same entry in $\{1, \ldots, r\}$ in the sequence
$a_1, \ldots, a_t$.

\item For each {\large \textcircled{\small{$i$}}}, we define the
{\it passing number} to be the number of ${\large
\textcircled{\small{$j$}}} $'s such that $j<i$ and {\large
\textcircled{\small{$i$}}} passes on {\large
\textcircled{\small{$j$}}} when {\large \textcircled{\small{$i$}}}
goes to a good vertex of a new edge. Let $p_1(d_1,d_2)$ be the sum
of passing numbers for all {\large \textcircled{\small{$i$}}}.

%We define the \emph{passing number} for each {\large \textcircled{\small{$i$}}}.
%The passing number is the number of ${\large \textcircled{\small{$j$}}} $'s
%such that $j<i$ and
%{\large \textcircled{\small{$i$}}} passes on {\large \textcircled{\small{$j$}}} when
%{\large \textcircled{\small{$i$}}} goes to the good position of the
%new edge. Let $p_1(d_1,d_2)$ be the number obtained by summing
%passing numbers for all {\large \textcircled{\small{$i$}}}.

\item We carry out the same calculation of (3),(4) for the color $\boxed{i} $
to obtain the numbers $\ell_2(d_1,d_2)$ and $\rho_2(d_1,d_2)$.

\item The \emph{passing number} for $\boxed{i}$ is defined to be the number of $\boxed{j} $'s
such that $i < j$ and $\boxed{i}$ passes on $\boxed{j}$ when
$\boxed{i}$ goes to the good vertex of a new edge. Let
$p_2(d_1,d_2)$ be the sum of passing numbers for all $\boxed{i}$.

\item We define
\begin{align*}
&\ell(d_1,d_2):=\ell_1(d_1,d_2)+\ell_2(d_1,d_2), \ \ \ \rho(d_1,d_2):=\rho_1(d_1,d_2)+\rho_2(d_1,d_2),\\
& \text{ \ and\ }\ \  p(d_1,d_2):=p_1(d_1,d_2)+p_2(d_1,d_2).
\end{align*}
\end{enumerate}

With these data, we define the multiplication on $\BB$ by
\begin{equation}
d_1d_2=(-1)^{c(d_1,d_2)+\ell(d_1,d_2)+\rho(d_1,d_2)+p(d_1,d_2)}d_1
*d_2.
\end{equation}

Note that if the new edge in $d_1 *d_2$ consists only of the
vertical edges of $d_1$ and $d_2$, the passing number of {\large
\textcircled{\small{$i$}}} included in this new edge is always $0$.

\begin{example} \label{ex:example of the multiplication}
  If
  \begin{center}
  ${\beginpicture
\setcoordinatesystem units <0.78cm,0.39cm>
\setplotarea x from -1 to 6, y from -2 to 4
\put{$d_1 = $} at 0 1.5
\put{$\bullet$} at  1 0  \put{$\bullet$} at  1 3
\put{$\bullet$} at  2 0  \put{$\bullet$} at  2 3
\put{$\bullet$} at  3 0  \put{$\bullet$} at  3 3
\put{$\bullet$} at  4 0  \put{$\bullet$} at  4 3
\put{$\bullet$} at  5 0  \put{$\bullet$} at  5 3
\put{$\bullet$} at  6 0  \put{$\bullet$} at  6 3
\put{,} at 6.5 1

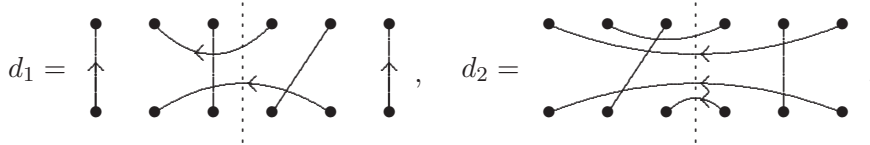
\plot 1 3 1 0 /
\plot 3 3 3 0 /
\plot 4 0 5 3 /
\plot 6 0 6 3 /
\arrow <3 pt> [1,2] from 1 0 to  1 1.7
\arrow <3 pt> [1,2] from 3.65 1 to 3.6 1
\arrow <3 pt> [1,2] from 2.75 2.05 to 2.7 2.05
\arrow <3 pt> [1,2] from 6 0 to 6 1.7
\setdashes  <.4mm,1mm>
\plot 3.5 -1   3.5 4 /
\setsolid
\setquadratic
\plot 2 3 3 2 4 3 /
\plot 2 0 3.5 1 5 0  /
\endpicture}$  $ \ \ \ {\beginpicture
\setcoordinatesystem units <0.78cm,0.39cm>
%\setplotarea x from 0 to 6, y from -2 to 3
\put{$d_2 = $} at 0 1.5
\put{$\bullet$} at  1 0  \put{$\bullet$} at  1 3
\put{$\bullet$} at  2 0  \put{$\bullet$} at  2 3
\put{$\bullet$} at  3 0  \put{$\bullet$} at  3 3
\put{$\bullet$} at  4 0  \put{$\bullet$} at  4 3
\put{$\bullet$} at  5 0  \put{$\bullet$} at  5 3
\put{$\bullet$} at  6 0  \put{$\bullet$} at  6 3
\put{,} at 6.5 1
%\plot 1 3 1 0 /
\plot 3 3 2 0 /
\plot 5 0 5 3 /
%\plot 6 0 6 3 /
\arrow <3 pt> [1,2] from 3.65 0.45 to  3.6 0.45
\arrow <3 pt> [1,2] from 3.65 1 to 3.6 1
\arrow <3 pt> [1,2] from 3.65 2 to 3.6 2
\setdashes  <.4mm,1mm>
\plot 3.5 -1   3.5 4 /
\setsolid
\setquadratic
\plot 3 0 3.5 .5 4 0 /
\plot 1 0 3.5 1 6 0  /
\plot 1 3 3.5 2 6 3 /
\plot 2 3 3 2.5 4 3 /
\endpicture}$ \end{center}
  then the colored diagram is given as follows:
  $$
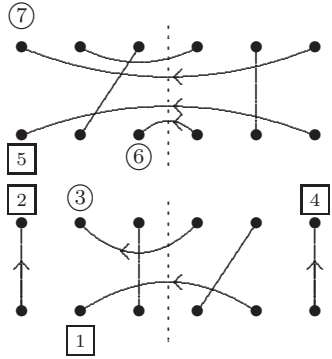
{\beginpicture
\setcoordinatesystem units <0.78cm,0.39cm>
%\setplotarea x from 0 to 6, y from -2 to 3

%\put{$d_2 = $} at 0 1.5
\put{$\bullet$} at  1 0  \put{$\bullet$} at  1 3
\put{$\bullet$} at  2 0  \put{$\bullet$} at  2 3
\put{$\bullet$} at  3 0  \put{$\bullet$} at  3 3
\put{$\bullet$} at  4 0  \put{$\bullet$} at  4 3
\put{$\bullet$} at  5 0  \put{$\bullet$} at  5 3
\put{$\bullet$} at  6 0  \put{$\bullet$} at  6 3

%\plot 1 3 1 0 /
\plot 3 3 2 0 /
\plot 5 0 5 3 /
%\plot 6 0 6 3 /
\arrow <3 pt> [1,2] from 3.65 0.45 to  3.6 0.45
\arrow <3 pt> [1,2] from 3.65 1 to 3.6 1
\arrow <3 pt> [1,2] from 3.65 2 to 3.6 2
\setdashes  <.4mm,1mm>
\plot 3.5 -1   3.5 4 /
\setsolid

%\put{아래 시작} at 0 1.5
\put{$\bullet$} at  1 -3  \put{$\bullet$} at  1 -6
\put{$\bullet$} at  2 -3  \put{$\bullet$} at  2 -6
\put{$\bullet$} at  3 -3  \put{$\bullet$} at  3 -6
\put{$\bullet$} at  4 -3  \put{$\bullet$} at  4 -6
\put{$\bullet$} at  5 -3  \put{$\bullet$} at  5 -6
\put{$\bullet$} at  6 -3  \put{$\bullet$} at  6 -6
\plot 1 -3 1 -6 /
\plot 3 -3 3 -6 /
\plot 4 -6 5 -3 /
\plot 6 -6 6 -3 /
\arrow <3 pt> [1,2] from 1 -6 to  1 -4.3
\arrow <3 pt> [1,2] from 6 -6 to  6 -4.3
\arrow <3 pt> [1,2] from 3.65 -5 to 3.6 -5
\arrow <3 pt> [1,2] from 2.75 -3.95 to 2.7 -3.95
\setdashes  <.4mm,1mm>
\plot 3.5 -7   3.5 -2 /
\setsolid

\put{\tiny{$\boxed{1}$}} at 2 -7
\put{\tiny{$\boxed{2}$}} at 1 -2.2
\put{\tiny{$\boxed{4}$}} at 6 -2.2
\put{\tiny{$\boxed{5}$}} at 1 -0.8
\put{\textcircled{{\scriptsize $3$}}} at 2 -2.2
\put{\textcircled{{\scriptsize $6$}}} at 3 -0.8
\put{\textcircled{{\scriptsize $7$}}} at 1 4

\setquadratic
\plot 2 -3 3 -4 4 -3 /
\plot 2 -6 3.5 -5 5 -6  /

\plot 3 0 3.5 .5 4 0 /
\plot 1 0 3.5 1 6 0  /
\plot 1 3 3.5 2 6 3 /
\plot 2 3 3 2.5 4 3 /
\endpicture}$$

\vskip 3mm

\noindent We observe the following:

\vskip 3mm

\begin{enumerate}

\item The colored vertices  {\large \textcircled{\small{$3$}}},
$\boxed{4}$ and $\boxed{5}$ yield $c(d_1,d_2)=2$.

\item The good vertices of the edges that are connected to {\large
\textcircled{\small{$3$}}}, {\large \textcircled{\small{$6$}}},
{\large \textcircled{\small{$7$}}} are $3,3,1$, respectively. So
$\ell_1(d_1,d_2)=2, \rho_1(d_1,d_2)=1$.

\item  The colored vertices {\large \textcircled{\small{$3$}}} and
{\large \textcircled{\small{$6$}}} yield $p_1(d_1,d_2)=1$.

\item  The good vertices of the edges that are connected to
$\boxed{1}, \boxed{2}, \boxed{4}, \boxed{5}$ are $2,1,1,1$,
respectively. So $\ell_2(d_1,d_2)=3,\rho_2(d_1,d_2)=1$.

\item The colored vertices $\boxed{4}$ and $\boxed{5}$ yield $p_2(d_1,d_2)=1$.

\end{enumerate}
\vskip 3mm

Consequently, we obtain

\vskip 3mm

$$
  {\beginpicture
\setcoordinatesystem units <0.78cm,0.39cm>
\setplotarea x from -1 to 6.5, y from -2.3 to 2.5
\put{$d_1 d_2= - $} at -0.4 .2

\put{$\bullet$} at  1 -1  \put{$\bullet$} at  1 2
\put{$\bullet$} at  2 -1  \put{$\bullet$} at  2 2
\put{$\bullet$} at  3 -1  \put{$\bullet$} at  3 2
\put{$\bullet$} at  4 -1  \put{$\bullet$} at  4 2
\put{$\bullet$} at  5 -1  \put{$\bullet$} at  5 2
\put{$\bullet$} at  6 -1  \put{$\bullet$} at  6 2
\put{.} at 6.5 0.2
\plot 3 2 3 -1 /
\plot 5 2 4 -1 /

\arrow <3 pt> [1,2] from 3.65 -0.5 to  3.6 -0.5
\arrow <3 pt> [1,2] from 2.65 -0.1 to  2.6  -0.1
\arrow <3 pt> [1,2] from 3.65 1 to 3.6 1
\setdashes  <.4mm,1mm>
\plot 3.5 -2   3.5 3 /
\setsolid
\setquadratic
\plot 2 -1 3.5 -0.5 5 -1 /
\plot 1 -1 3.5 0  6 -1  /
\plot 1 2 3.5 1  6 2  /
\plot 2 2 3 1.5  4 2 /
\endpicture}
$$
\end{example}

%If $d_1 *d_2$ is not zero and $d_1$ has $i$ marked edges and $d_2$
%has $j$ marked edges, then the number of the marked edges in
%$d_1d_2$ is $i+j-2 \rho(d_1,d_2)$. So
From the definition of marked concatenation, we observe that $\BB$ is a superalgebra (which may
not be associative at this point). We call $\BB$ the
\emph{$(r,s)$-walled Brauer superalgebra}, or simply the
\emph{walled Brauer superalgebra}. The identity element is the
diagram such that each vertex on the top row is connected with the
corresponding vertex on the bottom row by the normal edge.

Note that the multiplication of $(r,s)$-superdiagrams without marked edges is the same as the
multiplication on the walled Brauer algebra $B_{r,s}(0)$.
Therefore, the even part of $\BB$ contains the walled Brauer algebra $B_{r,s}(0)$ as a subalgebra.

Now we will show that the linear map $\Psi_{r,s}$ gives a
well-defined (right) action of $\BB$ on $\mix$. For $1 \leq p \leq
r$, $r+1 \leq q \leq r+s$, we define $e_{p,q}$ to be the following
$(r,s)$-superdiagram:
$$
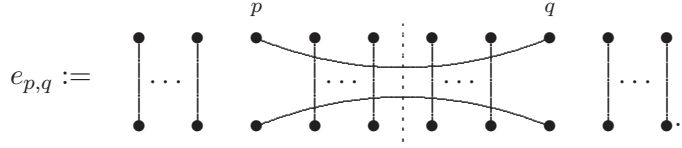
{\beginpicture
\setcoordinatesystem units <0.78cm,0.39cm>
%\setplotarea x from 0 to 6, y from -2 to 3
\put{$e_{p,q} := $} at -0.5 1.5
\put{$\bullet$} at  1 0  \put{$\bullet$} at  1 3
\put{$\bullet$} at  2 0  \put{$\bullet$} at  2 3
\put{$\bullet$} at  3 0  \put{$\bullet$} at  3 3
\put{$\bullet$} at  4 0  \put{$\bullet$} at  4 3
\put{$\bullet$} at  5 0  \put{$\bullet$} at  5 3
\put{$\bullet$} at  6 0  \put{$\bullet$} at  6 3
\put{$\bullet$} at  7 0  \put{$\bullet$} at  7 3
\put{$\bullet$} at  8 0  \put{$\bullet$} at  8 3
\put{$\bullet$} at  9 0  \put{$\bullet$} at  9 3
\put{$\bullet$} at  10 0  \put{$\bullet$} at  10 3

\put{$\cdots$} at 1.5 1.5
\put{$\cdots$} at 4.5 1.5
\put{$\cdots$} at 6.5 1.5
\put{$\cdots$} at 9.5 1.5

\put{{\scriptsize $p$}} at 3 4
\put{{\scriptsize $q$}} at 8 4

\plot 1 0 1 3 /
\plot 2 0 2 3 /
\plot 4 0 4 3 /
\plot 5 0 5 3 /
\plot 6 0 6 3 /
\plot 7 0 7 3 /
\plot 9 0 9 3 /
\plot 10 0 10 3 /

%\arrow <3 pt> [1,2] from 2 0 to  2 1.7
%\arrow <3 pt> [1,2] from 1.45 0.75 to  1.4 0.75
%\arrow <3 pt> [1,2] from 4 0 to 3.5 1.5
\setdashes  <.4mm,1mm>
\plot 5.5 -0.5   5.5 3.5 /
\setsolid
\setquadratic
\plot 3 3 5.5 2 8 3 /
\plot 3 0 5.5 1 8 0  /
\endpicture}.$$
\vskip0.5em

% For $e_{p,q}$ we assume $ 1 \leq p \leq r$, $r+1 \leq q \leq r+s$.
We use the same definition $c_i$ given in \eqref{def:a superdiagram
c_i-1}, \eqref{def:a superdiagram c_i-2}. A diagram in $\BB$ with normal vertical edges only will
usually be denoted by $\sigma$ and will be identified with an
element in $\Sigma_r \times \Sigma_s$. For simplicity, we write
$\Psi_{r,s}(\sigma)=\widetilde{\sigma}$,
$\Psi_{r,s}(e_{p,q})=\widetilde{e}_{p,q}$ and
$\Psi_{r,s}(c_i)=\widetilde{c}_i$.

\nc{\tepq}{\widetilde{e}_{p,q}}
\nc{\ts}{\widetilde{s}}
\nc{\tsigma}{\widetilde{\sigma}}
\nc{\tc}{\widetilde{c}}
\nc{\tle}{\widetilde{e}}
\nc{\cent}{\End_{\qn}(\mix)}

\begin{lemma} \label{lem:relations in the centralizer}
{\rm  With the above notations, the following relations hold in
$\End_{\qn}(\mix)^{\rm
 op}$.
  \begin{equation}
    \begin{aligned}
 &\tsigma \tc_j=\tc_{\sigma (j)} \tsigma \text{  for  } \sigma \in \Sigma_r \times \Sigma_s,    \\
&\tepq^{\ 2}=0, \ \ \  \tepq \tc_p \tepq =0,\\
&\tc_i^{\ 2}=-1 \text{  for   } 1 \leq i \leq r, \ \ \  \tc_i^{\ 2}=1  \text{  for   } r+1 \leq i \leq r+s, \\
& \tc_i\tc_j=-\tc_j\tc_i \text{  for  } i \neq j, \\
&\tc_p \tepq =\tc_q \tepq, \  \tepq \tc_p=\tepq \tc_q, \  \tepq \tc_{p'}=\tc_{p'} \tepq \text{  for different  } p,q,p'.
    \end{aligned}
  \end{equation} }
\end{lemma}

\begin{proof}
Each of the relations can be checked by direct calculations. For
instance, we show the first relation for the case $r+1 \leq j \leq
r+s$. Let $\iiL \in I^r, \iiR \in I^s $. We have
  \begin{equation*}
  \begin{aligned}
  \tc_j \tsigma(v_{\iiL} \tensor w_{\iiR})
= &\tc_j
  ( (-1)^{x+y} v_{i_{\sigma(1)}} \tensor \cdots \tensor v_{i_{\sigma(r)}}
  \tensor w_{i_{\sigma(r+1)}} \tensor \cdots \tensor w_{i_{\sigma(r+s)}})\\
  = &(-1)^{x+y + |\iiL| +|i_{\sigma(r+1)}| + \cdots + |i_{\sigma(j-1)}|}
  v_{i_{\sigma(1)}} \tensor \cdots \tensor v_{i_{\sigma(r)}} \\
  &  \hskip 5mm  \tensor w_{i_{\sigma(r+1)}} \tensor \cdots \tensor
w_{\ol{i_{\sigma (j)}}}
  \tensor \cdots \tensor w_{i_{\sigma(r+s)}},
\end{aligned}
\end{equation*}
where $x$ and $y$ arise from  crossings of the left-hand side and
right-hand side of the wall in $\sigma$, respectively.

Let $\sigma(j)=k$. Then
\begin{align*}
   \tsigma \tc_k (v_{\iiL} \tensor w_{\iiR})
   =&\tsigma ((-1)^{|\iiL| +|i_{r+1}| + \cdots +|i_{k-1}|}
   \ v_{i_{1}} \tensor \cdots \tensor v_{i_{r}} \\
 &   \hskip 5mm \tensor w_{i_{r+1}} \tensor \cdots \tensor w_{\ol{i_{k}}} \tensor \cdots \tensor w_{i_{r+s}}),\\
  = &(-1)^{|\iiL| +|i_{r+1}| + \cdots +|i_{k-1}|+x+y'}
  v_{i_{\sigma(1)}} \tensor \cdots \tensor v_{i_{\sigma(r)}} \\
& \hskip 5mm  \tensor w_{i_{\sigma(r+1)}} \tensor \cdots \tensor
w_{\ol{i_{\sigma(j)}}} \tensor \cdots \tensor w_{i_{\sigma(r+s)}},
\end{align*}
  where $x$ and $y'$ arise from crossings  of the left-hand side  and right-hand side of the wall in $\sigma$, respectively.
One can verify that $$(-1)^{y +|i_{\sigma(r+1)}| + \cdots
+|i_{\sigma(j-1)}|} =(-1)^{y' +|i_{r+1}| + \cdots +|i_{k-1}|}.$$
There are three types of the edges of $\sigma$ as shown below:

$$
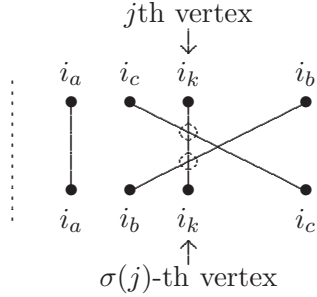
{\beginpicture
\setcoordinatesystem units <0.78cm,0.39cm>
\setplotarea x from 0 to 7, y from -5 to 6
%\put{$d_1 = $} at 0 1.5
%\put{$\bullet$} at  1 -1  \put{$\bullet$} at  1 2
\put{$\bullet$} at  2 -1  \put{$\bullet$} at  2 2
\put{$\bullet$} at  3 -1  \put{$\bullet$} at  3 2
\put{$\bullet$} at  4 -1  \put{$\bullet$} at  4 2
%\put{$\bullet$} at  5 -1   \put{$\bullet$} at  5 2
\put{$\bullet$} at  6 -1  \put{$\bullet$} at  6 2
\put{$i_a$} at 2 -2
\put{$i_b$} at 3 -2
\put{$i_k$} at 4 -2
\put{$i_c$} at 6 -2
\put{$\sigma(j)$-th vertex} at 4 -4
\arrow <3 pt> [1,2] from 4 -3.5 to  4 -2.7
\put{$j$th vertex} at 4 5
\arrow <3 pt> [1,2] from 4 4.5 to  4 3.7

\put{$i_a$} at 2 3
\put{$i_b$} at 6 3
\put{$i_k$} at 4 3
\put{$i_c$} at 3 3

\plot 2 2 2 -1 /
\plot 3 -1 6 2 /
\plot 4 -1 4 2 /
\plot 6 -1 3 2 /

%\arrow <3 pt> [1,2] from 2 0 to  2 1.7
%\arrow <3 pt> [1,2] from 1.45 0.75 to  1.4 0.75
%\arrow <3 pt> [1,2] from 4 0 to 3.5 1.5
\setdashes  <.4mm,1mm>
\plot 1 -2   1 3 /
\setsolid

\setdashes  <.4mm,.4mm>
\setquadratic
\plot 4 0.3 4.15 0 4 -0.3 /
\plot 4 0.3 3.85 0 4 -0.3 /
\plot 4 1.3 4.15 1 4 0.7 /
\plot 4 1.3 3.85 1 4 0.7 /
\setsolid
\endpicture}$$

We have $y=|i_k||i_b|+|i_k||i_c|+|i_b||i_c|$ and $|i_{\sigma(r+1)}|
+ \cdots |i_{\sigma(j-1)}|=|i_a|+|i_c|$. Since
$y'=(|i_k|+1)|i_b|+(|i_k|+1)|i_c| + |i_b| |i_c|$ and $|i_{r+1}|
+\cdots +|i_{k-1}|=|i_a|+|i_b|$, we obtain the desired result.
\end{proof}

The following proposition is one of the key ingredients in proving
our main theorem.

\begin{prop} \label{prop:Psi is an algebra homomorphism}

{\rm The linear map $\Psi_{r,s}$ defines a superalgebra homomorphism
between $\BB$ and $\End_{\qn}(\mix)^{\op}$. That is, for all $d_1,
d_2 \in \BB$, we have
$$\Psi_{r,s}(d_1 d_2) = \Psi_{r,s}(d_1) \Psi_{r,s}(d_2) \quad
\text{in} \ \ \End_{\qn}(\mix)^{\op}. $$}
\end{prop}

\nc\Psirs{\Psi_{r,s}}
\begin{proof}
By \cite[Lemma 7.4]{BS}, the map $\Psi_{r,s}$ preserves the
multiplication for $(r,s)$-superdiagrams without marked edges.

First, we assume that there is a loop in the middle row of $d_1
*d_2$. If the number of marked edges to form a loop is odd, there is
no $\jj$ such that $_{\ii} {d_1}_{\jj}, \ _{\jj} {d_2}_{\kk}$ are
all consistently labeled. For any starting point of the loop, we
label $j$ and $\ol j$ simultaneously since there is an odd number of
the horizontal edges. Therefore, $\sum_{\jj \in I^{r+s}} \wt(_{\ii}
{d_1}_{\jj}) \wt(_{\jj} {d_2}_{\kk})=0$ for all $\kk \in I^{r+s}$,
and hence $\Psirs(d_1)\Psirs(d_2)=0$.

Suppose  the number of marked edges to form a loop is even. Simply,
we write $\widetilde{x}=\Psirs(x)$ for $x \in \BB$.
 Let $\tc_A:=\tc_{a_1} \cdots \tc_{a_m}$ for a non-empty subset
$A=\{a_1< \cdots< a_m \} \subset \{1, \cdots, r+s \}$, and
$\tc_A:=1$ for an empty set $A$. Using the decomposition in Lemma
\ref{lem:decomposition of a superdiagram}, we obtain
 $$\widetilde{d}_1=\widetilde{c}_{a_1} \cdots \widetilde{c}_{a_p}
 \widetilde{d_1'} \widetilde{c}_{b_1} \cdots \widetilde{c}_{b_q}, \ \ \
 \widetilde{d}_2=\widetilde{c}_{a'_1} \cdots \widetilde{c}_{a'_{t}}
 \widetilde{d_2'} \widetilde{c}_{b'_1} \cdots \widetilde{c}_{b'_{u}}.$$
%Since there are many indices,

Since the proof for the general case is too messy, we will work with
the following example to explain the main idea of proof.  Let us
draw only a loop in the middle row in $d_1 *d_2$ and let the
vertices in the loop be $x_1, x_2,x_3, y_1, y_2, y_3$ from left to
right.
$$
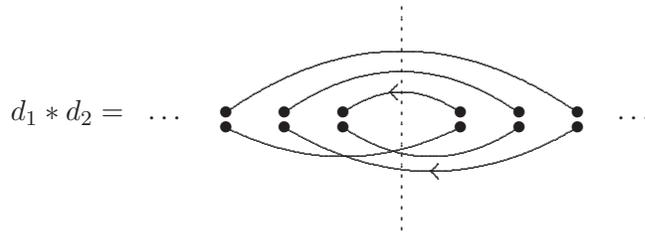
{\beginpicture
\setcoordinatesystem units <0.78cm,0.39cm>
\setplotarea x from -2.5 to 8.5, y from -4.5 to 4.5
\put{$d_1 *d_2 = $} at -1.7 0
\put{$\bullet$} at  1 0  \put{$\bullet$} at  1 -0.5
\put{$\bullet$} at  2 0  \put{$\bullet$} at  2 -0.5
\put{$\bullet$} at  3 0  \put{$\bullet$} at  3 -0.5
\put{$\bullet$} at  5 0  \put{$\bullet$} at  5 -0.5
\put{$\bullet$} at  6 0  \put{$\bullet$} at  6 -0.5
\put{$\bullet$} at  7 0  \put{$\bullet$} at  7 -0.5
\put{$\cdots$} at 0 -0.2
\put{$\cdots$} at 8 -0.2

\setdashes  <.4mm,1mm>
\plot 4 -4   4 4 /
\setsolid

\arrow <3 pt> [1,2] from 4.55 -2 to  4.5 -2
\arrow <3 pt> [1,2] from 3.85 0.7 to  3.8 0.7

\setquadratic
\plot 3 0 4 0.7 5 0 /
\plot 2 0 4 1.4 6 0  /
\plot 1 0 4 2.1 7 0  /
\plot 1 -0.5  3 -1.5   5 -0.5 /
\plot 2 -0.5  4.5 -2   7 -0.5 /
\plot 3 -0.5  4.5 -1.5   6 -0.5 /
\endpicture}$$
Observe that
$$d_1'=\sigma e_{x_1,y_1}e_{x_2,y_3}e_{x_3,y_2}, \quad
d_2'=e_{x_1,y_3} e_{x_2,y_2} e_{x_3,y_1} \tau$$ for some
$(r,s)$-superdiagrams $\sigma, \tau$ without the marked edges. Since
$\Psirs$ preserves the multiplication for $(r,s)$-superdiagrams
without the marked edges, combined with Lemma \ref{lem:relations in
the centralizer}, we obtain
$$\widetilde{d}_1= \pm \widetilde{c}_{A_1} \widetilde{\sigma} \widetilde{c}_{B_1}
\widetilde{e}_{x_1,y_1} \widetilde{e}_{x_2,y_3}
\widetilde{e}_{x_3,y_2} \widetilde{c}_{x_2}, \ \ \widetilde{d}_2=
\pm \widetilde{c}_{x_3} \te_{x_1,y_3} \te_{x_2,y_2} \te_{x_3,y_1}
\widetilde{c}_{A_2} \widetilde{\tau} \widetilde{c}_{B_2},$$ where
$A_i,B_i$ are the set of indices of the marked edges. Since $\tc_i
\tc_j=-\tc_j\tc_i$ for $i \neq j$, we need to take the sign into
account.
%may need $-1$.

By the associativity of $\End_{\C}(\mix)^{\op}$ and the relations in
Lemma \ref{lem:relations in the centralizer}, we know
$\widetilde{c}_{x_3}$ can be transformed  into $\widetilde{c}_{x_2}$
in $\widetilde{e}_{x_1,y_1} \widetilde{e}_{x_2,y_3}
\widetilde{e}_{x_3,y_2} \widetilde{c}_{x_2} \widetilde{c}_{x_3}
\te_{x_1,y_3} \te_{x_2,y_2} \te_{x_3,y_1}$. In general case, since
there are even number of the marked edges in a loop, the
$\widetilde{c}_i$'s will be canceled out, too. Notice that
\begin{align*}
\widetilde{e}_{x_1,y_1} \widetilde{e}_{x_2,y_3} \widetilde{e}_{x_3,y_2}
\te_{x_1,y_3} \te_{x_2,y_2} \te_{x_3,y_1}
=\Psirs(e_{x_1,y_1}e_{x_2,y_3} e_{x_1,y_3} e_{x_2,y_2} e_{x_3,y_1})
=0.
\end{align*}
Since $\widetilde{d}_1 \widetilde{d}_2  = \pm \cdots \widetilde{e}_{x_1,y_1}
\widetilde{e}_{x_2,y_3} \widetilde{e}_{x_3,y_2} \widetilde{c}_{x_2}
\widetilde{c}_{x_3}
\te_{x_1,y_3} \te_{x_2,y_2} \te_{x_3,y_1} \cdots$,
we obtain the desired result.

The general case can be handled in this manner.

\vskip 3mm

Next, let us consider the case $d_1 d_2 \neq 0$. For an
$(r,s)$-superdiagram $d$ without marked edges, we can decompose $d$
as follows.

(1) For $d$ in $\BB$, read the left vertex of each horizontal edges
on the bottom row  from left to right to obtain a sequence $p_1
\cdots p_a $. We denote by $q_a$ the right vertex of the horizontal
edge connected with $p_a$-th vertex.

(2) Read the left vertex of each horizontal edge on the top row from
left to right to obtain a sequence $p'_1 \cdots p'_a $. Let $q'_a$
be the right vertex of the edge connected with $p'_a$-th vertex.

(3) Let $\sigma$ be the $(r,s)$-superdiagram such that $p'_i$-th
vertex (respectively, $q'_i$-th vertex) on the top row is connected
with $p_i$-th vertex (respectively, $q_i$-th vertex) on the bottom
row by a normal edge. The other edges of $\sigma$ are all normal and
have the same connection as $d$. Therefore, $\sigma \in \Sigma_r
\times \Sigma_s$ and $\sigma^{-1}(p_i)=p'_i, \
\sigma^{-1}(q_i)=q'_i$ for all $i$, which implies $d=e_{p_1,q_1}
\cdots e_{p_a,q_a} \sigma$.

For a general $d \in \BB$, let $P$ (respectively, $Q$) be the set of
good vertices of marked horizontal edges on the bottom row
(respectively, marked vertical edges or marked horizontal edges on
the top row) of $d$. Let $c_A:=(((c_{a_1} c_{a_2}) \cdots )c_{a_m})$
for a non-empty set $A=\{a_1< \cdots< a_m \} \subset \{1, \cdots,
r+s \}$, and $c_A:=1$ for an empty set $A$, where $c_i$ is the
$(r,s)$-superdiagram defined in \eqref{def:a superdiagram c_i-1}, \eqref{def:a superdiagram c_i-2}. Then
we can decompose $d$ in the following form:
\begin{align} \label{eq:basis element form of B_r,s}
d=((((c_P e_{p_1,q_1}) \cdots e_{p_a,q_a}) \sigma) c_{Q})
\end{align}
satisfying the conditions

(i) $1 \leq p_1 < \cdots < p_a \leq r$, and $r+1 \leq q_i \leq r+s$ are all distinct,

(ii) $\sigma \in \Sigma_r \times \Sigma_s$ and $\sigma^{-1}(p_1) <
\cdots <\sigma^{-1}(p_a)$,

(iii) $ P \subset \{p_1, \cdots, p_a \}$, $Q \subset \{1, \cdots,
r+s \} \setminus \{\sigma^{-1}(q_1), \cdots, \sigma^{-1}(q_a) \} $.

\vskip3mm Note that when we calculate the multiplication in the
right-hand side of \eqref{eq:basis element form of B_r,s}, the
exponent of $-1$ is always $0$.  By Lemma \ref{lem:decomposition of
a superdiagram} and \cite[Lemma 7.4]{BS}, we obtain the
decomposition
$$\widetilde{d}_1= \widetilde{c}_P \widetilde{e}_{p_1,q_1} \cdots \widetilde{e}_{p_a,q_a}
 \widetilde{\sigma} \widetilde{c}_Q, \ \ \
\widetilde{d}_2= \widetilde{c}_{T} \widetilde{e}_{t_1,u_1} \cdots
\widetilde{e}_{t_x,u_x} \widetilde{\tau} \widetilde{c}_{U} .$$ For
simplicity,  let $\td_1'=\widetilde{e}_{p_1,q_1} \cdots
\widetilde{e}_{p_a,q_a}
 \widetilde{\sigma}$, $\td_2'=\widetilde{e}_{t_1,u_1} \cdots
\widetilde{e}_{t_x,u_x} \widetilde{\tau}$.

Note that we color all the elements of $P,Q,T,U$ with $\boxed{i}$ or
{\large \textcircled{\small{$i$}}}. Also we color the elements of
$P$ (respectively, $U$) with the square (respectively, circle). Let
$A_1$ (respectively, $A_2$) be the set of elements in $Q$ which are
colored with the square (respectively circle). Similarly, let $B_1$
(respectively, $B_2$) be the set of elements in $T$ which are
colored with the square (respectively, circle). Then we have
$$\td_1 \td_2 = (-1)^{c(d_1,d_2)} \tc_{P} \td_1' \tc_{A_1} \tc_{B_1} \tc_{A_2}\tc_{B_2} \td_2' \tc_{U}.$$
We shift the element $\tc_i$ ($i \in A_2, B_2,U$) using the
relations in Lemma \ref{lem:relations in the centralizer}. Then we
have
$$
\td_1 \td_2
=(-1)^{c(d_1,d_2)+p_1(d_1,d_2)} \tc_{P} \td_1' \tc_{A_1} \tc_{B_1} \td_2' \tc_{a_1} \cdots \tc_{a_t},
$$
where $a_i$'s are the good vertices of new edges in $d_1 *d_2$ and
$p_1(d_1,d_2)$ is the sum of passing numbers for all {\large
\textcircled{\small{$i$}}}. The exponent  $p_1(d_1,d_2)$ arises from
the following formula:
$$\te_{p,q} \tc_{b_1} \cdots \tc_{b_m} \tc_q
=(-1)^{x} \te_{p,q} \tc_{b_1} \cdots \tc_{b_m} \tc_p,$$ where all
$b_i \neq q$ and $x$ is the number of $b_i$'s such that $b_i=p$. For
example, if $$d_1={\beginpicture \setcoordinatesystem units
<0.78cm,0.39cm> \setplotarea x from 0 to 4, y from -2 to 3
%\put{$d_1 = $} at 0 1.5
%\put{$\bullet$} at  1 -1  \put{$\bullet$} at  1 2
\put{$\bullet$} at  1 -1  \put{$\bullet$} at  1 2
\put{$\bullet$} at  2 -1  \put{$\bullet$} at  2 2
\put{$\bullet$} at  3 -1  \put{$\bullet$} at  3 2

\plot 1 -1 1 2 /

%\arrow <3 pt> [1,2] from 2 0 to  2 1.7
%\arrow <3 pt> [1,2] from 1.45 0.75 to  1.4 0.75
\arrow <3 pt> [1,2] from 2.55 1 to 2.5 1
\setdashes  <.4mm,1mm>
\plot 2.5 -2   2.5 3 /
\setsolid

\setquadratic
\plot 2 -1 2.5 0 3 -1 /
\plot 2 2 2.5 1 3 2  /
\setsolid
\endpicture} \quad \text{and} \quad d_2={\beginpicture
\setcoordinatesystem units <0.78cm,0.39cm>
%\setplotarea x from 0 to 6, y from -2 to 3
%\put{$d_1 = $} at 0 1.5
%\put{$\bullet$} at  1 -1  \put{$\bullet$} at  1 2
\put{$\bullet$} at  1 -1  \put{$\bullet$} at  1 2
\put{$\bullet$} at  2 -1  \put{$\bullet$} at  2 2
\put{$\bullet$} at  3 -1  \put{$\bullet$} at  3 2
\plot 2 -1 1 2 /
%\arrow <3 pt> [1,2] from 2 0 to  2 1.7
%\arrow <3 pt> [1,2] from 1.45 0.75 to  1.4 0.75
\arrow <3 pt> [1,2] from 2.05 0 to 2 0
\setdashes  <.4mm,1mm>
\plot 2.5 -2   2.5 3 /
\setsolid

\setquadratic
\plot 1 -1 2 0 3 -1 /
\plot 2 2 2.5 1 3 2  /
\setsolid
\endpicture},$$ the labeled diagram is as follows:
$$
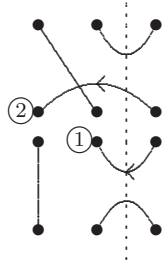
{\beginpicture
\setcoordinatesystem units <0.78cm,0.39cm>
\setplotarea x from 0 to 4, y from -2 to 3
%\put{$d_1 = $} at 0 1.5
%\put{$\bullet$} at  1 -1  \put{$\bullet$} at  1 2
\put{$\bullet$} at  1 -1  \put{$\bullet$} at  1 2
\put{$\bullet$} at  2 -1  \put{$\bullet$} at  2 2
\put{$\bullet$} at  3 -1  \put{$\bullet$} at  3 2

\plot 2 -1 1 2 /

%\arrow <3 pt> [1,2] from 2 0 to  2 1.7
%\arrow <3 pt> [1,2] from 1.45 0.75 to  1.4 0.75
\arrow <3 pt> [1,2] from 2.05 0 to 2 0
\setdashes  <.4mm,1mm>
\plot 2.5 -2   2.5 3 /
\setsolid

%밑에 것 시작
\put{$\bullet$} at  1 -5  \put{$\bullet$} at  1 -2
\put{$\bullet$} at  2 -5  \put{$\bullet$} at  2 -2
\put{$\bullet$} at  3 -5  \put{$\bullet$} at  3 -2
\put{\textcircled{{\scriptsize 1}}} at 1.7 -2
\put{\textcircled{{\scriptsize 2}}} at .7 -1

\plot 1 -5 1 -2 /

%\arrow <3 pt> [1,2] from 2 0 to  2 1.7
%\arrow <3 pt> [1,2] from 1.45 0.75 to  1.4 0.75
\arrow <3 pt> [1,2] from 2.55 -3 to 2.5 -3
\setdashes  <.4mm,1mm>
\plot 2.5 -6   2.5 -2 /
\setsolid

\setquadratic
\plot 1 -1 2 0 3 -1 /
\plot 2 2 2.5 1 3 2  /
\setquadratic
\plot 2 -5 2.5 -4 3 -5 /
\plot 2 -2 2.5 -3 3 -2  /
 \endpicture}.$$
By Lemma \ref{lem:relations in the centralizer}, we obtain
\begin{align*}
  \td_1 \td_2 &=(\te_{2,3}\tc_2)(\tc_1\te_{1,3}\ts_1) =  \te_{2,3}\tc_2 (\tc_1\te_{1,3})\ts_1
  =\te_{2,3}\tc_2 (\tc_3\te_{1,3})\ts_1
  =(\te_{2,3} \tc_2 \tc_3)\te_{1,3}\ts_1 \\
&  =(-1) \ (\te_{2,3} \tc_2 \tc_2)\te_{1,3}\ts_1
 =(-1) \ \te_{2,3} \tc_2 (\tc_2\te_{1,3}\ts_1)
 =(-1) \ \te_{2,3} \tc_2 (\te_{1,3}\ts_1 \tc_1) \\
 & =(-1) \ \te_{2,3} (\tc_2\te_{1,3}\ts_1) \tc_1
 =(-1) \ \te_{2,3} (\te_{1,3}\ts_1 \tc_1) \tc_1=(-1) \ \te_{2,3} \te_{1,3}\ts_1 \tc_1 \tc_1.
\end{align*}
Moreover, we have
$$
\tc_{a_1} \cdots \tc_{a_t}
=(-1)^{\ell_1(d_1,d_2) +\rho_1(d_1,d_2)}
\tc_{h'_1} \cdots \tc_{h'_z},
$$
where $h'_1 < h'_2 < \cdots < h'_z$. It follows that
$$
\td_1 \td_2
=(-1)^{c(d_1,d_2)+p_1(d_1,d_2) +\ell_1(d_1,d_2) +\rho_1(d_1,d_2)}
\tc_{P} \td_1' \tc_{A_1} \tc_{B_1} \td_2' \tc_{h'_1} \cdots \tc_{h'_z}.
$$

Similarly, by shifting $\tc_i$'s for $i \in P, A_1, B_1$, we have
$$
\td_1 \td_2
=(-1)^{c(d_1,d_2)+p(d_1,d_2) +\ell(d_1,d_2) +\rho(d_1,d_2)}
\tc_{h_1} \cdots \tc_{h_y} \td_1'  \td_2' \tc_{h'_1} \cdots \tc_{h'_z},
$$
where $h_1 < h_2 < \cdots < h_y$. Note that $\Psirs(d_1
*d_2)=\tc_{h_1} \cdots \tc_{h_y} \td_1'  \td_2' \tc_{h'_1} \cdots
\tc_{h'_z}$. Hence we obtain
$$\Psirs(d_1)\Psirs(d_2)=\Psirs(d_1d_2)$$
as desired.
\end{proof}

\begin{example}
For $(3,3)$-superdiagrams $d_1,d_2$ in Example \ref{ex:example of the multiplication}, we can decompose
  $$d_1=c_2 e_{2,5} s_4 c_1 c_2 c_6, \ d_2=c_1c_3e_{1,6}e_{3,4} s_2 c_1.$$
 The decomposition of $d_1$ is shown below:
  \begin{center}
  $
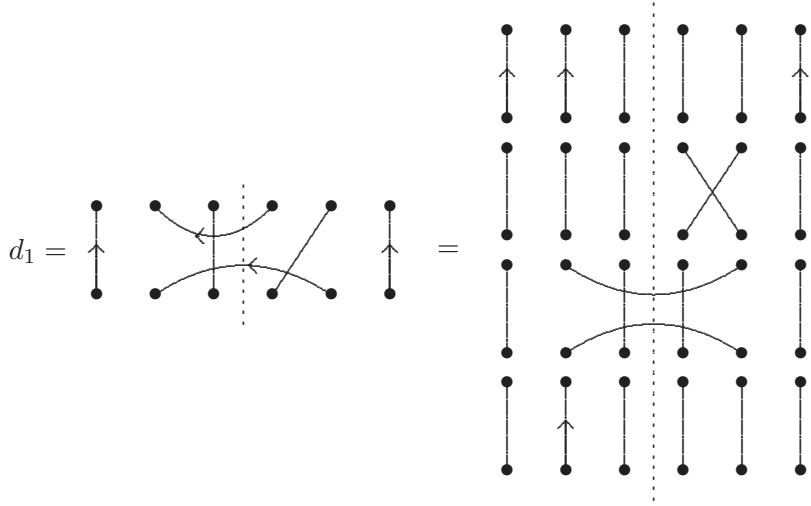
{\beginpicture
\setcoordinatesystem units <0.78cm,0.39cm>
\setplotarea x from -8 to 7, y from -7 to 10
\put{$d_1 = $} at -7 1.5
\put{$\bullet$} at  -6 0  \put{$\bullet$} at  -6 3
\put{$\bullet$} at  -5 0  \put{$\bullet$} at  -5 3
\put{$\bullet$} at  -4 0  \put{$\bullet$} at  -4 3
\put{$\bullet$} at  -3 0  \put{$\bullet$} at  -3 3
\put{$\bullet$} at  -2 0  \put{$\bullet$} at  -2 3
\put{$\bullet$} at  -1 0  \put{$\bullet$} at  -1 3

\plot -6 3 -6 0 /
\plot -4 3 -4 0 /
\plot -3 0 -2 3 /
\plot -1 0 -1 3 /
\arrow <3 pt> [1,2] from -6 0 to  -6 1.7
\arrow <3 pt> [1,2] from -3.35 1 to -3.4 1
\arrow <3 pt> [1,2] from -4.25 2 to -4.3 2
\arrow <3 pt> [1,2] from -1 0 to -1 1.7

\setdashes  <.4mm,1mm>
\plot -3.5 -1   -3.5 4 /
\setsolid

%옆에것 시작
\put{$ = $} at 0 1.5
\put{$\bullet$} at  1 -2  \put{$\bullet$} at  1 1
\put{$\bullet$} at  2 -2  \put{$\bullet$} at  2 1
\put{$\bullet$} at  3 -2  \put{$\bullet$} at  3 1
\put{$\bullet$} at  4 -2  \put{$\bullet$} at  4 1
\put{$\bullet$} at  5 -2  \put{$\bullet$} at  5 1
\put{$\bullet$} at  6 -2  \put{$\bullet$} at  6 1
%\put{.} at 6.5 1.5

\plot 1 -2 1 1 /
\plot 3 -2 3 1 /
\plot 6 -2 6 1 /
\plot 4 -2 4 1 /

\setdashes  <.4mm,1mm>
\plot 3.5 -7   3.5 10 /
\setsolid

\put{$\bullet$} at  1 -3  \put{$\bullet$} at  1 -6
\put{$\bullet$} at  2 -3  \put{$\bullet$} at  2 -6
\put{$\bullet$} at  3 -3  \put{$\bullet$} at  3 -6
\put{$\bullet$} at  4 -3  \put{$\bullet$} at  4 -6
\put{$\bullet$} at  5 -3  \put{$\bullet$} at  5 -6
\put{$\bullet$} at  6 -3  \put{$\bullet$} at  6 -6

\plot 1 -3 1 -6 /
\plot 2 -3 2 -6 /
\plot 3 -3 3 -6 /
\plot 4 -3 4 -6 /
\plot 5 -3 5 -6 /
\plot 6 -3 6 -6 /
\arrow <3 pt> [1,2] from 2 -6 to 2 -4.3

\put{$\bullet$} at  1 2  \put{$\bullet$} at  1 5
\put{$\bullet$} at  2 2  \put{$\bullet$} at  2 5
\put{$\bullet$} at  3 2  \put{$\bullet$} at  3 5
\put{$\bullet$} at  4 2  \put{$\bullet$} at  4 5
\put{$\bullet$} at  5 2  \put{$\bullet$} at  5 5
\put{$\bullet$} at  6 2  \put{$\bullet$} at  6 5

\plot 1 2 1 5 /
\plot 2 2 2 5 /
\plot 3 2 3 5 /
\plot 4 2 5 5 /
\plot 6 2 6 5 /
\plot 4 5 5 2 /

\put{$\bullet$} at  1 6  \put{$\bullet$} at  1 9
\put{$\bullet$} at  2 6  \put{$\bullet$} at  2 9
\put{$\bullet$} at  3 6  \put{$\bullet$} at  3 9
\put{$\bullet$} at  4 6  \put{$\bullet$} at  4 9
\put{$\bullet$} at  5 6  \put{$\bullet$} at  5 9
\put{$\bullet$} at  6 6  \put{$\bullet$} at  6 9

\plot 1 9  1 6 /
\plot 2 9  2 6 /
\plot 3 9  3 6 /
\plot 4 9  4 6 /
\plot 5 9  5 6 /
\plot 6 9  6 6 /
\arrow <3 pt> [1,2] from 1 6 to 1 7.7
\arrow <3 pt> [1,2] from 2 6 to 2 7.7
\arrow <3 pt> [1,2] from 6 6 to 6 7.7

\setquadratic
\plot 2 1 3.5 0 5 1 /
\plot 2 -2 3.5 -1 5 -2  /

\plot -5 3 -4 2 -3 3 /
\plot -5 0 -3.5 1 -2 0  /
\endpicture}$
\end{center}
\end{example}

\vskip 3mm

Now we are ready to state and prove our main theorem.

\begin{theorem} \label{th:mixed Schur-Weyl-Sergeev duality}\hfill
{\rm

(a) The walled Brauer superalgebra $\BB$ is an associative
superalgebra for all $r,s \ge 0$.

(b) When $n \geq r+s$, the walled Brauer superalgebra $\BB$ is
isomorphic to
  the supercentralizer algebra $\cent^{\op}$ as an associative superalgebra.
}
\begin{proof}
(a) When $n \ge r+s$, since $\cent^{\op}$ is an associative
superalgebra, Theorem \ref{th:Psi is a bijection} and Theorem
\ref{prop:Psi is an algebra homomorphism} show that $\BB$ is an
associative superalgebra. Note that the associativity condition
$(d_1d_2)d_3=d_1(d_2d_3)$ does not depend on $n$ for all
$(r,s)$-superdiagrams $d_1,d_2,d_3 \in \BB$. Therefore, if $\BB$ is
associative for large enough $n$, then $\BB$ is associative for all
$r,s \ge 0$.

Now the assertion (b) is obvious.
 \end{proof}
\end{theorem}

%\vskip 1cm

\section{Presentation of walled Brauer superalgebras}

Another main result of this paper is a presentation of $\BB$. Assume
that $1\leq i \leq r-1, \ r+1 \leq j \leq r+s-1$, $1 \leq k \leq r$,
$r+1 \leq l \leq r+s$. Let us consider the following diagrams in
$\BB$\,:

$
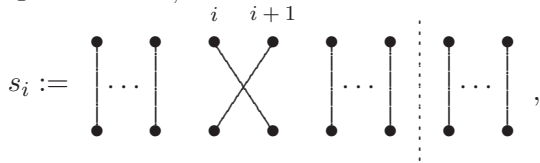
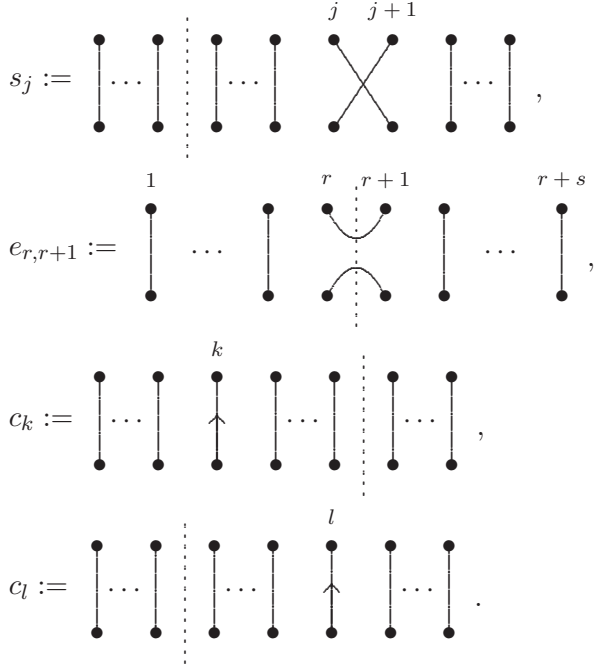
{\beginpicture
\setcoordinatesystem units <0.78cm,0.39cm>
\setplotarea x from 0 to 9, y from -1 to 4
\put{$s_i:= $} at 0 1.5
\put{$\bullet$} at  1 0  \put{$\bullet$} at  1 3
\put{$\bullet$} at  2 0  \put{$\bullet$} at  2 3
\put{$\bullet$} at  3 0  \put{$\bullet$} at  3 3
\put{$\bullet$} at  4 0  \put{$\bullet$} at  4 3
\put{$\bullet$} at  5 0  \put{$\bullet$} at  5 3
\put{$\bullet$} at  6 0  \put{$\bullet$} at  6 3
\put{$\bullet$} at  7 0  \put{$\bullet$} at  7 3
\put{$\bullet$} at  8 0  \put{$\bullet$} at  8 3

\put{$\cdots$} at 1.5 1.5
\put{$\cdots$} at 5.5 1.5
\put{$\cdots$} at 7.5 1.5
\put{{\scriptsize$i$}} at 3 4
\put{{\scriptsize$i+1$}} at 4 4
\put{,} at 8.5 1
\plot 1 3 1 0 /
\plot 2 3 2 0 /
\plot 3 3 4 0 /
\plot 4 3 3 0 /
\plot 5 3 5 0 /
\plot 6 3 6 0 /
\plot 7 3 7 0 /
\plot 8 3 8 0 /
\setdashes  <.4mm,1mm>
\plot 6.5 -1   6.5 4 /
\setsolid
\endpicture}$

${\beginpicture \setcoordinatesystem units <0.78cm,0.39cm>
\setplotarea x from 0 to 9, y from -1 to 4 \put{$s_j:= $} at 0 1.5
\put{$\bullet$} at  1 0 \put{$\bullet$} at  1 3 \put{$\bullet$} at 2
 0  \put{$\bullet$} at  2 3 \put{$\bullet$} at  3 0 \put{$\bullet$}
at 3 3 \put{$\bullet$} at  4 0  \put{$\bullet$} at 4 3
\put{$\bullet$} at  5 0  \put{$\bullet$} at  5 3 \put{$\bullet$} at
6 0 \put{$\bullet$} at  6 3 \put{$\bullet$} at  7 0 \put{$\bullet$}
at 7 3 \put{$\bullet$} at  8 0  \put{$\bullet$} at 8 3

\put{$\cdots$} at 1.5 1.5
\put{$\cdots$} at 3.5 1.5
\put{$\cdots$} at 7.5 1.5
\put{,} at 8.5 1
\put{{\scriptsize$j$}} at 5 4
\put{{\scriptsize$j+1$}} at 6 4

\plot 1 3 1 0 /
\plot 2 3 2 0 /
\plot 3 3 3 0 /
\plot 4 3 4 0 /
\plot 5 3 6 0 /
\plot 6 3 5 0 /
\plot 7 3 7 0 /
\plot 8 3 8 0 /
\setdashes  <.4mm,1mm>
\plot 2.5 -1   2.5 4 /
\setsolid
\endpicture}$ \\

${\beginpicture
\setcoordinatesystem units <0.78cm,0.39cm>
%\setplotarea x from 0 to 6, y from -2 to 3
\put{$e_{r,r+1} := $} at -0.5 1.5
\put{$\bullet$} at  1 0  \put{$\bullet$} at  1 3
\put{$\bullet$} at  3 0  \put{$\bullet$} at  3 3
\put{$\bullet$} at  4 0  \put{$\bullet$} at  4 3
\put{$\bullet$} at  5 0  \put{$\bullet$} at  5 3
\put{$\bullet$} at  6 0  \put{$\bullet$} at  6 3
\put{$\bullet$} at  8 0  \put{$\bullet$} at  8 3

\put{$\cdots$} at 2 1.5
\put{$\cdots$} at 7 1.5
\put{,} at 8.5 1
\put{{\scriptsize $1$}} at 1 4
\put{{\scriptsize $r$}} at 4 4
\put{{\scriptsize $r+1$}} at 5 4
\put{{\scriptsize $r+s$}} at 8 4
\plot 1 3 1 0 /
\plot 3 3 3 0 /
\plot 8 3 8 0 /
\plot 6 3 6 0 /
%\arrow <3 pt> [1,2] from 2 0 to  2 1.7
%\arrow <3 pt> [1,2] from 1.45 0.75 to  1.4 0.75
%\arrow <3 pt> [1,2] from 4 0 to 3.5 1.5
\setdashes  <.4mm,1mm>
\plot 4.5 -1   4.5 4 /
\setsolid
\setquadratic
\plot 4 3 4.5 2 5 3 /
\plot 4 0 4.5 1 5 0  /
\endpicture}$ \\

$ {\beginpicture
\setcoordinatesystem units <0.78cm,0.39cm>
\setplotarea x from 0 to 8, y from -1 to 4
\put{$c_k: = $} at 0 1.5
\put{$\bullet$} at  1 0  \put{$\bullet$} at  1 3
\put{$\bullet$} at  2 0  \put{$\bullet$} at  2 3
\put{$\bullet$} at  3 0  \put{$\bullet$} at  3 3
\put{$\bullet$} at  4 0  \put{$\bullet$} at  4 3
\put{$\bullet$} at  5 0  \put{$\bullet$} at  5 3
\put{$\bullet$} at  6 0  \put{$\bullet$} at  6 3
\put{$\bullet$} at  7 0  \put{$\bullet$} at  7 3

\put{$\cdots$} at 1.5 1.5
\put{$\cdots$} at 4.5 1.5
\put{$\cdots$} at 6.5 1.5
\put{{\scriptsize $k$}} at 3 4
\put{,} at 7.5 1
\plot 1 3 1 0 /
\plot 2 3 2 0  /
\plot 3 3 3 0 /
\plot 4 3 4 0 /
\plot 5 3 5 0 /
\plot 6 3 6 0 /
\plot 7 3 7 0 /
\arrow <3 pt> [1,2] from 3 0 to  3 1.7
\setdashes  <.4mm,1mm>
\plot 5.5 -1   5.5 4 /
\endpicture}$ \\

${\beginpicture \setcoordinatesystem units <0.78cm,0.39cm>
\setplotarea x from 0 to 7.5, y from -1 to 4 \put{$c_l: = $} at 0
1.5 \put{$\bullet$} at  1 0  \put{$\bullet$} at  1 3 \put{$\bullet$}
at  2 0  \put{$\bullet$} at  2 3 \put{$\bullet$} at  3 0
\put{$\bullet$} at  3 3 \put{$\bullet$} at  4 0  \put{$\bullet$} at
4 3 \put{$\bullet$} at  5 0  \put{$\bullet$} at  5 3 \put{$\bullet$}
at  6 0  \put{$\bullet$} at  6 3 \put{$\bullet$} at  7 0
\put{$\bullet$} at  7 3

\put{$\cdots$} at 1.5 1.5
\put{$\cdots$} at 3.5 1.5
\put{$\cdots$} at 6.5 1.5
\put{.} at 7.5 1
\put{{\scriptsize $l$}} at 5 4

\plot 1 3 1 0 /
\plot 2 3 2 0  /
\plot 3 3 3 0 /
\plot 4 3 4 0 /
\plot 5 3 5 0 /
\plot 6 3 6 0 /
\plot 7 3 7 0 /
\arrow <3 pt> [1,2] from 5 0 to  5 1.7
\setdashes  <.4mm,1mm>
\plot 2.5 -1   2.5 4 /
\endpicture}$

\nc{\er}{e_{r,r+1}}
\nc{\epq}{e_{p,q}}
\nc{\eppq}{e_{p',q}}
\nc{\epqp}{e_{p,q'}}
\nc\Srs{\Sigma_r \times \Sigma_s}

\begin{theorem}  \label{th:presentation of B_r,s}
  {\rm The walled Brauer superalgebra $\BB$ is the associative superalgebra generated by (even generators)
  $1, s_1, \ldots, s_{r-1},$ $s_{r+1}, \ldots,$ $s_{r+s-1}$,
  $e_{r,r+1}$ and (odd generators) $c_1, \ldots, c_{r+s}$
  with the following defining relations (for admissible $i,j$)\,:
  \begin{align}
    &s_i^2=1, \ \  s_is_{i+1}s_i=s_{i+1}s_is_{i+1}, \ \ s_is_j=s_js_i \ \  (|i-j|>1), \label{rel:1}\\
 & \er^2=0, \ \ \er s_j =s_j \er \ \ (j \neq r-1,r+1),\label{rel:2}\\
 & \er=\er s_{r-1} \er =\er s_{r+1} \er, \label{rel:3}\\
& s_{r-1} s_{r+1} \er s_{r+1} s_{r-1} \er  = \er s_{r-1} s_{r+1} \er s_{r+1} s_{r-1},  \label{rel:4}\\
  &c_i^2=-1 \ \ (1 \leq i \leq r), \ \ c_i^2=1 \ \ (r+1 \leq i \leq r+s), \ \ c_i c_j=-c_j c_i \ \ (i \neq j),  \label{rel:6}\\
  & s_i c_i s_i=c_{i+1}, \ \  s_i c_j=c_j s_i \ \ (j \neq i, i+1), \label{rel:7}\\
   & c_r \er = c_{r+1} \er, \ \ \er c_r =\er c_{r+1},  \label{rel:8}\\
   &\er c_r \er=0, \ \ \er c_j=c_j \er \ \ (j \neq r, r+1). \label{rel:9}
  \end{align}}
\end{theorem}

\begin{proof}
We observe that the subalgebra of $\BB$ generated by $s_i$'s ($i=1,
\ldots, r-1, r+1, \ldots,  r+s-1$) is isomorphic to the group
algebra of $\Sigma_r \times \Sigma_s$. Also we have $\sigma \er
\sigma^{-1}=e_{\sigma(r), \sigma(r+1)} \ \ \text{for} \ \ \sigma \in
\Sigma_r \times \Sigma_s.$
  Since every element $d \in \BB$ can be decomposed into the form \eqref{eq:basis element form of B_r,s},
the elements $s_1, \ldots, s_{r-1},s_{r+1}, \ldots, s_{r+s-1}$,
$\er$
   and $c_1, \ldots, c_{r+s}$ generate the associative superalgebra $\BB$.
  By direct calculations, one can verify that these elements satisfy the above relations.

  Now take the associative superalgebra $F$ over $\C$ generated by
  $s_i$ ($i=1, \ldots, r-1, r+1, \ldots,  r+s-1$), $\er$
   and $c_j$ ($1 \leq j \leq r+s$) with the above defining relations.
  Then $\BB$ is a quotient of $F$. In particular, $\dim_{\C}(\BB) \leq \dim_{\C} F$.
  Therefore, to prove our theorem, it is enough to show that $\dim_{\C} F \leq \dim_{\C}(\BB)$.

Let $\Sigma$ be the subalgebra of $F$ generated by $s_i$'s  and let
$\ts_i$ be the transposition in $\Srs$ corresponding to $i=1,
\ldots, r-1, r+1, \ldots,  r+s-1$. We define an algebra homomorphism
$\psi: \C(\Srs) \longrightarrow \Sigma$ by $\ts_i \mapsto s_i$. It
is a well-defined surjective homomorphism. So $\Sigma$ is a quotient
of $\C (\Sigma_r \times \Sigma_s)$. Let $e_{p,q}:=\sigma e_{r,r+1}
\sigma^{-1}$, where $\sigma= s_{q-1} \cdots s_{r+1} s_{p} \cdots
s_{r-1}$ for $1 \leq p \leq r-1, \ r+2 \leq q \leq r+s$. Set
$c_A:=c_{a_1} \cdots c_{a_m}$ for a non-empty set $A=\{a_1< \cdots<
a_m\} \subset \{1, \cdots, r+s \}$ and $c_A:=1$ for an empty set
$A$. We  define
   \begin{align*}
   X:=\{ c_P e_{p_1,q_1}   & \cdots  e_{p_a,q_a} c_{Q} \sigma  \;  | \; \text{
(i) $1 \leq p_1 < \cdots < p_a \leq r$,}\\
& \text{(ii)  $r+1 \leq q_i \leq r+s$ and $q_i$ are all distinct,}\\
&\text{(iii) $\sigma \in \Sigma$, and $^{\exists}
\widetilde{\sigma}$ such that $\psi(\widetilde{\sigma})=\sigma$, and
$\widetilde{\sigma}^{-1}(p_1) < \cdots <\widetilde{\sigma}^{-1}(p_a)$,}\\
&\text{(iv) $ P \subset \{p_1, \ldots, p_a \}$, $Q \subset \{1,
\ldots, r+s \} \setminus \{q_1, \ldots, q_a \} $} \}.
\end{align*}

 We will prove our assertion  in two steps:

{\bf Step 1:} Every word in the generators of $F$ belongs to the
linear span of elements of $X$.

{\bf Step 2:} The number of elements in $X$ is less than or equal to
the dimension of $\BB$.

\vskip3mm

\noindent {\it Proof of Step 1}\,:

For $1 \leq i < j \leq r$ or $r+1 \leq i <j \leq r+s$, we define
$$(i \ j) := s_i s_{i+1} \cdots s_{j-2} s_{j-1} s_{j-2} \cdots s_{i+1}
s_i \in \Sigma.$$
 For $1 \leq j < i \leq r$ or $r+1 \leq j < i \leq r+s$, define
 $(i \ j) := (j \ i)^{-1} \in \Sigma$. From the relations
\eqref{rel:1}--\eqref{rel:9}, we obtain the following relations:
\begin{align*}
& (1) \ \ (p'\ q') \epq =\epq (p'\ q'),\\
 & (2) \ \ (p  \ p')\epq = \eppq (p \ p'), \ (q \ q') \epq =\epqp (q \ q'),  \\
 & (3) \ \ \epq e_{p',q'}=e_{p',q'} \epq,\\
 & (4) \ \ \epq \epqp=\epq (q \ q'), \ \epq \eppq=\epq (p  \ p'), \\
 & (5) \ \ \epq e_{p',q'} (p \ p') (q \ q') =\epq e_{p', q'} , \\
 &  (6) \ \  c_q \epq =c_p e_{p,q},\  \epq c_q = \epq c_p,\\
 & (7) \ \ c_{p'} e_{p,q}=\epq c_{p'},\\
 & (8) \ \ \epq^2=0, \ \  \epq c_p \epq=0.
  \end{align*}
Here, all  $p,q,p',q'$ are admissible and distinct.

Suppose that the relations (1)--(8) are satisfied. For a given word
$w$ in the generators of $F$, using the relations \eqref{rel:7}, we
can interchange the place of $c_i$ with the place of $s_j$. From the
relation \eqref{rel:1} and the definition of $e_{p,q}$, we may write
$w$ in terms of the elements $e_{p,q}$. From the relations (1) and
(2), we may move $e_{p,q}$ in the left-hand side of all $s_i$'s. We
would like to move $c_i$ in the left-hand side of all $e_{p,q}$'s or
in the right-hand side of all $e_{p,q}$'s.

From the relation (7), we move each $c_j$ in the left-most side of $w$,
except for $c_p$'s which are in the right-hand side of all $e_{p,q}$'s.
By the relations \eqref{rel:6} and (6), we have
$$w= \pm c_{X} e_{a_1,b_1} \cdots e_{a_{\ell},b_{\ell}} e_{i_1, i'_1} c_{i_1}
 e_{a'_1, b'_1} \cdots e_{a'_x, b'_x} e_{i_2,i'_2}c_{i_2} \cdots $$
 for some set $X$.
 Using the relations (1)--(4) and (8), we rearrange
 $e_{a_1,b_1} \cdots e_{a_{\ell},b_{\ell}} e_{i_1, i'_1}$
to obtain $e_{h_1, g_1} \cdots e_{h_m,g_m} s_{x_1} \cdots s_{x_u}$
or $0$, where $h_j,g_j$ are all distinct. From the relation
\eqref{rel:7}, we obtain $s_{x_1} \cdots s_{x_m}c_{i_1}=c_{k_1}
s_{x_1} \cdots s_{x_m}$ for some $k_1$. If all $h_{j}$ is not $k_1$,
then we can shift $c_{k_1}$ to the left-hand side of $e_{h_1, g_1}
\cdots e_{h_m,g_m}$. If $h_{j}=k_1$ for some $j$, $e_{h_p,g_p}$ ($p
\neq j$) can be shifted to the right-hand side of $c_{k_1}$. In this
way, we can obtain $w=0$ or
$$w=\pm c_{A} e_{k_1, k'_1} c_{k_1} e_{k_2, k'_2}c_{k_2} \cdots
e_{k_t, k'_t} c_{k_t} e_{j_1, j'_1} \cdots e_{j_u,j'_u} c_B \sigma
,$$ for some set $A,B$, $\sigma \in \Sigma$, and $j_i,j'_i$ are all
distinct. If there is $e_{k_t,k'_t}$ in $e_{j_1, j'_1} \cdots
e_{j_u,j'_u}$, then $w=0$ by (3),(8). If all $j'_i$ is not $k'_t$,
then we can shift $c_{k_t}$ to the right-hand side of all
$e_{p,q}$'s using the relation (6),(7). If there is $p$ such that
$j'_p=k'_t$ and $j_p \neq k_t$, then we have
$$w=\pm \cdots e_{k_{t-1},k'_{t-1}}c_{k_{t-1}}c_{j_p} e_{k_t, k'_t} e_{j_1, j'_1} \cdots e_{j_u,j'_u} c_B \sigma. $$
If $j_p=k_{t-1}$, $c_{j_p}$ becomes a constant by relation
\eqref{rel:6}. If not, we can shift $c_{j_p}$ to the left-hand side
of $e_{k_{t-1},k'_{t-1}}$. In this way, we shift $c_{j_p}$ to the
left-side of all $e_{p,q}$'s or $c_{j_p}$ becomes a constant.
Repeating this process on $c_{k_1}, \cdots, c_{k_{t-1}}$, we can
shift all $c_i$ to the left-hand side or right-hand side of all
$e_{p,q}$'s, or obtain $w=0$.

 Combining the relations \eqref{rel:6}, (3),(4),(6)--(8), we
finally obtain $w=0$ or
 $$w= \pm c_P e_{p_1, q_1} \cdots e_{p_a, q_a} c_{Q}
\sigma,$$
where $P \subset \{ p_1, \cdots, p_a \}$, $Q
\subset \{ 1, \cdots, r+s \} \setminus  \{q_1, \ldots, q_a \}$ and
$\sigma \in \Sigma$.

Since $\psi$ is surjective, there exists an element
$\widetilde{\sigma} \in \Srs$ such that
$\psi(\widetilde{\sigma})=\sigma$. If $\tsigma^{-1}(p_1) >
\tsigma^{-1}(p_2)$, we replace $\sigma$ with $(p_1 \ p_2) (q_1 \
q_2) \sigma$ using the  relation (5). Clearly, $\psi((p_1 \ p_2)'
(q_1 \ q_2)' \tsigma)=(p_1 \ p_2) (q_1 \ q_2) \sigma$, where $(x \
y)'=\ts_x \cdots \ts_{y-2} \ts_{y-1} \ts_{y-2} \cdots \ts_x \in
\Srs$. In this way, we can find $\sigma \in \Sigma$ satisfying the
condition (ii). Therefore, every word in generators of $F$ belongs
in the linear span of the elements of $X$.

The relations (1)--(4) were verified in the proof of \cite[Theorem
4.10]{E}. The relation \eqref{rel:4} implies
$e_{r-1,r+2}e_{r,r+1}=e_{r,r+1}e_{r-1,r+2}$.

For the relation (5), we first prove $e_{r,r+1} e_{r-1,r+2}
s_{r-1}s_{r+1}=e_{r,r+1}e_{r-1,r+2}$. Note that
\begin{align*}
  e_{r,r+1}e_{r-1,r+2}s_{r-1}s_{r+1}&= \er s_{r-1}s_{r+1}\er \ \ \ &\text{by \eqref{rel:1}} \\
  &=(\er s_{r-1} \er) s_{r-1} s_{r+1} \er \ \ \ &\text{by \eqref{rel:3}} \\
  &=\er (s_{r+1} s_{r+1})s_{r-1} \er s_{r-1} s_{r+1} \er \ \ \ &\text{by \eqref{rel:1}} \\
 & =\er s_{r+1} \er s_{r+1}s_{r-1} \er s_{r-1} s_{r+1} \ \ \ &\text{by \eqref{rel:4}} \\
 & = \er e_{r-1,r+2}  \ \ \ &\text{by \eqref{rel:3}}.
\end{align*}
For $(p,q) \neq (r,r+1) ,(r-1, r+2)$,  we have
\begin{align*}
 & \epq e_{r,r+1} (p \ r)(r+1 \ q)\\
 & =(r+2 \ q)(r-1 \ p) e_{r-1, r+2} (r+2 \ q )(r-1 \ p) e_{r, r+1} (p \ r) (r+1 \ q)\\
 & =(r+2 \ q)(r-1 \ p) e_{r-1, r+2} e_{r, r+1} (r+2 \ q )(r-1 \ p)  (p \ r) (r+1 \ q)\\
 & = (r+2 \ q)(r-1 \ p) e_{r-1, r+2} e_{r, r+1} (r-1 \ r)(r+1 \ r+2) (r-1 \ p) (r+2 \ q)\\
 & = (r+2 \ q)(r-1 \ p) e_{r-1, r+2} e_{r, r+1}  (r-1 \ p) (r+2 \ q) \\
 & =e_{p, q} e_{r, r+1}.
\end{align*}
Here, if $p=r-1$ or $q=r+2$, we define $(r-1 \  p)=(r+2 \ q)=1$. The general case can be verified in a similar manner.

From the relation \eqref{rel:8}, we obtain
\begin{align*}
c_q \epq &= c_q (r+1 \  q) (r \ p) e_{r, r+1} (r+1 \  q) (r \ p) = (r+1 \ q)( r \ p) c_{r+1} e_{r,r+1}  (r+1 \  q) (r \ p)\\
         &=  (r+1 \ q)( r \ p) c_{r} e_{r,r+1}  (r+1 \  q) (r \ p) =c_p (r+1 \ q)( r \ p) e_{r,r+1}  (r+1 \  q) (r \ p)\\
         &= c_p e_{p,q}.
\end{align*}
The other relations in (6), (7) can be checked similarly.

The relations $\epq^2=0$ and $\epq c_p \epq =0$ follow from
\eqref{rel:1}, \eqref{rel:2} and  \eqref{rel:1}, \eqref{rel:7}, \eqref{rel:9}, respectively.

\vskip3mm

\noindent {\it Proof of Step 2}\,:

One can compute the dimension of $\BB$ by counting the
$(r,s)$-superdiagrams with $i$ horizontal edges in each row and then
summing up over all $i$. Thus we obtain

$$\dim_{\C}(\BB)=\sum_{i=0}^{\min(r,s)} 2^{r+s} \left(\binom{r}{i}
\binom{s}{i}
  \; i! \right)^2 (r-i)!(s-i)!. $$

Now we count the number of the elements in $X$. We define a set
\begin{align*}
Y= \{ ( a, &(p_1,q_1) , \ldots, (p_a,q_a), P, Q, \sigma )\ | \ \text{(i) } 0 \leq a \leq \min(r,s),\\
&\text{(ii) }1 \leq p_1 < \cdots < p_a \leq r, \text{  and  } r+1 \leq q_i \leq r+s  \text{  are all distinct},\\
&\text{(iii) }  P \subset \{p_1, \ldots, p_a \},\\
&\text{(iv) }  Q \subset \{1, \ldots, r+s \} \setminus \{q_1, \ldots, q_a \},\\
&\text{(v) } \sigma \in \Sigma_r \times \Sigma_s, \text{  and  } \sigma^{-1}(p_1) < \cdots <\sigma^{-1}(p_a)  \}.
\end{align*}
We know $|X| \leq |Y|$ and we will count the number of elements of
$Y$. For given $0 \leq a \leq \min(r,s)$, the number of all distinct
subsets of $\{p_1, \ldots, p_a \}$ is $2^a$ and the number of all
distinct subsets of $\{1, \ldots, r+s \} \setminus \{q_1, \ldots,
q_a \}$ is $2^{r+s-a}$. Observe that the number of distinct
$\left((p_1,q_1) , \ldots, (p_a,q_a)\right)$'s satisfying the
condition (ii) is $ \dbinom{r}{a}  \dbinom{s}{a} a!$. Since $\sigma
\in \Sigma_r \times \Sigma_s$ and $\sigma^{-1}(p_1) < \cdots
<\sigma^{-1}(p_a)$, the number of $\sigma$'s satisfying these
conditions is $ \dbinom{r}{a} (r-a)!s!$. Therefore
$$|Y|=\sum_{a=0}^{\min(r,s)} 2^{r+s}
         \left( \binom{r}{a}\binom{s}{a}a!\right) \left( \binom{r}{a}(r-a)! s!\right).$$
From the identity $s!=  \dbinom{s}{a} a! (s-a)!$, we get
$|Y|=\dim_{\C}(\BB)$.

  Combining Step 1 and Step 2, we obtain $\dim_{\C} F \leq |X|  \leq
  \dim_{\C}(\BB)$, which completes the proof.
\end{proof}

\begin{remark} \label{rem:presentation of B_r,s} \hfill
  \begin{enumerate}
  \item If $r=0$ or $s=0$, then we omit the generator $e_{r,r+1}$ and the corresponding relations.

  \item Since $c_{i+1}=s_i c_i s_i$ for admissible $i$, the even
generators together with $c_{1}, c_{r+1}$ generate the whole associative superalgebra $\BB$.

   \item All the relations \eqref{rel:1}--\eqref{rel:4}
   involving $s_i, \er$ ($i=1, \ldots, r-1, r+1, \ldots,  r+s-1$)
also appear in the presentation for the walled Brauer algebra; see \cite[Corollary 4.5]{H}
or \cite[Theorem 4.1]{N}. The following relation in \cite[Corollary 4.5]{H}
$$s_{r-1} s_{r+1} \er s_{r+1} s_{r-1} \er=  \er s_{r-1} s_{r+1} \er $$
can be derived from the relations \eqref{rel:1},\eqref{rel:3} and \eqref{rel:4}.
 \end{enumerate}
\end{remark}

\end{document}